\documentclass[final,10pt]{elsarticle}
\usepackage{lineno,hyperref}
\usepackage{epsfig,amssymb,amsmath,version,amsthm}
\usepackage{amssymb,version,graphicx,fancybox,mathrsfs,multirow}
\usepackage[notcite,notref]{showkeys}
\usepackage{url,hyperref}
\usepackage{subfigure,epstopdf,bm}
\usepackage{color, xcolor}
\usepackage{appendix}
\usepackage{algorithm}
\usepackage{algorithmicx}
\usepackage{algpseudocode}
\usepackage{amsmath}
\usepackage{amsfonts}
\usepackage{amssymb}
\usepackage{graphicx}%

\journal{Computer Physics Communications}

 \catcode`\@=11 \theoremstyle{plain}
 \@addtoreset{equation}{section}   % Makes \section reset 'equation' counter.
 \renewcommand{\theequation}{\arabic{section}.\arabic{equation}}
 \@addtoreset{figure}{section}
 \renewcommand\thefigure{\thesection.\@arabic\c@figure}
 \@addtoreset{table}{section}
 \renewcommand\thetable{\thesection.\@arabic\c@table}
 \newtheorem{thm}{\bf Theorem}
 
 \newtheorem{proposition}{Proposition}[section]
 \newenvironment{theorem}{\begin{thm}} {\end{thm}}
 \newtheorem{cor}{\bf Corollary}

 \newtheorem{lmm}{\bf Lemma}
 
 \newenvironment{lemma}{\begin{lmm}}{\end{lmm}}
 \theoremstyle{remark}
 \newtheorem{rem}{Remark}[section]

 \def \ri {{\rm i}}
 \def \widebar {\accentset{{\cc@style\underline{\mskip10mu}}}}

\newcommand{\bs}[1]{\boldsymbol{#1}}

\begin{document}
	
\begin{frontmatter}
\title{Fast multipole method for 3-D Laplace equation in layered media}

\author[fn1,fn2]{Bo Wang}
\author[fn2]{Wen Zhong Zhang}
\author[fn2]{Wei Cai\corref{cor1}}
\ead{cai@smu.edu}

\cortext[cor1]{Corresponding author}

\address[fn1]{LCSM(MOE), School of Mathematics and Statistics, Hunan Normal University, Changsha, Hunan, 410081, P. R. China.}
\address[fn2]{Department of Mathematics, Southern Methodist University, Dallas, TX 75275, USA.}

\begin{abstract}
In this paper, a fast multipole method (FMM) is proposed for 3-D Laplace equation in layered media. The potential due to charges embedded in layered media is decomposed into a free space component and four types of reaction field components, and the
latter can be associated with the potential of a polarization source defined for each type.
New multipole expansions (MEs) and local expansions (LEs), as well as the multipole to local (M2L) translation operators are derived for the reaction components, based on which the FMMs for reaction components are then proposed. The resulting FMM for charge interactions in layered media is a combination of using the classic FMM for the free space components and the new FMMs for the reaction field components. With the help of a recurrence formula for the run-time computation of the Sommerfeld-type integrals used in M2L translation operators, pre-computations of a large number of tables are avoided. The new FMMs for the reaction components are found to be much faster than the classic FMM for the free space components due to the separation of equivalent polarization charges and the associated target charges by a material interface. As a result, the FMM for potential in layered media costs almost the same as the classic FMM in the free space case. Numerical results validate the fast convergence of the MEs for the reaction components, and the $O(N)$ complexity of the FMM {\color{black}with a given truncation number $p$} for charge interactions in 3-D layered media.
\end{abstract}
\end{frontmatter}

%\linenumbers

\section{Introduction}

Solving the Laplace equation in layered media is connected to many
important applications in science and engineering. For instance, finding the
electric charge distribution over conductors embedded in a layered dielectric
medium has important application in semi-conductor industry, especially in
calculating the capacitance of interconnects (ICs) in very large-scale
integrated (VLSI) circuits for microchip designs (cf. \cite{yu2014advanced,
seidl1988capcal-a,ruehli1973efficient,oh1994capacitance}). Due to complex geometric structure of
the ICs, the charge potential solution to the Laplace equation is usually
solved by an integral method with the Green's function of the layered
media (cf. \cite{oh1994capacitance, zhao1998efficient}), which results
in a huge dense linear algebraic system to be solved by an iterative method such
as GMRES (cf. \cite{campbell1996gmres}), etc. Other applications of the Laplace equation can be found in
medical imaging of brains (cf. \cite{xu2005}), elasticity of composite materials (cf. \cite{babu1994}),
and electrical impedance tomography for geophysical applications (cf. \cite{borcea2002}).

Due to the full matrix resulted from the discretization of integral equations, it will incur an
$O(N^{2})$ computational cost for computing the product of the matrix with a
vector (a basic operation for the GMRES iterative solver). The fast
multipole method (FMM) for the free space Green's function (the Coulomb
potential) has been used in the development of FastCap
(cf. \cite{nabors1991fastcap}) to accelerate this product to $O(N)$. However, the
original FMM of Greengard and Rokhlin (cf. \cite{greengard1987fast,
grengard1988rapid}) is only designed for the free space Green's function. To treat the dielectric material interfaces
in the IC design, unknowns representing the polarization charges from the
dielectric inhomogeneities have to be introduced over the infinite material
interfaces, thus creating unnecessary unknowns and contributing to larger linear systems. These extra unknowns over material interfaces can be avoided by using the Green's function of the layered media in the formulation of the integral equations. To find fast algorithm to solve the discretized linear system, image charges are used to approximate the Green's function of the layered media
\cite{chow1991closed,aksun1996robust, alparslan2010closed}, converting
the reaction potential to the free space Coulomb potential from the charges and their images, thus, the free space FMM can be used
\cite{jandhyala1995multipole,gurel1996electromagnetic, geng2001fast}. Apparently,
this approach is limited to the ability of finding image charge approximation
for the layered media Green's function. Unfortunately, finding such an image approximation can be challenging if not impossible when many layers are present in the problem.

In this paper, we will first derive the multipole expansions (MEs) and local
expansions (LEs) for the reaction components of the layered media Green's function of the Laplace equation. Then, the original FMM for the interactions of charges in free space can be extended to those of charges embedded in layered media. The approach closely follows our recent work for the Helmholtz equation in layered media (cf. \cite{wang2019fast, zhang2018exponential}), where the generating function of the Bessel function (2-D case) or a Funk-Hecke formula (3-D case) were used to connect Bessel functions and plane wave functions. The reason of using Fourier (2-D case) and spherical harmonic (3-D case) expansions of
plane waves is that the Green's function of layered media
has a Sommerfeld-type integral representation involving the plane waves. Even though, the Laplace equation could be considered as a zero limit of the wave number $k$ in the Helmholtz equation,  some special treatments of the  $k\rightarrow0$ limit is required to derive a limit version of the extended Funk-Hecke formula, which is the key in the derivation of MEs, LEs and M2L for the reaction components of the Laplacian Green's function in layered media. Similar to our previous work for the Helmholtz equation in layered media, the potential due to sources embedded in layered media is decomposed into free space and reaction components and equivalent polarization charges are introduced to re-express the reaction components. The FMM in layered media will then consist of classic FMM for the free space components and FMMs for reaction components, using equivalent polarization sources and the new MEs, LEs and M2L translations. Moreover, in order to avoid making pre-computed tables (cf. \cite{wang2019fast}), we introduce a recurrence formula for efficient computation of the Sommerfeld-type integrals used in M2L translation operators. As in the Helmholtz equation case, the FMMs for the reaction field components are much faster than that for the free space components due to the fact that the introduced equivalent polarization charges are always separated from the associated target charges by a material interface. As a result, the new FMM for charges in layered media costs almost the same as the classic FMM for the free space case.

The rest of the paper is organized as follows. In section 2, we will consider the limit case of the extended Funk-Hecke formula introduced in \cite{wang2019fast}, which leads to an spherical harmonic expansion of the exponential kernel in the Sommerfeld-type integral representation of the Green's function. By using this expansion, we present alternative derivation, via the Fourier spectral domain, for the ME, LE and M2L operators of the free space Green's function. The same approach will be then used to derive MEs, LEs and
M2L translation operators for the reaction components of the layered Green's function. In Section 3, after a short discussion
on the Green's function in layered media consisting of free
space and reaction components, we present the formulas for the potential induced by sources embedded in layered media. Then, the concept of equivalent
polarization charge of a source charge is introduced for each type of the reaction
components. The reaction components of the layered Green's function and the potential are then re-expressed by using the equivalent polarization charges. Further, we derive the MEs, LEs and M2L translation operators for the reaction
components based on the new expressions using equivalent polarization charges. Combining the original source charges and the equivalent polarization charges associated to each reaction component, the FMMs for reaction components can be implemented.
%To avoid using pre-computed tables, a recurrence formula is introduced for the computation of  the Sommerfeld type integrals used in M2L translation operators.
Section 4 will give numerical results to show the spectral accuracy and $O(N)$ complexity of the proposed FMM for charge interactions in layered
media. Finally, a conclusion is given in Section 5.

\section{A new derivation for the ME, LE, and M2L operator of the Green's function
of 3-D Laplace equation in free space}

In this section, we first review the multipole and local expansions of the free
space Green's function of the Laplace equation and the corresponding shifting and
translation operators. They are the key formulas in the classic FMM and can be derived by using the addition theorems for Legendre polynomials. Then, we present a
new derivation for them by using the Sommerfeld-type integral representation
of the Green's function. The key expansion formula used in the new derivation is a limiting case of the extended Funk-Hecke formula introduced in \cite{wang2019fast}. This new technique shall be applied to derive MEs and LEs for the reaction components of the layered media Green's function later on.

\subsection{The multipole and local expansions of free space Green's function}
Let us review some addition theorems (cf. \cite{grengard1988rapid, epton1995multipole}), which have been used for the derivation of the ME, LE and corresponding shifting and translation operators of the free space Green's function.
In this paper, we adopt the definition
\begin{equation}\label{sphericalharmonics}
Y_n^m(\theta,\varphi)=(-1)^m\sqrt{\frac{2n+1}{4\pi}\frac{(n-m)!}{(n+m)!}}P_n^{m}(\cos\theta)e^{\ri m\varphi}:=\widehat P_n^{m}(\cos\theta)e^{\ri m\varphi}
\end{equation}
for the spherical harmonics where $P_n^m(x)$ (resp. $\widehat P_n^m(x)$) is the associated (resp. normalized) Legendre function of degree $n$ and order $m$. Recall that
\begin{equation}
P_n^m(x)=(-1)^m(1-x^2)^{\frac{m}{2}}\frac{d^m}{dx^m}P_n(x)
\end{equation}
for integer order $0\leq m\leq n$ and
\begin{equation}
P_n^{-m}=(-1)^m\frac{(n-m)!}{(n+m)!}P_n^m(x), \quad {\rm so}\quad \widehat P_n^{-m}(x)=(-1)^m\widehat P_n^m(x)
\end{equation}
for $0<m\leq n$, where $P_n(x)$ is the Legendre polynomial of degree $n$. The so-defined spherical harmonics constitute a complete orthogonal basis of $L(\mathbb S^2)$ (where $\mathbb S^2$ is the unit spherical surface) and
$$\langle Y_n^m, Y_{n'}^{m'}\rangle=\delta_{nn'}\delta_{mm'},\quad Y_n^{-m}(\theta,\varphi)=(-1)^m\overline{Y_n^m(\theta,\varphi)}.$$
It is worthy to point out that the spherical harmonics with different scaling constant defined as
\begin{equation}
\widetilde Y_n^m(\theta,\varphi)=\sqrt{\frac{(n-|m|)!}{(n+|m|)!}}P_n^{|m|}(\cos\theta)e^{\ri m\varphi}=\ri^{m+|m|}\sqrt{\frac{4\pi}{2n+1}}Y_n^m(\theta,\varphi),\;\;
\end{equation}
have been frequently adopted in published FMM papers (e.g., \cite{greengard1997new,grengard1988rapid}).
By using the spherical harmonics defined in \eqref{sphericalharmonics}, we will re-present the addition theorems derived in \cite{grengard1988rapid, epton1995multipole}. For this purpose, we define constants
\begin{equation}
c_n=\sqrt{\frac{2n+1}{4\pi}},\quad A_n^m=\frac{(-1)^nc_n}{\sqrt{(n-m)!(n+m)!}},\quad |m|\leq  n.
\end{equation}
\begin{theorem}\label{addthmleg}
	{\bf (Addition theorem for Legendre polynomials)} Let $P$ and $Q$ be points with spherical coordinates $(r,\theta,\varphi)$ and $(\rho,\alpha,\beta)$, respectively, and let $\gamma$ be the angle subtended between them. Then
	\begin{equation}
	P_n(\cos\gamma)=\frac{4\pi}{2n+1}\sum\limits_{m=-n}^n\overline{Y_n^{m}(\alpha,\beta)}Y_n^m(\theta,\varphi).
	\end{equation}
\end{theorem}

\begin{theorem}\label{theorem:firstaddition}
	Let $Q=(\rho,\alpha,\beta)$ be the center of expansion of an arbitrary spherical harmonic of negative degree. Let the point $P=(r,\theta,\varphi)$, with $r>\rho$, and $P-Q=(r', \theta', \varphi')$. Then
	\begin{equation*}
	\frac{Y_{n'}^{m'}(\theta', \varphi')}{r'^{n'+1}}=\sum\limits_{n=0}^{\infty}\sum\limits_{m=-n}^n\frac{(-1)^{|m+m'|-|m'|}A_n^mA_{n'}^{m'}\rho^nY_n^{-m}(\alpha,\beta)}{c_n^2A_{n+n'}^{m+m'}}\frac{Y_{n+n'}^{m+m'}(\theta,\varphi)}{r^{n+n'+1}}.
	\end{equation*}
\end{theorem}
\begin{theorem}\label{theorem:secondaddition}
	Let $Q=(\rho,\alpha,\beta)$ be the center of expansion of an arbitrary spherical harmonic of negative degree. Let the point $P=(r,\theta,\varphi)$, with $r<\rho$, and $P-Q=(r', \theta', \varphi')$. Then
	\begin{equation*}
	\frac{Y_{n'}^{m'}(\theta', \varphi')}{r'^{n'+1}}=\sum\limits_{n=0}^{\infty}\sum\limits_{m=-n}^n\frac{(-1)^{n'+|m|}A_n^mA_{n'}^{m'}\cdot Y_{n+n'}^{m'-m}(\alpha,\beta)}{c_n^2A_{n+n'}^{m'-m}\rho^{n+n'+1}}r^nY_n^{m}(\theta,\varphi).
	\end{equation*}
\end{theorem}
\begin{theorem}\label{theorem:fourthaddition}
	Let $Q=(\rho,\alpha,\beta)$ be the center of expansion of an arbitrary spherical harmonic of negative degree. Let the point $P=(r,\theta,\varphi)$ and $P-Q=(r', \theta', \varphi')$. Then
	\begin{equation*}
	r'^{n'}Y_{n'}^{m'}(\theta', \varphi')=\sum\limits_{n=0}^{n'}\sum\limits_{m=-n}^n\frac{(-1)^{n-|m|+|m'|-|m'-m|}c_{n'}^2 A_n^mA_{n'-n}^{m'-m}\cdot \rho^nY_n^{m}(\alpha,\beta) }{c_{n}^2c_{n'-n}^2A_{n'}^{m'}r^{n-n'}}Y_{n'-n}^{m'-m}(\theta,\varphi),
	\end{equation*}
	where $A_n^m=0$, $Y_n^m(\theta,\varphi)\equiv 0$ for $|m|>n$ is used.
\end{theorem}
%\begin{theorem}\label{theorem:thirdaddition}
%	Let $Q=(\rho,\alpha,\beta)$ be the center of arbitrary spherical harmonic of nonnegative degree. Let the point $P=(r,\theta,\varphi)$, with  $P-Q=(r', \theta', \varphi')$. Then
%	\begin{equation}
%	\frac{Y_{n'}^{m'}(\theta', \varphi')}{r'^{n'+1}}=\sum\limits_{n=0}^{n'}\sum\limits_{m=-n}^n\frac{(-1)^n\ri^{|m'-m|-|m'|-|m|}c_{n'}^2A_n^mA_{n'-n}^{m'-m}\rho^nY_n^{m}(\alpha,\beta)}{c_n^2c_{n'-n}^2A_{n'}^{m'}}Y_{n'-n}^{m'-m}(\theta,\varphi)r^{n'-n}.
%	\end{equation}
%\end{theorem}

Denote by $(r,\theta, \varphi)$ and $(r',\theta',\varphi')$ the spherical coordinates of given points $\bs r, \bs r'\in\mathbb R^3$. The law of cosines gives
\begin{equation}\label{cosineslaw}
|\bs r-\bs r'|^2=r^2+(r')^2-2r r'\cos\gamma,
\end{equation}
where
\begin{equation}
\cos\gamma=\cos\theta\cos\theta'+\sin\theta\sin\theta'\cos(\varphi-\varphi').
\end{equation}
Then, the Green's function of the Laplace equation in free space is given by
\begin{equation}
G(\bs r, \bs r')=\frac{1}{|\bs r-\bs r'|}=\frac{1}{r\sqrt{1-2\mu\cos\gamma+\mu^2}}=\frac{1}{r'\sqrt{1-2\frac{\cos\gamma}{\mu}+\frac{1}{\mu^2}}},
\end{equation}
where $\mu=r'/r$ and the scaling constant $1/4\pi$ has been omitted through out this paper. Furthermore, we have the following Taylor expansions
\begin{equation}\label{legendreexp}
\frac{1}{r\sqrt{1-2\mu\cos\gamma+\mu^2}}=\sum\limits_{n=0}^{\infty}P_n(\cos\gamma)\frac{\mu^n}{r}=\sum\limits_{n=0}^{\infty}P_n(\cos\gamma)\frac{r'^n}{r^{n+1}}, \;\; \mu=\frac{r'}{r}<1,
\end{equation}
and
\begin{equation}\label{legendreexp2}
\frac{1}{r'\sqrt{1-2\frac{\cos\gamma}{\mu}+\frac{1}{\mu^2}}}=\sum\limits_{n=0}^{\infty}P_n(\cos\gamma)\frac{1}{r'\mu^n}=\sum\limits_{n=0}^{\infty}P_n(\cos\gamma)\frac{r^{n}}{r'^{n+1}},\;\; \mu=\frac{r'}{r}>1.
\end{equation}
Straightforwardly, we have error estimates
\begin{equation}\label{meerror}
\left|\frac{1}{|\bs r-\bs r'|}-\sum\limits_{n=0}^{p}\frac{P_n(\cos\gamma_j)(r')^n}{r^{n+1}}\right|\leq \frac{1}{r-r'}\Big(\frac{r'}{r}\Big)^{p+1}, \quad r>r',
\end{equation}
and
\begin{equation}\label{leerror}
\left|\frac{1}{|\bs r-\bs r'|}-\sum\limits_{n=0}^{\infty}P_n(\cos\gamma_j)\frac{r^{n}}{(r')^{n+1}}\right|\leq\frac{1}{r'-r}\Big(\frac{r}{r'}\Big)^{p+1}, \quad r>r',
\end{equation}
by using the fact $|P_n(x)|\leq 1$ for all $x\in[-1, 1]$.
\begin{figure}[ht!]\label{3dspherical}
	\centering
	\includegraphics[scale=0.9]{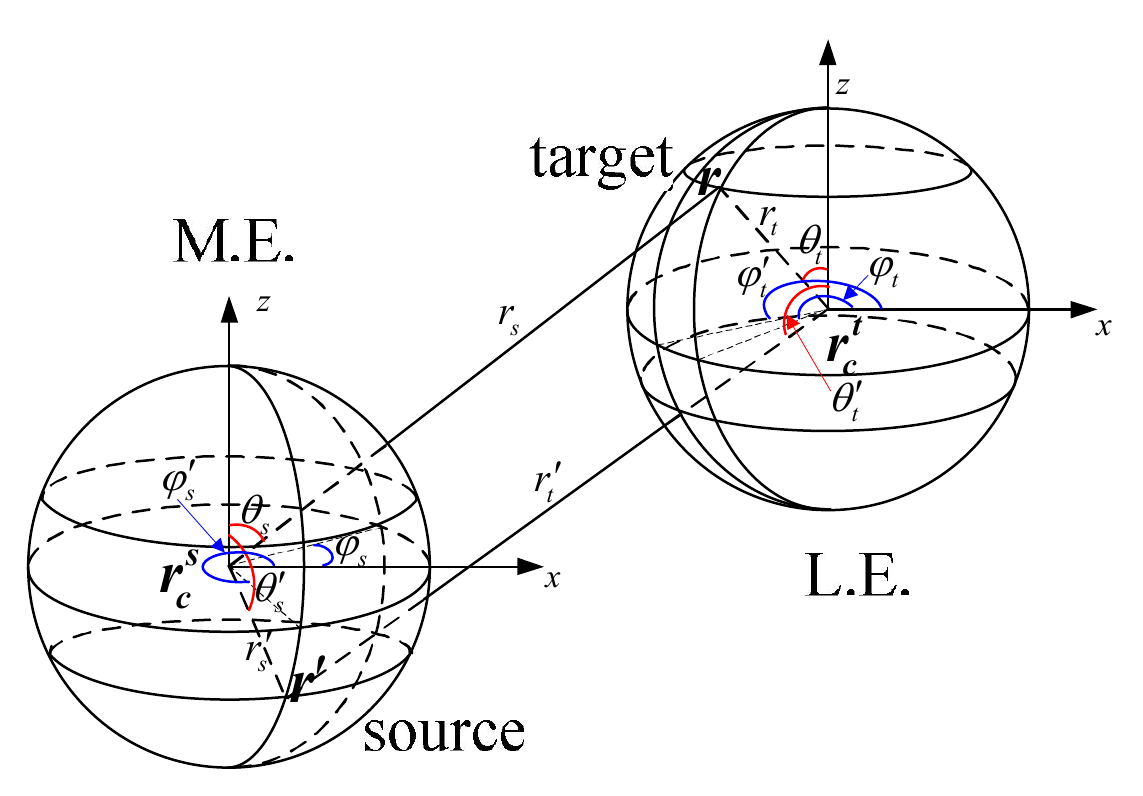}
	\caption{Spherical coordinates used in multipole and local expansions.}
\end{figure}

Based on the discussion above, we are ready to present ME, LE and corresponding shifting and translation operators of the free space Green's function. Let $\bs r_c^s$ and $\bs r_c^t$ be source and target centers close to source $\bs r'$ and target $\bs r$, i.e,  $|\bs r'-\bs r_c^s|<|\bs r-\bs r_c^s|$ and $|\bs r'-\bs r_c^t|>|\bs r-\bs r_c^t|$. Following the derivation in \eqref{cosineslaw}-\eqref{legendreexp2} we have Taylor expansions
\begin{equation}\label{expansionbeforeme}
\frac{1}{|\bs r-\bs r'|}=\frac{1}{|(\bs r-\bs r_c^s)-(\bs r'-\bs r_c^s)|}=\sum\limits_{n=0}^{\infty}\frac{P_n(\cos\gamma_s)}{ r_s}\Big(\frac{r'_s}{ r_s}\Big)^n,
\end{equation}
and
\begin{equation}\label{localexpansionbeforeme}
\frac{1}{|\bs r-\bs r'|}=\frac{1}{|(\bs r-\bs r_c^t)-(\bs r'-\bs r_c^t)|}=\sum\limits_{n=0}^{\infty}\frac{P_n(\cos\gamma_t)}{r_t'}\Big(\frac{r_t}{r_t'}\Big)^n,
\end{equation}
where  $(r_s, \theta_s,\varphi_s)$, $(r_t,\theta_t,\varphi_t)$ are the spherical coordinates of $\bs r-\bs r_c^s$ and $\bs r-\bs r_c^t$, $(r'_s, \theta'_s,\varphi'_s)$, $(r'_t,\theta'_t,\varphi'_t)$ are the spherical coordinates of $\bs r'-\bs r_c^s$ and $\bs r'-\bs r_c^t$( see Fig. \ref{3dspherical}) and
\begin{equation}
\begin{split}
&\cos\gamma_s=\cos\theta_s\cos\theta'_s+\sin\theta_s\sin\theta'_s\cos(\varphi_s-\varphi'_s),\\
&\cos\gamma_t=\cos\theta_t\cos\theta_t'+\sin\theta_t\sin\varphi_t'\cos(\varphi_t-\varphi_t').
\end{split}
\end{equation}
Note that $P_n(\cos\gamma_s)$, $P_n(\cos\gamma_t)$ still mix the source and target information ($\bs r$ and $\bs r'$). Applying Legendre addition theorem \ref{addthmleg} to expansions \eqref{expansionbeforeme} and \eqref{localexpansionbeforeme} gives a ME
\begin{equation}\label{mefreespace}
\frac{1}{|\bs r-\bs r'|}=\sum\limits_{n=0}^{\infty}\sum\limits_{m=-n}^nM_{nm}r_s^{-n-1}Y_n^m(\theta_s,\varphi_s),
\end{equation}
and a LE
\begin{equation}\label{lefreespace}
\frac{1}{|\bs r-\bs r'|}=\sum\limits_{n=0}^{\infty}\sum\limits_{m=-n}^nL_{nm}r^n_tY_n^m(\theta_t,\varphi_t),
\end{equation}
where
\begin{equation}\label{freespaceexpcoef}
M_{nm}=c_n^{-2}r'^n_s\overline{Y_n^{m}(\theta'_s,\varphi'_s)},\quad L_{nm}=c_n^{-2}r'^{-n-1}_t\overline{Y_n^{m}(\theta'_t,\varphi'_t)}.
\end{equation}

The FMM also need shifting and translation operators between expansions. Applying the addition Theorem \ref{theorem:secondaddition} to expansion functions in ME \eqref{mefreespace} provides a translation from ME \eqref{mefreespace} to LE \eqref{lefreespace} as follows
\begin{equation}\label{metole}
L_{nm}=\sum\limits_{n'=0}^{\infty}\sum\limits_{m'=-n'}^{n'}\frac{(-1)^{n'+|m|}A_{n'}^{m'}A_n^mY_{n+n'}^{m'-m}(\theta_{st}, \varphi_{st})}{c_{n'}^2A_{n+n'}^{m'-m}r_{st}^{n+n'+1}}M_{n'm'},
\end{equation}
where $(r_{st}, \theta_{st}, \varphi_{st})$ is the spherical coordinate of $\bs r_c^s-\bs r_c^t$.
Similarly, the following center shifting operators for ME and LE,
\begin{eqnarray}
\displaystyle\tilde{M}_{nm}
=\sum\limits_{n'=0}^{n}\sum\limits_{m'=-n'}^{n'}\frac{(-1)^{|m|-|m-m'|}A_{n'}^{m'}A_{n-n'}^{m-m'}r_{ss}^{n'}Y_{n'}^{-m'}(\theta_{ss}, \varphi_{ss})}{c_{n'}^2A_{n}^{m}}M_{n-n',m-m'},\label{metome}\\
\displaystyle\tilde L_{nm}=\sum\limits_{n'=n}^{\infty}\sum\limits_{m'=-n'}^{n'} \frac{(-1)^{n'-n-|m'-m|+|m'|-|m|}c_{n'}^2A_{n'-n}^{m'-m}A_n^mr_{tt}^{n'-n}Y_{n'-n}^{m'-m}(\theta_{tt}, \varphi_{tt})}{c_{n'-n}^2c_n^2A_{n'}^{m'}}L_{n'm'},\label{letole}
\end{eqnarray}
can be derived by using addition Theorem \ref{theorem:firstaddition} and \ref{theorem:fourthaddition}. Here, $(r_{ss}, \theta_{ss}, \varphi_{ss})$ and $(r_{tt},\theta_{tt}, \varphi_{tt})$ are the spherical coordinates of $\bs r_c^s-\tilde{\bs r}_c^s$ and $\bs r_c^t-\tilde{\bs r}_c^t$,
\begin{equation}
\tilde{M}_{nm}=c_n^{-2}\tilde r'^n_s\overline{Y_{n}^{m}(\tilde\theta'_s,\tilde\varphi'_s)},\quad \tilde L_{nm}=c_n^{-2}\tilde r'^{-n-1}_t\overline{Y_{n}^{m}(\tilde\theta'_t,\tilde\varphi'_t)},
\end{equation}
are the ME and LE coefficients with respect to new centers $\tilde{\bs r}_c^s$ and $\tilde{\bs r}_c^{t}$, respectively.

A very important fact in the expansions \eqref{mefreespace}-\eqref{lefreespace} is that the source and target coordinates are separated. It is one of the key features for the compression in the FMM (cf. \cite{greengard1987fast,greengard1997new}). Besides using the addition theorems, this target/source separation can also be achieved in the Fourier spectral domain. We shall give a new derivation for \eqref{mefreespace} and \eqref{lefreespace} by using the integral representation of $1/|\bs r-\bs r'|$.  More importantly, this methodology can be further applied to derive multipole and local expansions for the reaction components of the Green's function in layered media to be discussed in section 3.

\subsection{A new derivation of the multipole and local expansions}
For the Green's function $G(\bs r, \bs r')$, we have the well known identity
\begin{equation}\label{freegreenintegral}
\frac{1}{|\bs r-\bs r'|}
=\frac{1}{2\pi}\int_{0}^{\infty}\int_{0}^{2\pi}e^{\ri k_{\rho}((x-x')\cos\alpha+(y-y')\sin\alpha)-k_{\rho}|z-z'|}d\alpha dk_{\rho}.
\end{equation}
By this identity, we straightforwardly have source/target separation in spectral domain as follows
\begin{equation}
\label{positivecase}
\begin{split}
\frac{1}{|\bs r-\bs r'|}=\frac{1}{2\pi}\int_{0}^{\infty}\int_{0}^{2\pi}e^{\ri k_{\rho}\bs k_0\cdot (\bs r-\bs r_c^s)}e^{-\ri k_{\rho}\bs k_0\cdot(\bs r'-\bs r_c^s)}d\alpha dk_{\rho},\\
\frac{1}{|\bs r-\bs r'|}=\frac{1}{2\pi}\int_{0}^{\infty}\int_{0}^{2\pi}e^{\ri k_{\rho}\bs k_0\cdot (\bs r-\bs r_c^t)}e^{-\ri k_{\rho}\bs k_0\cdot(\bs r'-\bs r_c^t)}d\alpha dk_{\rho},
\end{split}
\end{equation}
for $z\geq z'$ where
\begin{equation}
\bs k_0=(\cos\alpha, \sin\alpha, \ri),
\label{k0}
\end{equation}
and without loss of generality, here we only consider the case $z\geq z'$ as an example.

A FMM for the Helmholtz kernel in layered media has been proposed in \cite{wang2019fast} based on a similar source/target separation in the spectral domain. One of the key ingredients is the following extension of the well-known Funk-Hecke formula (cf. \cite{watson,martin2006multiple}).
\begin{proposition}\label{prop:Funk-Hecke}
	Given $\bs r=(x, y, z)\in \mathbb R^3$, $k>0$, $\alpha\in[0, 2\pi)$ and denoted by $(r,\theta,\varphi)$ the spherical coordinates of $\bs r$, $\bs k=(\sqrt{k^2-k_z^2}\cos\alpha, \sqrt{k^2-k_z^2}\sin\alpha, k_z)$ is a vector of complex entries. Choosing branch \eqref{branch} for $\sqrt{k^2-k_z^2}$ in $e^{\ri \bs k\cdot{\bs r}}$ and $\widehat P_n^m(\frac{k_z}{k})$, then
	\begin{equation}\label{extfunkhecke}
	e^{\ri \bs k\cdot{\bs r}}=\sum\limits_{n=0}^{\infty}\sum\limits_{m=-n}^n A_{n}^m(\bs r)\ri^n\widehat{P}_n^m\Big(\frac{k_z}{k}\Big)e^{-\ri m\alpha}=\sum\limits_{n=0}^{\infty}\sum\limits_{m=-n}^n \overline{A_{n}^m(\bs r)}\ri^n\widehat{P}_n^m\Big(\frac{k_z}{k}\Big)e^{\ri m\alpha},
	\end{equation}
	holds for all $k_z\in\mathbb C$, where
	$$A_{n}^m(\bs r)=4\pi j_n(kr)Y_n^m(\theta,\varphi).$$
\end{proposition}
This extension enlarges the range of the classic Funk-Hecke formula from $k_z\in (-k, k)$ to the whole complex plane by choosing the branch
\begin{equation}\label{branch}
\sqrt{k^2-k_z^2}=-\ri \sqrt{r_1r_2}e^{\ri\frac{\theta_1+\theta_2}{2}},
\end{equation}
for the square root function $\sqrt{k^2-k_z^2}$.
Here $(r_i,\theta_i), i=1, 2$ are the modules and principal values of the arguments of complex numbers $k_z+k$ and $k_z-k$, i.e.,
$$k_z+k=r_1e^{\ri\theta_1}, \quad -\pi<\theta_1\leq\pi,\quad  k_z-k=r_2e^{\ri\theta_2},\quad -\pi<\theta_2\leq\pi.$$
It is worthy to point out that the normalized associated Legendre function $\widehat P_n^m(x)$ has also been extended to the whole complex plain by using the same branch.

{\color{black} Although we have $k_{\rho}\bs k_0=\lim\limits_{k\rightarrow 0^+}(\sqrt{k^2-k_z^2}\cos\alpha, \sqrt{k^2-k_z^2}\sin\alpha, k_z)$, with $k_z=\ri k_{\rho}$, taking limit directly in the expansion \eqref{extfunkhecke} will induce singularity in the associated Legendre function. In the following, we will show how to cancel the singularity to obtain a limit version of \eqref{extfunkhecke}, which gives an expansion for $e^{\ri k_{\rho}\bs k_0\cdot\bs r}$. }For this purpose, we first need to recall the corresponding extended Legendre addition theorem (cf. \cite{wang2019fast}).
\begin{lemma}\label{lemma2}
	Let $\bs w=(\sqrt{1-w^2}\cos\alpha, \sqrt{1-w^2}\sin\alpha, w)$ be a vector with complex entries, $\theta, \varphi$ be the azimuthal angle and polar angles of a unit vector $\hat{\bs r}$. Define
	\begin{equation}
	\beta(w)=w\cos\theta+\sqrt{1-w^2}\sin\theta\cos(\alpha-\varphi),
	\end{equation}
	then
	\begin{equation}\label{legendreadd}
	P_n(\beta(w))=\frac{4\pi}{2n+1}\sum\limits_{m=-n}^n\widehat P_n^m(\cos\theta)\widehat P_n^m(w)e^{\ri m(\alpha-\varphi)},
	\end{equation}
	for all $w\in\mathbb C$.
\end{lemma}
From this extended Legendre addition theorem, the following expansion can be obtained by choosing a specific $\bs\omega$ and then taking limit carefully.
\begin{lemma}\label{lemma3}
	Let $\bs k_0=(\cos\alpha, \sin\alpha, \ri)$ be a vector with complex entry, $\theta, \varphi$ be the azimuthal angle and polar angles of a unit vector $\hat{\bs r}$. Then
	\begin{equation}\label{polynomialadd}
	\frac{(\ri\bs k_0\cdot\hat{\bs r})^n}{n!}=\sum\limits_{m=-n}^nC_n^m\widehat P_n^m(\cos\theta)e^{\ri m(\alpha-\varphi)},
	\end{equation}
	where
	\begin{equation}\label{constantcnm}
	C_n^m=\ri^{2n-m}\sqrt{\frac{4\pi}{(2n+1)(n+m)!(n-m)!}}.
	\end{equation}
\end{lemma}
\begin{proof}
	For any $k\in \mathbb{R}^+$, define $\bs{k} = ( \sqrt{k^2+1} \cos\alpha, \sqrt{k^2+1} \sin\alpha, \ri)$.	By lemma \ref{lemma2}, we have
	\begin{align}\label{eq-ext-Legendre}
	\begin{split}
	k^nP_n\Big(\frac{\bs k\cdot\hat{\bs r}}{k}\Big)
	&= \frac{4\pi}{2n+1}\sum\limits_{m=-n}^n\widehat P_n^m(\cos\theta)k^n\widehat P_n^m\Big(\frac{\ri}{k}\Big)e^{\ri m(\alpha-\varphi)}.
	\end{split}
	\end{align}
	Consider the limit of the above identity as $k \to 0^+$. Note that
	\begin{equation}
	\lim_{k \to 0^+} \bs{k}\cdot\hat{\bs{r}} =  \bs{k}_0 \cdot \hat{\bs{r}} ,
	\end{equation}
	together with the knowledge on the coefficient of the leading term in the Legendre polynomial $P_n(x)$ lead to
	\begin{equation}\label{legendrelimit}
	\lim_{k \to 0^+} k^n P_n\Big(\frac{\bs k\cdot\hat{\bs r}}{k}\Big)= \frac{(2n)!}{2^{n}(n!)^2}( \bs{k}_0 \cdot \hat{\bs{r}})^n.
	\end{equation}

	Recall the Rodrigues' formula of the associated Legendre function
	\begin{equation}
	\widehat P_n^m(x)=\frac{c_{nm}}{2^nn!}(1-x^2)^{\frac{m}{2}}\frac{d^{n+m}}{dx^{n+m}}(x^2-1)^n,\quad c_{nm}=\sqrt{\frac{2n+1}{4\pi}\frac{(n-m)!}{(n+m)!}}
	\end{equation}
	for $0 \le m \le n$, we have
	\begin{equation}
	k^n \widehat P_n^m\Bigl(\frac{\ri }{k}\Bigr) = \frac{c_{nm}}{2^n n!}\frac{(2n)!}{(n-m)!}(k^2+1)^{\frac{m}{2}}\cdot k^{n-m}\widetilde{Q}_{n-m}\Bigl(\frac{\ri }{k}\Bigr)
	\end{equation}
	where $\widetilde{Q}_n(z)$ is a \emph{monic} polynomial of degree $n$.
	Hence, we get similarly
	\begin{equation}\label{alegendrelimit}
	\lim_{k \to 0^+} k^n \widehat P_n^m\Bigl(\frac{\ri}{k}\Bigr) = \frac{c_{nm}}{2^n n!}\frac{(2n)!\ri ^{n-m}}{(n-m)!}.
	\end{equation}
	The identity $\widehat{P}_n^{-m}(x)=(-1)^m\widehat{P}_n^m(x)$ will give the limit for $-n\le m<0$ cases. Now, let $k\rightarrow 0^+$ in \eqref{eq-ext-Legendre} and use results \eqref{legendrelimit} and \eqref{alegendrelimit}, we complete the proof.
\end{proof}
\begin{proposition}\label{prop:Funk-Hecke-limit}
	Given $\bs r=(x, y, z)\in \mathbb R^3$, $\alpha\in[0, 2\pi)$ and denoted by $(r,\theta,\varphi)$ the spherical coordinates of $\bs r$, $\bs k_0=(\cos\alpha, \sin\alpha, \ri)$ is a vector of complex entries. Then
	\begin{equation}\label{extfunkheckelim}
	e^{\ri k_{\rho}\bs k_0\cdot{\bs r}}=\sum\limits_{n=0}^{\infty}\sum\limits_{m=-n}^nC_n^m r^nY_n^m(\theta,\varphi)k_{\rho}^ne^{-\ri m\alpha}=\sum\limits_{n=0}^{\infty}\sum\limits_{m=-n}^n C_n^mr^n\overline{Y_n^m(\theta,\varphi)}k_{\rho}^ne^{\ri m\alpha},
	\end{equation}
	holds for all $r>0$, $k_{\rho}> 0$, where $C_n^m$ is the constant defined in \eqref{constantcnm}.
\end{proposition}
\begin{proof}
	By Taylor expansion, we have
	\begin{equation}
	e^{\ri k_{\rho}\bs k_0\cdot{\bs r}}=\sum\limits_{n=0}^{\infty}\frac{(\ri \bs k_0\cdot\hat{\bs r})^n}{n!}k_{\rho}^nr^n.
	\end{equation}
	Then, \eqref{extfunkheckelim} follows by applying lemma \ref{lemma3} to each term in the above expansion.
\end{proof}

\begin{rem}
	By setting $k_z=\ri k_{\rho}$ and using the limit values given by \eqref{legendrelimit} and \eqref{alegendrelimit}, one can also verify that the expansions for $e^{\ri k_{\rho}\bs k_0\cdot{\bs r}}$ in proposition \ref{prop:Funk-Hecke-limit} are exactly the limiting cases of the expansions in proposition \ref{prop:Funk-Hecke}.
\end{rem}
%We give some numerical results to validate the Funk-Hecke formula in complex plane. Here, we set $k=1.8$, $\alpha=\frac{\pi}{4}$. The error is defined as
%\begin{equation}
%err=\frac{\Big|e^{\ri\bs k\cdot\bs r}-\sum\limits_{n=0}^{p}\sum\limits_{m=-n}^n \overline{A_{n}^m(\bs r)}\ri^n\widehat{P}_n^m\Big(\frac{k_z}{k}\Big)e^{\ri m\alpha}\Big|}{|e^{\ri\bs k\cdot\bs r}|}.
%\end{equation}
%The errors and convergence rates against $p$ are depicted in Fig. \ref{FunkHeckeApperr}. Spectral convergence rate against $p$ is observed for $k_z\in\mathbb C$.
%\begin{figure}[ht!]
%	\center
%	\subfigure[$\bs r=(0.3, 0.4, 0.5)$, $p=18$]{\includegraphics[scale=0.21]{funkheckeapp1}}\quad
%	\subfigure[$\bs r=(0.3, 0.4, -0.5)$, $p=18$]{\includegraphics[scale=0.21]{funkheckeapp2}}\quad
%	\subfigure[$\bs r=(0.3, 0.4, 0.5)$]{\includegraphics[scale=0.24]{funkheckeconvergence}}
%	\caption{Errors of the approximation using Funk-Hecke formula.}
%	\label{FunkHeckeApperr}%
%\end{figure}

Applying spherical harmonic expansion \eqref{extfunkheckelim} to exponential functions $e^{-\ri k_{\rho}\bs k_0\cdot(\bs r-\bs r_c^s)}$ and $e^{\ri \bs k\cdot(\bs r-\bs r_c^t)}$ in \eqref{positivecase} gives
\begin{equation}\label{meinspectraldomain}
\frac{1}{|\bs r-\bs r'|}=\sum\limits_{n=0}^{\infty}\sum\limits_{m=-n}^{n}M_{nm}\frac{(-1)^{n}c_n^2C_n^m}{2\pi}\int_0^{\infty}\int_0^{2\pi}k_{\rho}^{n}e^{\ri k_{\rho}\bs k_0\cdot(\bs r-\bs r_c^s)}e^{\ri m\alpha}d\alpha  dk_{\rho},
\end{equation}
and
\begin{equation}\label{leinspectraldomain}
\frac{1}{|\bs r-\bs r'|}=\sum\limits_{n=0}^{\infty}\sum\limits_{m=-n}^{n}\hat L_{nm}r_t^nY_n^m(\theta_t,\varphi_t),
\end{equation}
for $z\geq z'$, where $M_{nm}$ is defined in \eqref{freespaceexpcoef} and
\begin{equation}
\hat L_{nm}=\frac{C_n^m}{2\pi}\int_0^{\infty}\int_0^{2\pi}k_{\rho}^{n}e^{\ri k_{\rho}\bs k_0\cdot(\bs r_c^t-\bs r')}e^{-\ri m\alpha}d\alpha  dk_{\rho}.
\end{equation}
Recall the identity
\begin{equation}\label{wavefunspectralform}
r^{-n-1}Y_n^{-m}(\theta,\varphi)=\frac{(-1)^nc_n^2C_n^m}{2\pi}\int_0^{\infty}\int_0^{2\pi}k_{\rho}^{n}e^{\ri k_{\rho}\bs k_0\cdot\bs r}e^{-\ri m\alpha}d\alpha  dk_{\rho},
\end{equation}
for $z\geq 0$, we see that \eqref{meinspectraldomain} and \eqref{leinspectraldomain} are exactly the ME \eqref{mefreespace} and LE \eqref{lefreespace} in the case of $z\geq z'$.

To derive the translation from the ME \eqref{mefreespace} to the LE \eqref{lefreespace}, we perform further spliting
\begin{equation}
e^{\ri k_{\rho}\bs k_0\cdot(\bs r-\bs r_c^s)}=e^{\ri k_{\rho}\bs k_0\cdot(\bs r-\bs r_c^t)}e^{\ri k_{\rho}\bs k_0\cdot(\bs r_c^t-\bs r_c^s)},
\end{equation}
in \eqref{meinspectraldomain} and apply expansion \eqref{extfunkheckelim} again to obtain the translation
\begin{equation*}
\begin{split}
L_{nm}=&C_n^m\sum\limits_{n'=0}^{\infty}\sum\limits_{m'=-n'}^{n'}M_{n'm'}\frac{(-1)^{n'}c_{n'}^2C_{n'}^{m'}}{2\pi}\int_0^{\infty}\int_0^{2\pi}k_{\rho}^{n+n'}e^{\ri k_{\rho}\bs k_0(\bs r_c^t-\bs r_c^s)}e^{\ri(m'-m)\alpha}d\alpha  dk_{\rho}.
\end{split}
\end{equation*}
By using the identity \eqref{wavefunspectralform}, we can also verify that the above integral form is equal to the entries of the M2L translation matrix defined in \eqref{metole}.

\section{FMM for 3-D Laplace equation in layered media}
%In this section, the concept of equivalent polarization sources is first introduced based on the observation that the reaction components of the layered Green's function determined by the interface dependent local rather than the physical $z$-coordinate of the source particles. Then, the reaction components of the potential produced by source charges in layered media is re-expressed by using equivalent polarized sources. By using the new expression with equivalent polarized source, the multipole and local expansions for the reaction components of layered Green's function shall be derived by following the spirit introduced in the last section. Based on these expansions and translations, FMM for 3-D Laplace kernel in layered media is proposed.

In this section, the potential of charges in layered media is formulated using layered Green's function and then decomposed into a free space and four types of reaction components. Furthermore, the reaction components are re-expressed by using equivalent polarization charges. The new expressions are used to derive the MEs and LEs for the reaction components of the layered Green's function in the same spirit as in the last section. Based on these new expansions and translations, FMM for 3-D Laplace kernel in layered media can be developed.

\subsection{\color{black}Potential due to sources embedded in multi-layer media}
Consider a layered medium consisting of $L$-interfaces located at $z=d_{\ell
},\ell=0,1,\cdots,L-1$, see Fig. \ref{layerstructure}. The piece wise constant material parameter is described by $\{\varepsilon_{\ell}\}_{\ell=0}^L$.
Suppose we have a point
source at $\boldsymbol{r}^{\prime}=(x^{\prime},y^{\prime},z^{\prime})$ in the
$\ell^{\prime}$th layer ($d_{\ell^{\prime}}<z^{\prime}<d_{\ell^{\prime}-1}$), then, the layered media Green's function $u_{\ell\ell'}(\bs r, \bs r')$ for the Laplace equation satisfies
\begin{equation}\label{Laplaceeqlayer}
\boldsymbol{\Delta}u_{\ell\ell'}(\boldsymbol{r},\boldsymbol{r}^{\prime
})=-\delta(\boldsymbol{r},\boldsymbol{r}^{\prime}),
\end{equation}
at field point $\boldsymbol{r}=(x,y,z)$ in the $\ell$th layer ($d_{\ell
}<z<d_{\ell}-1$) where $\delta(\boldsymbol{r},\boldsymbol{r}^{\prime})$ is the
Dirac delta function.
By using Fourier transforms along $x-$ and $y-$directions, the problem can be solved analytically for each layer in $z$ by imposing
transmission conditions at the interface between $\ell$th and $(\ell-1)$th
layer ($z=d_{\ell-1})$, \textit{i.e.},
\begin{equation}\label{transmissioncond}
u_{\ell-1,\ell'}(x,y,z)=u_{\ell\ell'}(x,y,z),\quad \varepsilon_{\ell-1}\frac{\partial  u_{\ell-1,\ell'}(x,y,z)}{\partial z}=\varepsilon_{\ell}\frac{\partial \widehat u_{\ell\ell'}(k_{x},k_{y},z)}{\partial z},
\end{equation}
as well as the decaying conditions in the top and bottom-most layers as
$z\rightarrow\pm\infty$.
\begin{figure}[ht!]\label{layerstructure}
	\centering
	\includegraphics[scale=0.7]{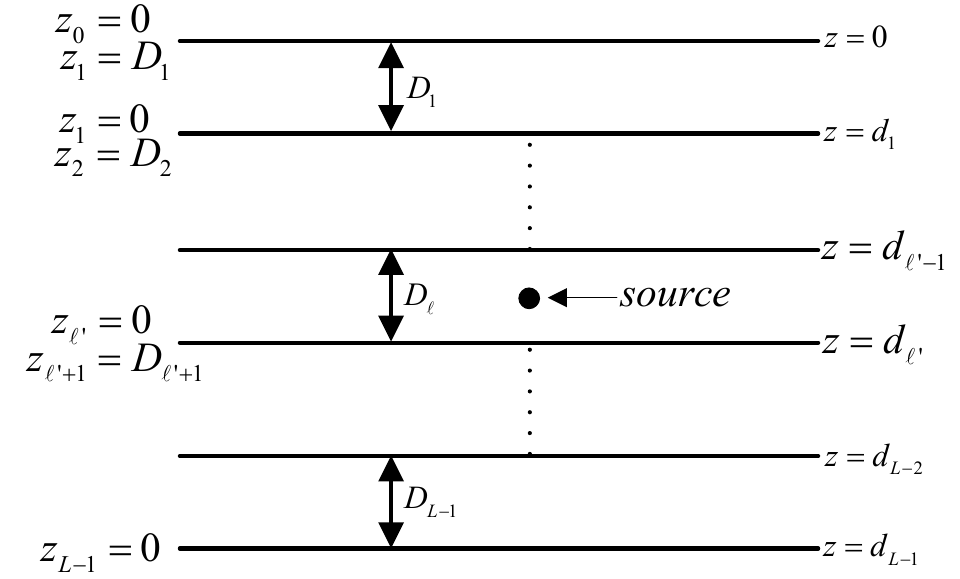}
	\caption{Sketch of the layer structure for general multi-layer media.}
\end{figure}

{\color{black}
Here, we give the expression for the analytic solution with detailed derivations included in the Appendix A. In general, the layered media Green's function in the physical domain takes the form
\begin{equation}\label{layeredGreensfun}
u_{\ell\ell^{\prime}}(\boldsymbol{r},\boldsymbol{r}^{\prime})=\begin{cases}
\displaystyle u_{\ell\ell'}^{\text{r}}(\boldsymbol{r},\boldsymbol{r}^{\prime})+\frac{1}{4\pi
	|\bs r-\bs r^{\prime}|},&\ell=\ell',\\
\displaystyle u_{\ell\ell'}^{\text{r}}(\boldsymbol{r},\boldsymbol{r}^{\prime}), & \text{otherwise},
\end{cases}
\end{equation}
where
\begin{equation}\label{reactioncomponent}
u^{\text{r}}_{\ell\ell'}(\bs r, \bs r')=\begin{cases}
\displaystyle u_{0\ell^{\prime}}^{11}%
(\boldsymbol{r},\boldsymbol{r}^{\prime})+u_{0\ell^{\prime}}^{12}%
(\boldsymbol{r},\boldsymbol{r}^{\prime}),\\
\displaystyle u_{\ell\ell^{\prime}}^{11}%
(\boldsymbol{r},\boldsymbol{r}^{\prime})+u_{\ell\ell^{\prime}}^{12}%
(\boldsymbol{r},\boldsymbol{r}^{\prime})+u_{\ell\ell^{\prime}}^{21
}(\boldsymbol{r},\boldsymbol{r}^{\prime})+u_{\ell\ell^{\prime}}^{22
}(\boldsymbol{r},\boldsymbol{r}^{\prime}), &
0<\ell<L,\\
\displaystyle u_{L\ell^{\prime}}^{21}%
(\boldsymbol{r},\boldsymbol{r}^{\prime})+u_{L\ell^{\prime}}^{22}%
(\boldsymbol{r},\boldsymbol{r}^{\prime}).
\end{cases}
\end{equation}
The reaction component $u_{\ell\ell'}^{\mathfrak{ab}}(\bs r, \bs r')$ is given in an integral form
\begin{equation}\label{generalcomponents}
u_{\ell\ell'}^{\mathfrak{ab}}(\bs r, \bs r')=\frac{1}{8\pi^2 }\int_0^{\infty}\int_0^{2\pi}e^{\ri\bs k_{\alpha}\cdot(\bs\rho-\bs\rho')}\mathcal{Z}_{\ell\ell'}^{\mathfrak{ab}}(z, z')\sigma_{\ell\ell'}^{\mathfrak{ab}}(k_{\rho})d\alpha dk_{\rho},\quad \mathfrak{a, b}=1, 2,
\end{equation}
where,
\begin{equation}
\bs k_{\alpha}=k_{\rho}(\cos\alpha,\sin\alpha),
\label{kalpha}
\end{equation}
and $\{\mathcal{Z}_{\ell\ell'}^{\mathfrak{ab}}(z, z')\}_{\mathfrak{a, b}=1}^2$ are exponential functions defined as
\begin{equation}\label{zexponential}
\begin{split}
&\mathcal{Z}_{\ell\ell'}^{11}(z, z'):=e^{- k_{\rho} (z-d_{\ell}+z'-d_{\ell'})},\quad \mathcal{Z}_{\ell\ell'}^{12}(z, z'):=e^{-k_{\rho} (z-d_{\ell}+d_{\ell'-1}-z')},\\
&\mathcal{Z}_{\ell\ell'}^{21}(z, z'):=e^{-k_{\rho} (d_{\ell-1}-z+z'-d_{\ell'})},\quad \mathcal{Z}_{\ell\ell'}^{22}(z, z'):=e^{-k_{\rho} (d_{\ell-1}-z+d_{\ell'-1}-z')},
\end{split}
\end{equation}
$\{\sigma_{\ell\ell'}^{\mathfrak{ab}}(k_{\rho})\}_{\mathfrak{a,b}=1}^2$ are reaction densities only dependent on the layer structure and the material parameter $k_{\ell}$ in each layer. The reaction densities can be calculated efficiently by using a recursive algorithm, see the Appendix A for more details. It is worthwhile to point out that the reaction components $u_{\ell\ell'}^{\mathfrak a2}$ or $u_{\ell\ell'}^{\mathfrak a1}$ will vanish if the source $\bs r'$ is in the top or bottom most layer.

Withe the expression of the Green's function in layered media, we are ready to consider the potential due to sources in layered media. Let $\mathscr{P}_{\ell}=\{(Q_{\ell j},\boldsymbol{r}_{\ell j}),$ $j=1,2,\cdots
,N_{\ell}\}$, $\ell=0, 1, \cdots, L$ be $L$ groups of source charges distributed in a multi-layer medium with $L+1$ layers (see Fig. \ref{layerstructure}). The group of charges in $\ell$-th layer is denoted by $\mathscr{P}_{\ell}$.  Apparently, the potential at $\bs r_{\ell i}$ due to all other charges is given by  the summation
\begin{equation}\label{potential1}
\hspace{-3pt}\Phi_{\ell}(\boldsymbol{r}_{\ell i})=\sum\limits_{\ell'=0}^{L}\sum\limits_{j=1}^{N_{\ell'}}Q_{\ell' j}u_{\ell\ell'}(\bs r_{\ell i},\bs r_{\ell' j})
=\sum\limits_{j=1,j\neq i}^{N_{\ell}}\frac{Q_{\ell j}}{4\pi|\bs r_{\ell i}-\bs r_{\ell j}|}+\sum\limits_{\ell'=0}^{L}\sum\limits_{j=1}^{N_{\ell'}}Q_{\ell' j}u_{\ell\ell'}^{\text{r}}(\bs r_{\ell i},\bs r_{\ell' j}),
\end{equation}
where $u_{\ell\ell'}^{\text{r}}(\bs r,\bs r')$ are the reaction field components defined in \eqref{reactioncomponent}-\eqref{zexponential}. As the reaction components of the Green's function in multi-layer media have different expressions \eqref{generalcomponents} for sources and targets in different layers, it is necessary to perform calculation individually for interactions between any two groups of charges among the $L+1$ groups $\{\mathscr{P}_{\ell}\}_{\ell=0}^{L}$. Applying expressions \eqref{reactioncomponent} and \eqref{generalcomponents} in \eqref{potential1}, we obtain
\begin{equation}\label{totalinteraction}
\begin{split}
\Phi_{\ell}(\boldsymbol{r}_{\ell i})=&\Phi_{\ell}^{\text{free}}(\boldsymbol{r}_{\ell 	i})+\Phi_{\ell}^{\text{r}}(\boldsymbol{r}_{\ell i})\\
=&\Phi_{\ell}^{\text{free}}(\boldsymbol{r}_{\ell i})+\sum\limits_{\ell^{\prime}=0}^{L-1}[\Phi_{\ell\ell^{\prime}}^{11
}(\boldsymbol{r}_{\ell i})+\Phi_{\ell\ell^{\prime}}^{21
}(\boldsymbol{r}_{\ell i})]+\sum\limits_{\ell^{\prime}=1}^{L}[\Phi_{\ell\ell^{\prime}}^{12
}(\boldsymbol{r}_{\ell i})+\Phi_{\ell\ell^{\prime}}^{22
}(\boldsymbol{r}_{\ell i})],
\end{split}
\end{equation}
where
\begin{equation}\label{freereactioncomponents}
\begin{split}
&  \Phi_{\ell}^{\text{free}}(\boldsymbol{r}_{\ell i}):=\sum\limits_{j=1,j\neq
	i}^{N_{\ell}}\frac{Q_{\ell j}}{4\pi|\boldsymbol{r}_{\ell
		i}-\boldsymbol{r}_{\ell j}|},\quad  \Phi_{\ell\ell^{\prime}}^{\mathfrak{ab}}(\boldsymbol{r}_{\ell i}):=\sum
\limits_{j=1}^{N_{\ell^{\prime}}}Q_{\ell^{\prime}j}u_{\ell\ell^{\prime}%
}^{\mathfrak{ab}}(\boldsymbol{r}_{\ell i},\boldsymbol{r}_{\ell^{\prime}j}%
).
\end{split}
\end{equation}
Obviously, the free space component $\Phi_{\ell}^{\text{free}}(\boldsymbol{r}_{\ell i})$ can be computed using the traditional FMM. Thus, we will only focus on the computation of the reaction components $\{\Phi_{\ell\ell^{\prime}}^{\mathfrak{ab}}(\boldsymbol{r}_{\ell i})\}_{\mathfrak{a, b}=1}^2$.}

\subsection{Equivalent polarization sources for reaction components}
The expressions of the components given in \eqref{freereactioncomponents} show that the free space components only involve interactions between charges in the same layer. Interactions between charges in different layers are all included in the reaction components. Two groups of charges involved in the computation of a reaction component could be physically very far away from each other as there could be many layers between the source and target layers associated to the reaction component, see Fig. \ref{polarizedsource} (left).

Our recent work on the Helmholtz equation \cite{zhang2018exponential,wang2019fast}, of which the Laplace equation can be considered as a special case where the wave number $k=0$, has shown that the exponential convergence of the ME and LE for the reaction components $u_{\ell\ell'}^{\mathfrak{ab}}(\bs r, \bs r')$ in fact depends on the distance between the target charge $\bs r$ and a polarization charge defined for the source charge $\bs r'$, which uses the distance between the source charge $\bs r'$ and the nearest material interface and always locates next to the nearest interface adjacent to the target charge. Fig. \ref{sourceimages} illustrates the location of the polarization charge $\bs r'_{\mathfrak{ab}}$ for each of the four types of reaction fields
$\tilde u_{\ell\ell'}^{\mathfrak{ab}},\mathfrak{a},\mathfrak{b}=1,2 $.
% Such a distance was termed "polarziation distance" between $r$ and $r'$ determined by a interface dependent distance which can be transformed to Euclidean distance between target and artificial polarization source.  Based on this observation, we introduce four types of equivalent polarization sources associated to the four types of reaction components, respectively.
%
%However, the integral representation \eqref{generalcomponents} for a general reaction component shows that the source and target coordinates are only involved in the exponential kernels $e^{\ri\bs k_{\alpha}\cdot(\bs\rho-\bs\rho')}\mathcal{Z}_{\ell\ell'}^{\mathfrak{ab}}(z, z')$ where the kernels $\mathcal Z_{\ell\ell'}^{\mathfrak{ab}}(z, z')$ are not directly determined by the physical $z$-coordinates of $\bs r$, $\bs r'$ but by the local $z$-coordinates $z-d_{\ell}$, $d_{\ell-1}-z$ and $z'-d_{\ell'}$, $d_{\ell'-1}-z'$ with respect to corresponding interfaces.
Specifically, the equivalent polarization sources associated to reaction components $u^{\mathfrak{ab}}_{\ell\ell'}(\bs r, \bs r')$, $\mathfrak{a}, \mathfrak{b} =1, 2$ are set to be  at coordinates (see Fig. \ref{sourceimages})
\begin{equation}\label{eqpolarizedsource}
\begin{split}
&\bs r'_{11}:=(x', y', d_{\ell}-(z'-d_{\ell'})),\quad\quad\bs r'_{12}:=(x', y', d_{\ell}-(d_{\ell'-1}-z')),\\
&\bs r'_{21}:=(x', y', d_{\ell-1}+(z'-d_{\ell'})),\quad \bs r'_{22}:=(x', y', d_{\ell-1}+(d_{\ell'-1}-z')),
\end{split}
\end{equation}
and the reaction potentials are
\begin{equation}\label{reactfieldpolarization}
\tilde u_{\ell\ell'}^{\mathfrak{ab}}(\bs r, \bs r'_{\mathfrak{ab}}):=\frac{1}{8\pi^2 }\int_0^{\infty}\int_0^{2\pi}e^{\ri\bs k_{\alpha}\cdot(\bs\rho-\bs\rho')}e^{-k_{\rho}|z-z^{\prime}_{\mathfrak{ab}}|}\sigma_{\ell\ell'}^{\mathfrak{ab}}(k_{\rho})d\alpha dk_{\rho},\quad \mathfrak{a,b}=1, 2,
\end{equation}
where $z'_{\mathfrak{ab}}$ denotes the $z$-coordinate of $\bs r'_{\mathfrak{ab}}$, i.e.,
\begin{equation*}
z_{11}^{\prime}=d_{\ell}-(z'-d_{\ell'}),\; z_{12}^{\prime}=d_{\ell}-(d_{\ell'-1}-z'),\; z_{21}^{\prime}=d_{\ell-1}+(z'-d_{\ell'}),\;z_{22}^{\prime}=d_{\ell-1}+(d_{\ell'-1}-z').
\end{equation*}
\begin{figure}[ht!]
	\centering
	\includegraphics[scale=0.65]{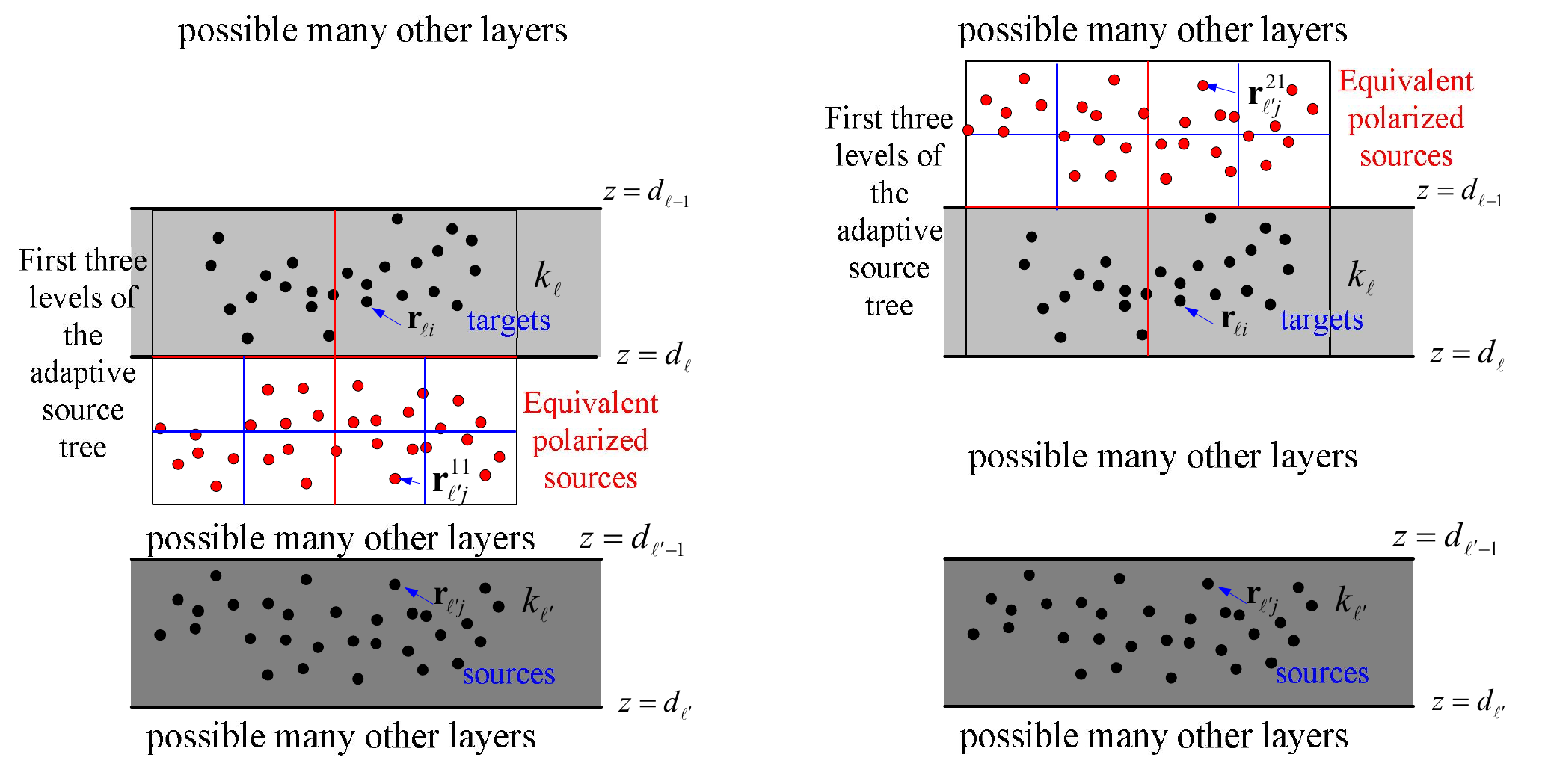}
	\caption{Equivalent polarized sources $\{\bs r_{\ell'j}^{11}\}$, $\{\bs r_{\ell'j}^{21}\}$ and boxes in source tree.}
	\label{polarizedsource}
\end{figure}
\begin{figure}[ht!]
	\center
	\subfigure[$u_{\ell\ell'}^{11}$]{\includegraphics[scale=0.45]{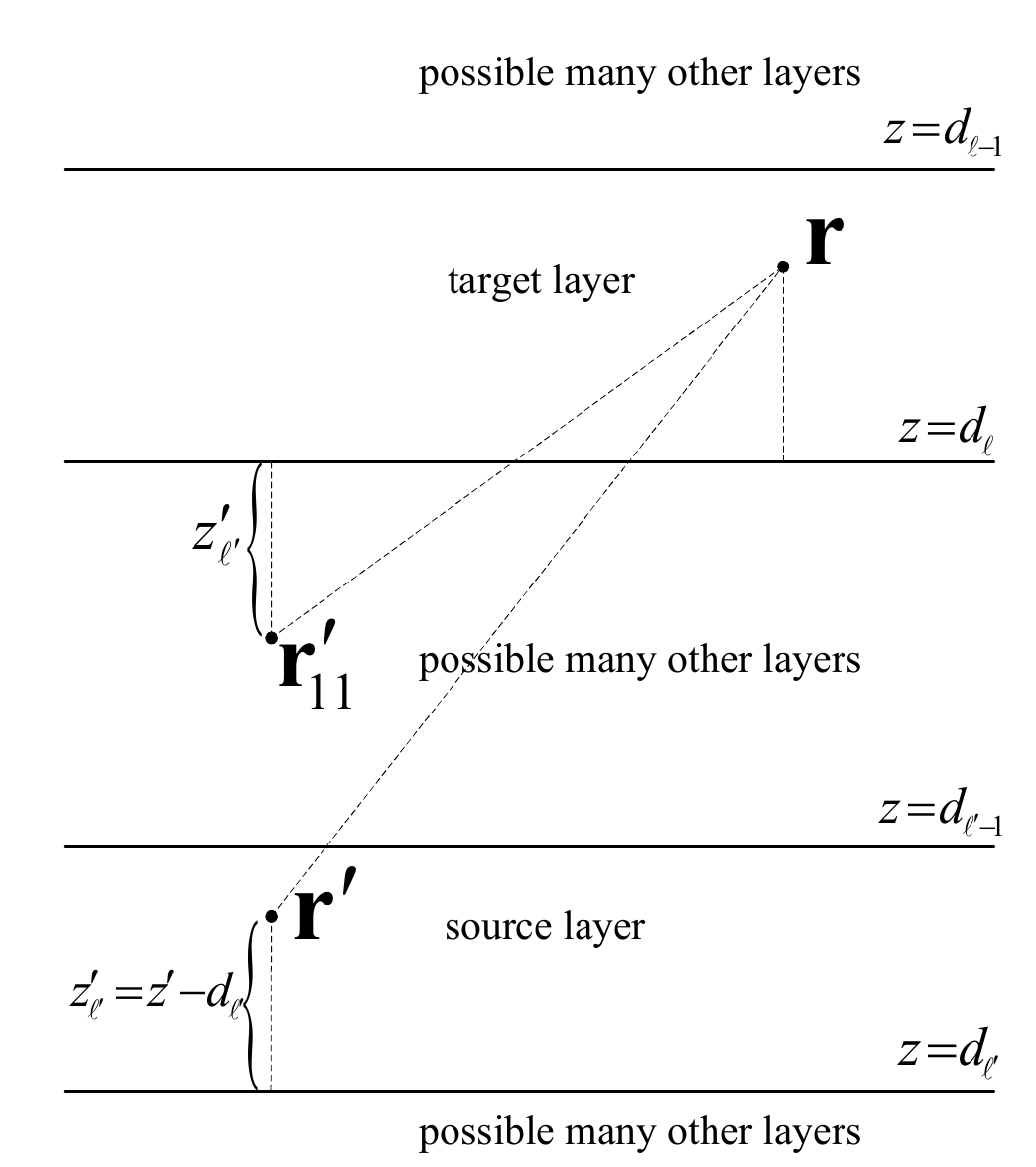}}
	\subfigure[$u_{\ell\ell'}^{12}$]{\includegraphics[scale=0.45]{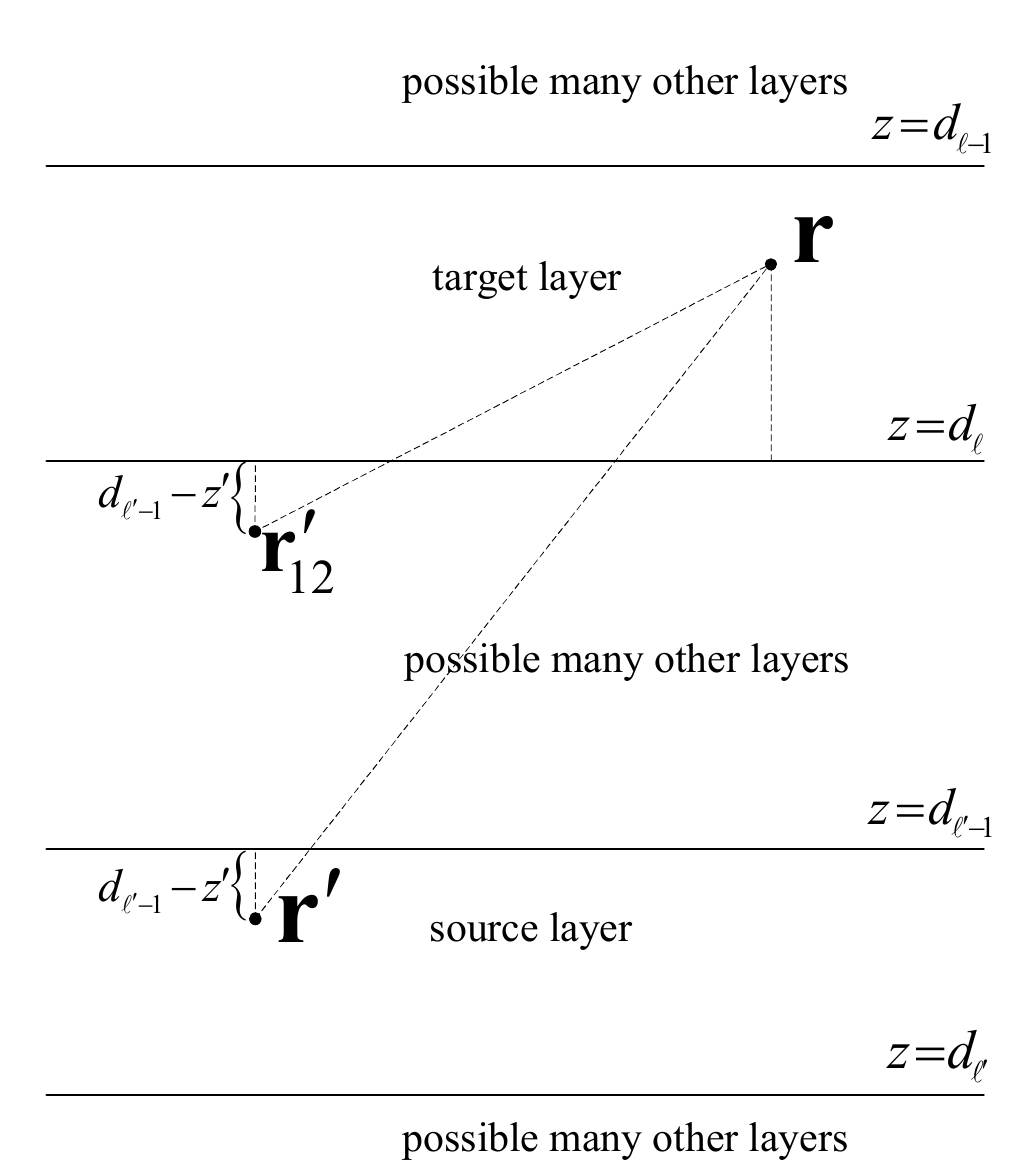}} \\
	\subfigure[$u_{\ell\ell'}^{21}$]{\includegraphics[scale=0.45]{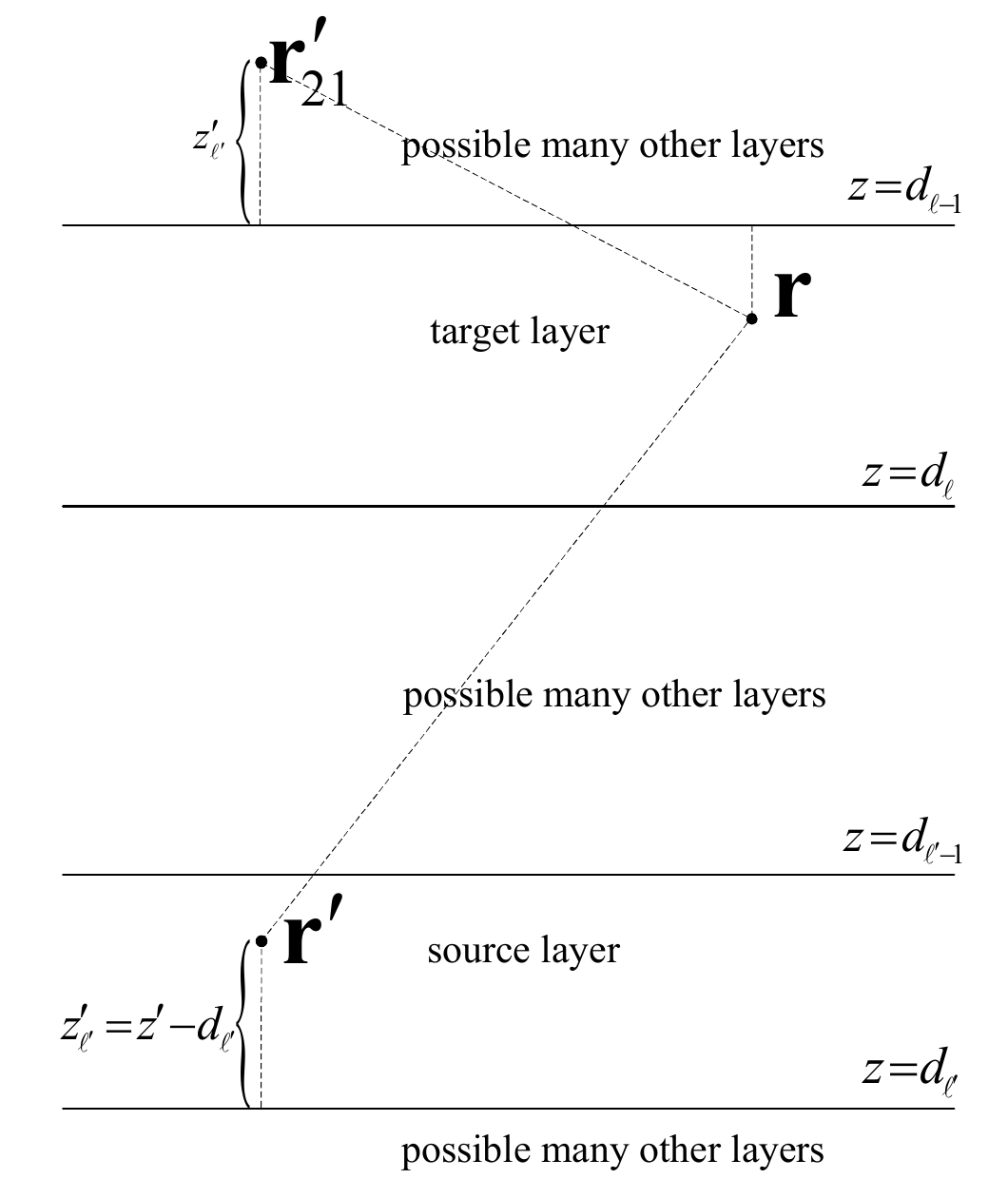}}
	\subfigure[$u_{\ell\ell'}^{22}$]{\includegraphics[scale=0.45]{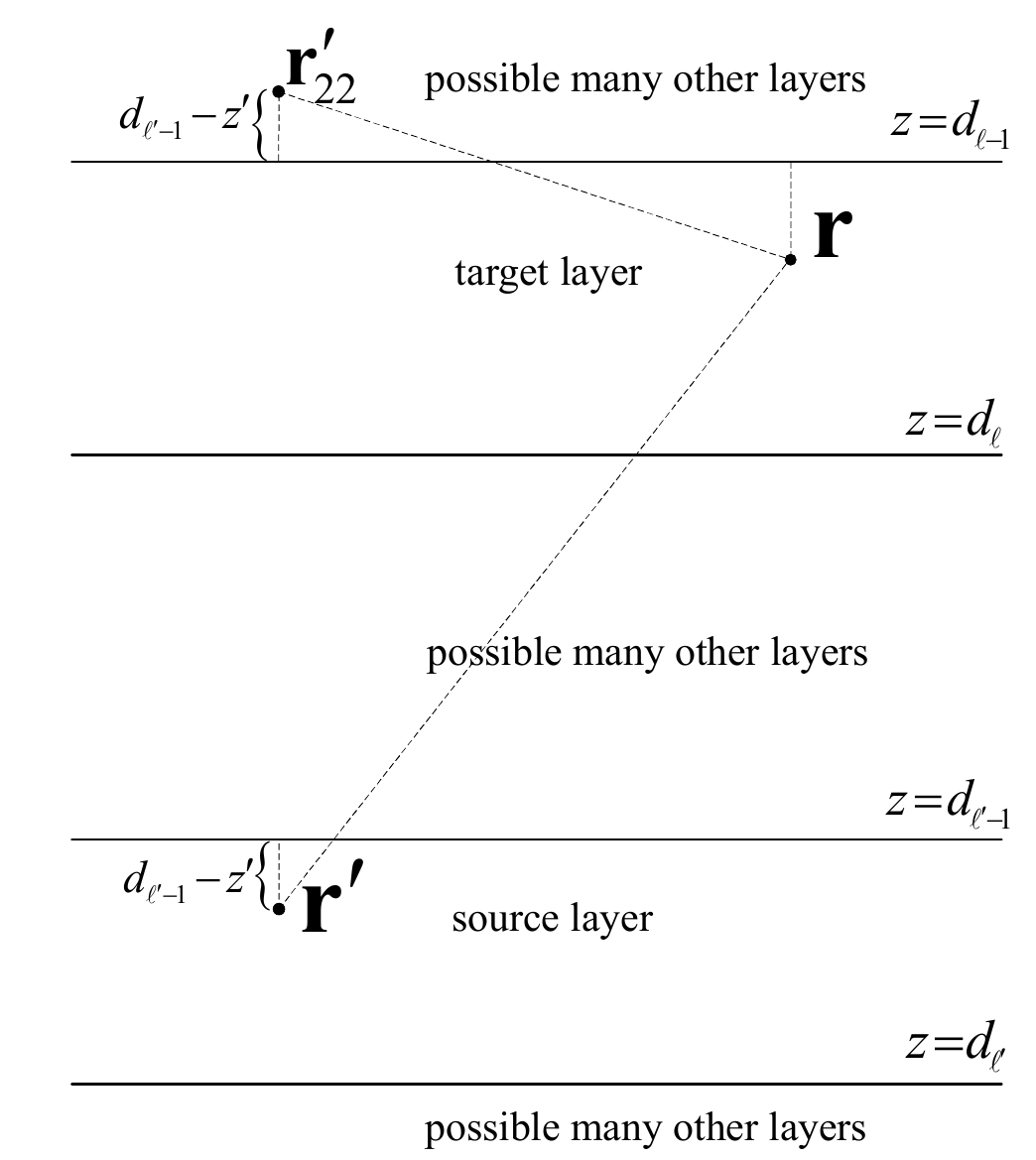}}
	\caption{Location of equivalent polarization sources for the computation of $u_{\ell\ell'}^{\mathfrak{ab}}$.}%
	\label{sourceimages}%
\end{figure}

We can see that the reaction potentials \eqref{reactfieldpolarization} represented by the equivalent polarization sources has similar form as the Sommerfeld-type integral representation \eqref{freegreenintegral} of the free space Green's function except for the extra density functions $\sigma_{\ell\ell'}^{\mathfrak{ab}}(k_{\rho})$. Moreover, recall the definition in \eqref{eqpolarizedsource} we have
\begin{equation*}
z>z'_{1\mathfrak b}, \quad{\rm and}\quad z<z'_{2\mathfrak b},\quad \mathfrak b=1, 2.
\end{equation*}
Therefore, the absolute value in the integral form \eqref{reactfieldpolarization} can be removed according to the index $\mathfrak a$. More precisely, define
\begin{equation}\label{mekernelimage}
{\mathcal E}^{+}(\bs r, \bs r'):=e^{\ri\bs k_{\alpha}\cdot(\bs\rho-\bs\rho')}e^{k_{\rho}(z-z^{\prime})},\quad {\mathcal E}^{-}(\bs r, \bs r^{\prime}):=e^{\ri\bs k_{\alpha}\cdot(\bs\rho-\bs\rho')}e^{-k_{\rho}(z-z^{\prime})},
\end{equation}
then
\begin{equation*}
\begin{split}
\tilde u_{\ell\ell'}^{1\mathfrak b}(\bs r, \bs r'_{1\mathfrak b})=\frac{1}{8\pi^2 }\int_0^{\infty}\int_0^{2\pi}{\mathcal E}^{-}(\bs r, \bs r'_{1\mathfrak b})\sigma_{\ell\ell'}^{1\mathfrak b}(k_{\rho})d\alpha dk_{\rho},\\
\tilde u_{\ell\ell'}^{2\mathfrak b}(\bs r, \bs r'_{2\mathfrak b})=\frac{1}{8\pi^2 }\int_0^{\infty}\int_0^{2\pi}\mathcal E^{+}(\bs r, \bs r'_{2\mathfrak b})\sigma_{\ell\ell'}^{2\mathfrak b}(k_{\rho})d\alpha dk_{\rho}.
\end{split}
\end{equation*}
Recall the expressions \eqref{zexponential}, we verify that
\begin{equation}\label{expkernelexp}
{\mathcal E}^{-}(\bs r, \bs r'_{1\mathfrak b})=e^{\ri\bs k_{\alpha}\cdot(\bs\rho-\bs\rho')}\mathcal{Z}_{\ell\ell'}^{1\mathfrak{b}}(z, z'),\quad {\mathcal E}^{+}(\bs r, \bs r'_{2\mathfrak b})=e^{\ri\bs k_{\alpha}\cdot(\bs\rho-\bs\rho')}\mathcal{Z}_{\ell\ell'}^{2\mathfrak{b}}(z, z'), \quad \mathfrak b=1, 2.
\end{equation}
Therefore, the reaction components \eqref{generalcomponents} is equal to the reaction potentials defined for associated equivalent polarization sources, i.e.,
\begin{equation}\label{generalcomponentsimag}
\begin{split}
u_{\ell\ell'}^{1\mathfrak b}(\bs r, \bs r')=\tilde u_{\ell\ell'}^{1\mathfrak b}(\bs r, \bs r'_{1\mathfrak b}),\quad
u_{\ell\ell'}^{2\mathfrak b}(\bs r, \bs r')=\tilde u_{\ell\ell'}^{2\mathfrak b}(\bs r, \bs r'_{2\mathfrak b}), \quad \mathfrak b=1, 2.
\end{split}
\end{equation}
A substitution into the expression of $\Phi_{\ell\ell^{\prime}}^{\mathfrak{ab}}(\boldsymbol{r}_{\ell i})$ in \eqref{freereactioncomponents} leads to
\begin{equation}\label{reactcompusingpolar}
\Phi_{\ell\ell^{\prime}}^{\mathfrak{ab}}(\boldsymbol{r}_{\ell i})=\sum
\limits_{j=1}^{N_{\ell^{\prime}}}Q_{\ell^{\prime}j}\tilde u_{\ell\ell^{\prime}%
}^{\mathfrak{ab}}(\boldsymbol{r}_{\ell i},\boldsymbol{r}^{\mathfrak ab}_{\ell^{\prime}j}%
),\quad \mathfrak{a,b}=1, 2,
\end{equation}
where
\begin{equation}\label{equivpolarcoord}
\begin{split}
&\boldsymbol{r}^{11}_{\ell^{\prime}j}=(x_{\ell'j}, y_{\ell'j}, d_{\ell}-(z_{\ell j}-d_{\ell'})), \quad\;\;\;\, \boldsymbol{r}^{12}_{\ell^{\prime}j}=(x_{\ell'j}, y_{\ell'j}, d_{\ell}-(d_{\ell'-1}-z_{\ell j})),\\
&\boldsymbol{r}^{21}_{\ell^{\prime}j}=(x_{\ell'j}, y_{\ell'j}, d_{\ell-1}+(z_{\ell j}-d_{\ell'})), \quad \boldsymbol{r}^{22}_{\ell^{\prime}j}=(x_{\ell'j}, y_{\ell'j}, d_{\ell-1}+(d_{\ell'-1}-z_{\ell j})),
\end{split}
\end{equation}
are coordinates of the associated equivalent polarization sources for the computation of reaction components $\Phi_{\ell\ell^{\prime}}^{\mathfrak{ab}}(\boldsymbol{r}_{\ell i})$, see Fig \ref{polarizedsource} for an illustration of $\{\bs r_{\ell'j}^{11}\}_{j=1}^{N_{\ell'}}$ and $\{\bs r_{\ell'j}^{21}\}_{j=1}^{N_{\ell'}}$.

By using the expression \eqref{reactcompusingpolar}, the computation of the reaction components can be performed between targets and associated equivalent polarization sources. The definition given by \eqref{equivpolarcoord} shows that the target particles $\{\bs r_{\ell i}\}_{i=1}^{N_{\ell}}$ and the corresponding equivalent polarization sources  are always located on different sides of an interface $z=d_{\ell-1}$ or $z=d_{\ell}$, see Fig. \ref{polarizedsource}. We still emphasize that the introduced equivalent polarization sources are separate with the target charges even in considering the reaction components for source and target charges in the same layer, see the numerical examples given in section 3.4. This property implies significant advantage of introducing equivalent polarization sources and using expression \eqref{reactcompusingpolar} in the FMMs for the reaction components $\Phi_{\ell\ell^{\prime}}^{\mathfrak{ab}}(\boldsymbol{r}_{\ell i})$, $\mathfrak a,b=1,2$. The numerical results presented in Section 4 show that the FMMs for reaction components have high efficiency as a direct consequence of the separation of the targets and equivalent polarization sources by interface.

%\begin{rem}
%{\color{black}	{\bf Interpretation of Polarized Sources:} In special situations such as the half space (a two layer medium) with an impedance boundary condition on the interface, the reaction
%	fields can be expressed in terms of those from point and line image charges located on the opposite side of the interface from the targets \cite{cho2018heterogeneous} \cite{cai2013computational}.
%	However, in multilayered cases considered here, the reaction fields, given in terms of complicated integral expressions, in general can not be
%	expressed in terms of explicit image charges. Nonetheless, we introduce the polarization sources, which will represent
%	the effect of the reaction fields and are given locations based on the convergence behaviors of the MEs and LEs in \eqref{melayerapp} for the corresponding reaction field components.	For the practical FMM implementation purpose, the locations of the polarized sources so defined enable us to decide wether the corresponding MEs can be used for the far field of the reaction components based on the distance between the far field location and the polarization sources, and thus, make the extension of FMMs to source interactions in layered media straightforward.}
%\end{rem}

\subsection{\color{black}Fast multipole algorithm}

In the development of FMM for reaction components $\Phi_{\ell\ell'}^{\mathfrak{ab}}(\bs r_{\ell i})$, we will adopt the expression \eqref{reactcompusingpolar} with equivalent polarization sources. Therefore, multipole and local expansions and corresponding translation operators for $\tilde u_{\ell\ell'}^{\mathfrak{ab}}(\bs r, \bs r'_{\mathfrak{ab}})$ are derived first. Inspired by source/target separation in \eqref{positivecase}, similar separations
\begin{equation}\label{sourcetargetseparationsc}
\begin{split}
{\mathcal E}^{-}(\bs r, \bs r'_{1\mathfrak b})=\mathcal E^{-}(\bs r, \bs r^{1\mathfrak b}_c)e^{\ri\bs k_{\alpha}\cdot(\bs\rho_c^{1\mathfrak b}-\bs\rho'_{1\mathfrak b})-k_{\rho}(z^{1\mathfrak b}_c-z'_{1\mathfrak b})},\\
{\mathcal E}^{+}(\bs r, \bs r'_{2\mathfrak b})=\mathcal E^{+}(\bs r, \bs r^{2\mathfrak b}_c)e^{\ri\bs k_{\alpha}\cdot(\bs\rho_c^{2\mathfrak b}-\bs\rho'_{2\mathfrak b})+k_{\rho}(z^{2\mathfrak b}_c-z'_{2\mathfrak b})},
\end{split}
\end{equation}
and
\begin{equation}\label{sourcetargetseparationtc}
\begin{split}
\mathcal E^-(\bs r, \bs r'_{1\mathfrak b})&=\mathcal E^{-}(\bs r_c^t, \bs r'_{1\mathfrak b})e^{\ri\bs k_{\alpha}\cdot(\bs\rho-\bs\rho_c^t)-k_{\rho}(z-z_c^t)},\\
\mathcal E^+(\bs r, \bs r'_{2\mathfrak b})&=\mathcal E^{+}(\bs r_c^t, \bs r'_{2\mathfrak b})e^{\ri\bs k_{\alpha}\cdot(\bs\rho-\bs\rho_c^t)+k_{\rho}(z-z_c^t)},
\end{split}
\end{equation}
for $\mathfrak b=1, 2$ are introduced by inserting the source center $\bs r^{\mathfrak {ab}}_c=(x_c^{\mathfrak{ab}},y_c^{\mathfrak{ab}}, z_c^{\mathfrak{ab}} )$ and the target center $\bs r_c^t=(x_c^t, y_c^t, z_c^t)$, respectively. Here, we also use notations $\bs\rho_c^{\mathfrak{ab}}=(x_c^{\mathfrak{ab}}, y_c^{\mathfrak{ab}})$, $\bs\rho_c^t=(x_c^t, y_c^t)$ for coordinates projected in $xy$-plane.
Now, applying proposition \ref{prop:Funk-Hecke-limit} gives us the following spherical harmonic expansions:
\begin{equation}\label{planewavesphexpmultipole}
\begin{split}
&e^{\ri\bs k_{\alpha}\cdot(\bs\rho_c^{2\mathfrak b}-\bs\rho'_{2\mathfrak b})+k_{\rho}(z^{2\mathfrak b}_c-z'_{2\mathfrak b})}=\sum\limits_{n=0}^{\infty}\sum\limits_{m=-n}^{n} C_n^m  (r_c^{2\mathfrak b})^{n}\overline{Y_n^m(\theta^{2\mathfrak b}_c,\pi+\varphi^{2\mathfrak b}_c)}k_{\rho}^ne^{\ri m\alpha},\\
&e^{\ri\bs k_{\alpha}\cdot(\bs\rho_c^{1\mathfrak b}-\bs\rho'_{1\mathfrak b})-k_{\rho}(z^{1\mathfrak b}_c-z'_{1\mathfrak b})}=\sum\limits_{n=0}^{\infty}\sum\limits_{m=-n}^{n} C_n^m  (r_c^{1\mathfrak b})^n\overline{Y_n^m(\pi-\theta^{1\mathfrak b}_c,\pi+\varphi^{1\mathfrak b}_c)}k_{\rho}^ne^{\ri m\alpha},\\
\end{split}
\end{equation}
and
\begin{equation}\label{planewavesphexplocal}
\begin{split}
&e^{\ri\bs k_{\alpha}\cdot(\bs\rho-\bs\rho_c^t)- k_{\rho}(z-z_c^t)}=\sum\limits_{n=0}^{\infty}\sum\limits_{m=-n}^{n} C_n^m r_t^nY_n^m(\theta_t,\varphi_t)k_{\rho}^ne^{-\ri m\alpha},\\
&e^{\ri\bs k_{\alpha}\cdot(\bs\rho-\bs\rho_c^t)+ k_{\rho}(z-z_c^t)}=\sum\limits_{n=0}^{\infty}\sum\limits_{m=-n}^{n} C_n^m r_t^nY_n^m(\pi-\theta_t,\varphi_t)k_{\rho}^ne^{-\ri m\alpha},
\end{split}
\end{equation}
where $(r_c^{\mathfrak{ab}}, \theta_c^{\mathfrak{ab}}, \varphi_c^{\mathfrak{ab}})$ is the spherical coordinates of $\bs r'_{\mathfrak{ab}}-\bs r_c^{\mathfrak{ab}}$.
By equalities
$$Y_n^{m}(\pi-\theta,\varphi)=(-1)^{n+m}Y_n^{m}(\theta,\varphi),\quad Y_n^{m}(\theta,\pi+\varphi)=(-1)^{m}Y_n^{m}(\theta,\varphi),$$
the above spherical harmonic expansions \eqref{planewavesphexpmultipole}-\eqref{planewavesphexplocal} together with source/target separation \eqref{sourcetargetseparationsc} and \eqref{sourcetargetseparationtc} lead to
\begin{equation}\label{integrandme}
\begin{split}
\mathcal E^-(\bs r, \bs r'_{1\mathfrak b})&=\mathcal E^-(\bs r, \bs r_c^{1\mathfrak b})\sum\limits_{n=0}^{\infty}\sum\limits_{m=-n}^{n} (-1)^nC_n^m ( r^{1\mathfrak b}_c)^n\overline{Y_n^m(\theta^{1\mathfrak b}_c,\varphi^{1\mathfrak b}_c)}k_{\rho}^ne^{\ri m\alpha},\\
\mathcal E^+(\bs r, \bs r'_{2\mathfrak b})&=\mathcal E^+(\bs r, \bs r_c^{2\mathfrak b})\sum\limits_{n=0}^{\infty}\sum\limits_{m=-n}^{n} (-1)^{m}C_n^m ( r^{2\mathfrak b}_c)^n\overline{Y_n^m(\theta^{2\mathfrak b}_c,\varphi^{2\mathfrak b}_c)}k_{\rho}^ne^{\ri m\alpha},
\end{split}
\end{equation}
and
\begin{equation}\label{integrandle}
\begin{split}
\mathcal E^-(\bs r, \bs r'_{1\mathfrak b})&=\mathcal E^-(\bs r_c^t, \bs r'_{1\mathfrak b})\sum\limits_{n=0}^{\infty}\sum\limits_{m=-n}^{n} C_n^m r_t^nY_n^m(\theta_t,\varphi_t)k_{\rho}^ne^{-\ri m\alpha},\\
\mathcal E^+(\bs r, \bs r'_{2\mathfrak b})&=\mathcal E^+(\bs r_c^t, \bs r'_{2\mathfrak b})\sum\limits_{n=0}^{\infty}\sum\limits_{m=-n}^{n} (-1)^{n+m}C_n^m r_t^nY_n^m(\theta_t,\varphi_t)k_{\rho}^ne^{-\ri m\alpha},
\end{split}
\end{equation}
for $\mathfrak b=1, 2$. Then, a substitution of \eqref{integrandme} and \eqref{integrandle} into \eqref{generalcomponentsimag} gives a ME
\begin{equation}\label{melayerupgoingimage}
\begin{split}
\tilde u_{\ell\ell'}^{\mathfrak{ab}}(\bs r, \bs r'_{\mathfrak{ab}})=\sum\limits_{n=0}^{\infty}\sum\limits_{m=-n}^{n}  M_{nm}^{\mathfrak{ab}}\widetilde{\mathcal F}_{nm}^{\mathfrak{ab}}(\bs r, \bs r_c^{\mathfrak{ab}}), \quad M_{nm}^{\mathfrak{ab}}=c_n^{-2} (r_c^{\mathfrak{ab}})^n\overline{Y_n^{m}(\theta_c^{\mathfrak{ab}},\varphi_c^{\mathfrak{ab}})},
\end{split}
\end{equation}
at equivalent polarization source centers $\bs r_c^{\mathfrak{ab}}$ and LE
\begin{equation}\label{lelayerimage}
\begin{split}
\tilde u_{\ell\ell'}^{\mathfrak{ab}}(\bs r, \bs r'_{\mathfrak{ab}})=\sum\limits_{n=0}^{\infty}\sum\limits_{m=-n}^{n} L_{nm}^{\mathfrak{ab}}r_t^nY_n^m(\theta_t,\varphi_t)
\end{split}
\end{equation}
at target center $\bs r_c^t$, respectively. Here, $\widetilde{\mathcal F}_{nm}^{\mathfrak{ab}}(\bs r, \bs r_c^{\mathfrak{ab}})$ are given in forms of Sommerfeld-type integrals
\begin{equation}\label{mebasis}
\begin{split}
\widetilde{\mathcal F}_{nm}^{1\mathfrak b}(\bs r, \bs r_c^{1\mathfrak b})=&\frac{(-1)^{n}c_n^2C_n^m}{8\pi^2}\int_0^{\infty}\int_0^{2\pi}{\mathcal E}^{-}(\bs r, \bs r_c^{1\mathfrak b})\sigma_{\ell\ell'}^{1\mathfrak b}(k_{\rho})k_{\rho}^{n}e^{\ri m\alpha}d\alpha dk_{\rho},\\
\widetilde{\mathcal F}_{nm}^{2\mathfrak b}(\bs r, \bs r_c^{2\mathfrak b})=&\frac{(-1)^mc_n^2C_n^m}{8\pi^2}\int_0^{\infty}\int_0^{2\pi}{\mathcal E}^{+}(\bs r, \bs r_c^{2\mathfrak b})\sigma_{\ell\ell'}^{2\mathfrak b}(k_{\rho})k_{\rho}^{n}e^{\ri m\alpha}d\alpha dk_{\rho},
\end{split}
\end{equation}
and the LE coefficients are given by
\begin{equation}\label{lecoeffimage}
\begin{split}
L_{nm}^{1\mathfrak b}=&\frac{C_n^m}{8\pi^2 }\int_0^{\infty}\int_0^{2\pi}{\mathcal E}^{-}(\bs r_c^t, \bs r'_{1\mathfrak b})\sigma_{\ell\ell^{\prime}}^{1\mathfrak b}(k_{\rho})k_{\rho}^{n}e^{-\ri m\alpha}d\alpha dk_{\rho},\\
L_{nm}^{2\mathfrak b}=&\frac{(-1)^{n+m}C_n^m}{8\pi^2 }\int_0^{\infty}\int_0^{2\pi}{\mathcal E}^{+}(\bs r_c^t, \bs r'_{2\mathfrak b})\sigma_{\ell\ell^{\prime}}^{2\mathfrak b}(k_{\rho})k_{\rho}^{n}e^{-\ri m\alpha}d\alpha dk_{\rho}.
\end{split}
\end{equation}

Let us give some numerical examples to show the convergence behavior of the MEs in \eqref{melayerupgoingimage}. Consider the MEs of $\tilde u_{11}^{11}(\bs r, \bs r'_{11})$ and $\tilde u_{11}^{22}(\bs r, \bs r'_{22})$ in a three-layer media with $\varepsilon_0=21.2$, $\varepsilon_1=47.5$, $\varepsilon_2=62.8$, $d_0=0, d_1=-1.2$. In all the following examples, we fix $\bs r'=(0.625, 0.5, -0.1)$ in the middle layer and use definition \eqref{eqpolarizedsource} to determine $\bs r_{11}^{\prime}=(0.625, 0.5, -2.3)$, $\bs r_{22}^{\prime}=(0.625, 0.5, 0.1)$. The centers for MEs are set to be  $\bs r_c^{11}=(0.6, 0.6, -2.4)$, $\bs r_c^{22}=(0.6, 0.6, 0.2)$ which implies $|\bs r_{11}'-\bs r_c^{11}|=|\bs r_{22}'-\bs r_c^{22}|\approx0.1436$. For both components, we shall test three targets given as follows
\begin{equation*}
\bs r_1=(0.5, 0.625, -0.1),\quad \bs r_2= (0.5, 0.625, -0.6), \quad \bs r_3=(0.5, 0.625, -1.1).
\end{equation*}
The relative errors against truncation number $p$ are depicted in Fig. \ref{meconvergence}. We also plot the convergence rates similar with that of the ME of free space Green's function, i.e., $O\Big[\Big(\frac{|\bs r-\bs r_c^{\mathfrak{ab}}|}{|\bs r_{\mathfrak{ab}}'-\bs r_c^{\mathfrak{ab}}|}\Big)^{p+1}\Big]$ as reference convergence rates. The results clearly show that the MEs of the reaction components $u_{11}^{\mathfrak{ab}}(\bs r, \bs r'_{\mathfrak{ab}})$ have spectral convergence rate $O\Big[\Big(\frac{|\bs r-\bs r_c^{\mathfrak{ab}}|}{|\bs r_{\mathfrak{ab}}'-\bs r_c^{\mathfrak{ab}}|}\Big)^{p+1}\Big]$ similar as that of free space Green's function. Actually, their exponential convergence has been determined by the Euclidean distance between target and polarization source. Therefore, the MEs \eqref{melayerupgoingimage} can be used to develop FMM for efficient computation of the reaction components as in the development of classic FMM for the free space Green's function.
\begin{figure}[ptbh]
	\center
	\subfigure[$\tilde u_{11}^{11}(\bs r, \bs r'_{11})$]{\includegraphics[scale=0.4]{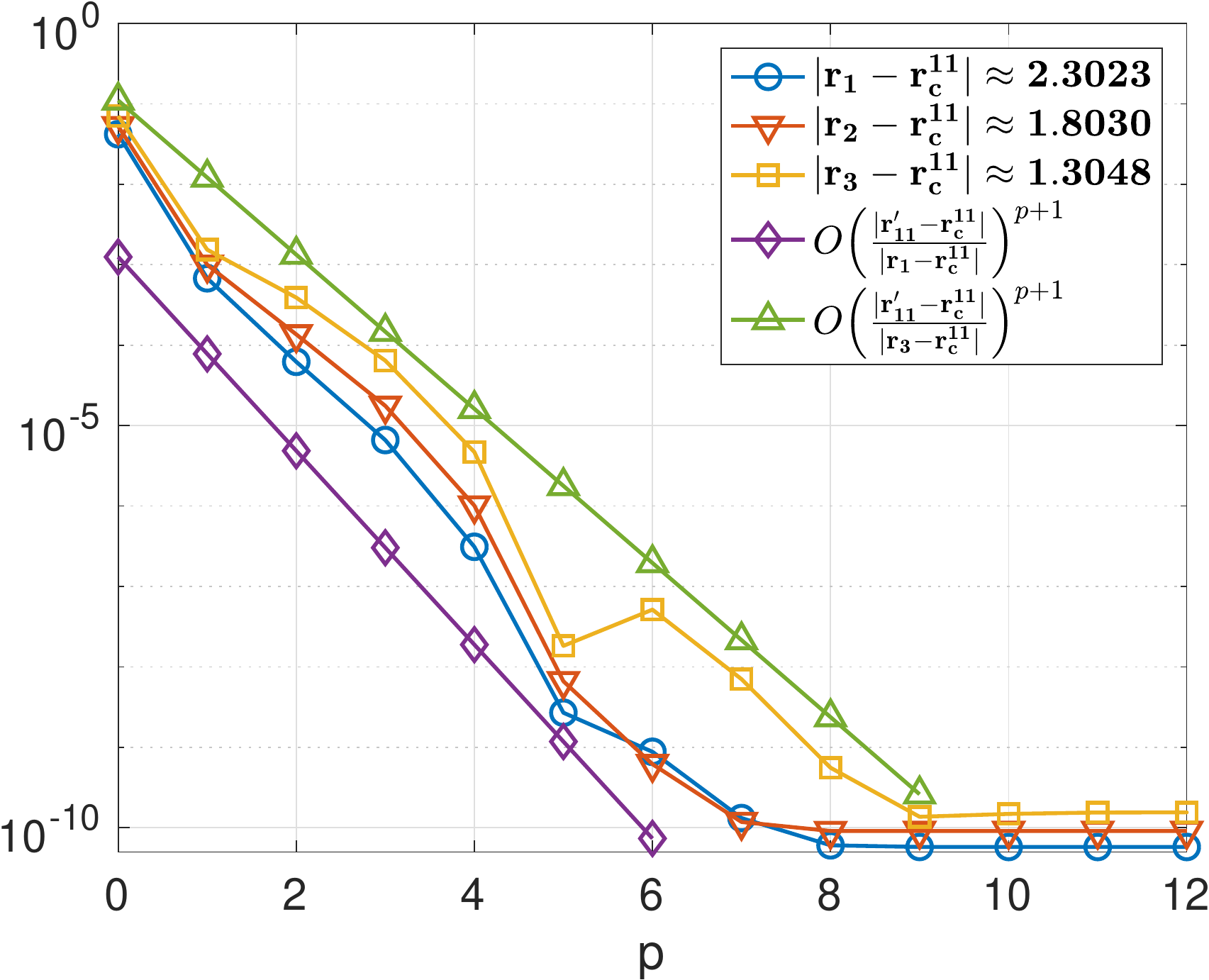}}\quad
	\subfigure[$\tilde u_{11}^{22}(\bs r, \bs r'_{22})$]{\includegraphics[scale=0.4]{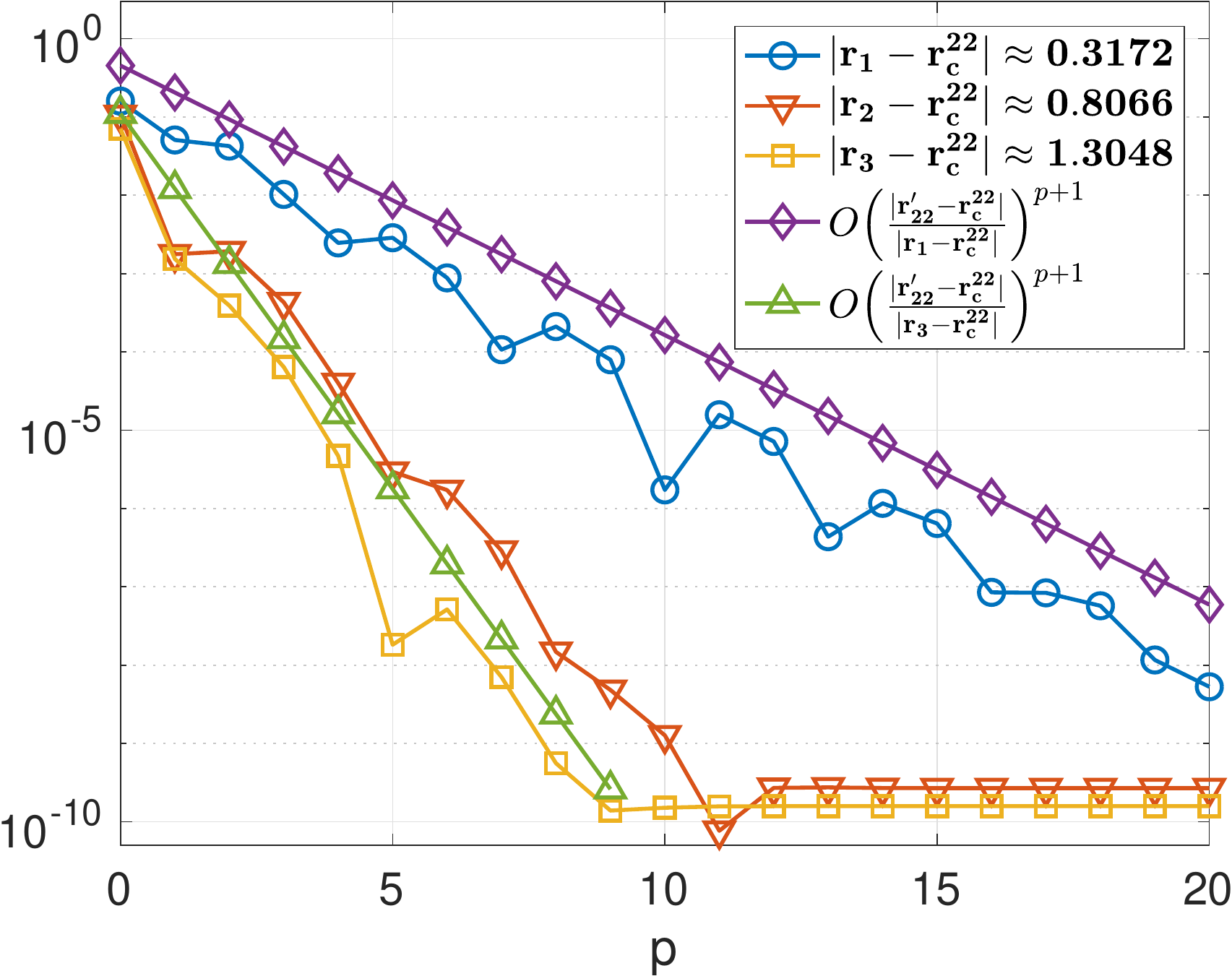}}
	\caption{Spectral convergence of the MEs for reaction components.}%
	\label{meconvergence}%
\end{figure}

According to the definition of $\mathcal E^-(\bs r, \bs r')$ and $\mathcal E^+(\bs r, \bs r')$  in \eqref{expkernelexp},  the centers $\bs r_c^t$ and $\bs r_c^{\mathfrak{ab}}$ have to satisfy
\begin{equation}\label{imagecentercond}
z_c^{1\mathfrak b}<d_{\ell}, \quad z_c^{2\mathfrak b}>d_{\ell-1},\quad z_c^t>d_{\ell} \;\; {\rm for}\;\; \tilde u_{\ell\ell'}^{1\mathfrak b}(\bs r, \bs r'_{1\mathfrak b}); \quad z_c^t<d_{\ell-1} \;\; {\rm for}\;\; \tilde u_{\ell\ell'}^{2\mathfrak b}(\bs r, \bs r'_{2\mathfrak b}),
\end{equation}
to ensure the exponential decay in $\mathcal E^-(\bs r, \bs r^{1\mathfrak b}_c), \mathcal E^+(\bs r, \bs r^{2\mathfrak b}_c)$ and $\mathcal E^{-}(\bs r_c^t, \bs r'_{1\mathfrak b}), \mathcal E^{+}(\bs r_c^t, \bs r'_{2\mathfrak b})$ as $k_{\rho}\rightarrow\infty$ and hence the convergence of the corresponding Sommerfeld-type integrals in \eqref{mebasis} and \eqref{lecoeffimage}. These restrictions can be met easily in practice, as we are considering targets in the $\ell$-th layer and the equivalent polarized coordinates are always located either above the interface $z=d_{\ell-1}$ or below the interface $z=d_{\ell}$. More details will discussed below in the presentation of the FMM algorithm.
	
%Recall convergence results \eqref{melayerapp} and the fact \eqref{distancerelation},
%we conclude that ME \eqref{melayerupgoingimage} now satisfies
%\begin{equation}
%\begin{split}
%\Big|\tilde u_{\ell\ell'}^{\mathfrak{ab}}(\bs r, \bs r'_{\mathfrak{ab}})-\sum\limits_{n=0}^{p}\sum\limits_{m=-n}^{n}  M_{nm}^{\mathfrak{ab}}\widetilde{\mathcal F}_{nm}^{\mathfrak{ab}}(\bs r, \bs r_c^{\mathfrak{ab}})\Big|=\mathcal{O}\left(\left(\frac{|\bs r'_{\mathfrak{ab}}-\bs r_c^{\mathfrak{ab}}|}{|\bs r-\bs r_c^{\mathfrak{ab}}|}\right)^p\right).
%\end{split}
%\end{equation}
%As a result, the Euclidean distance  $|\bs r-\bs r'_{\mathfrak{ab}}|$ can be used to determine if multipole expansions \eqref{melayerupgoingimage} are good approximations to the far reaction field. Similarly, the local expansion \eqref{lelayerimage} for $\tilde u_{\ell\ell'}^{\mathfrak{ab}}(\bs r, \bs r'_{\mathfrak{ab}})$ have a convergence of order $\mathcal{O}\left(\left(\frac{|\bs r-\bs r_c^l|}{|\bs r_c^l-\bs r'_{\mathfrak{ab}}|}\right)^p\right)$. These convergence results ensure that the hierarchical design can be applied in FMM with kernels $\tilde u_{\ell\ell'}^{\mathfrak{ab}}(\bs r, \bs r'_{\mathfrak{ab}})$, ME \eqref{melayerupgoingimage} and LE \eqref{lelayerimage}.

We still need to consider the center shifting and translation operators for ME \eqref{melayerupgoingimage} and LE \eqref{lelayerimage}. A desirable feature of the expansions of reaction components discussed above is that the formula \eqref{melayerupgoingimage} for the ME coefficients and the formula \eqref{lelayerimage} for the LE have exactly the same form as the formulas of ME coefficients and LE for the free space Green's function. Therefore, the center shifting for MEs and LEs of reaction components are exactly the same as free space case given in \eqref{metome}-\eqref{letole}.

Next, we derive the translation operator from the ME \eqref{melayerupgoingimage} to the LE \eqref{lelayerimage}. Recall the definition of exponential functions in \eqref{mekernelimage}, ${\mathcal E}^{-}(\bs r, \bs r_c^{1\mathfrak b})$ and ${\mathcal E}^{+}(\bs r, \bs r_c^{2\mathfrak b})$ can have splitting
\begin{equation*}
\begin{split}
{\mathcal E}^{-}(\bs r, \bs r_c^{1\mathfrak{b}})&={\mathcal E}^{-}(\bs r_c^t, \bs r_c^{1\mathfrak b})e^{\ri\bs k_{\alpha}\cdot(\bs\rho-\bs\rho_c^t)}e^{-k_{\rho} (z-z_c^t)},\\
{\mathcal E}^{+}(\bs r, \bs r_c^{2\mathfrak b})&={\mathcal E}^{+}(\bs r_c^t, \bs r_c^{2\mathfrak b})e^{\ri\bs k_{\alpha}\cdot(\bs\rho-\bs\rho_c^t)}e^{k_{\rho} (z-z_c^t)}.
\end{split}
\end{equation*}
Applying spherical harmonic expansion \eqref{extfunkheckelim} again, we obtain
\begin{equation*}
e^{\ri\bs k_{\alpha}\cdot(\bs\rho-\bs\rho_c^t)}e^{\pm k_{\rho} (z-z_c^t)}=\sum\limits_{n=0}^{\infty}\sum\limits_{m=-n}^{n} (\mp 1)^{n+m}C_n^m r_t^nY_n^m(\theta_t,\varphi_t)k_{\rho}^ne^{-\ri m\alpha}.
\end{equation*}
Substituting into \eqref{melayerupgoingimage}, the ME is translated to LE \eqref{lelayerimage} via
\begin{equation}\label{metoleimage1}
L_{nm}^{1\mathfrak{b}}=\sum\limits_{n'=0}^{\infty}\sum\limits_{|m'|=0}^{n'}T_{nm,n'm'}^{1\mathfrak{b}}M_{n'm'}^{1\mathfrak{b}},\quad L_{nm}^{2\mathfrak{b}}=(-1)^{n+m}\sum\limits_{n'=0}^{\infty}\sum\limits_{|m'|=0}^{n'}T_{nm,n'm'}^{2\mathfrak{b}}M_{n'm'}^{2\mathfrak{b}},
\end{equation}
and the M2L translation operators are given in integral forms as follows
\begin{equation}\label{metoleimage2}
\begin{split}
T_{nm,n'm'}^{1\mathfrak{b}}=&\frac{(-1)^{n'}D_{nm}^{n'm'}}{8\pi^2}\int_0^{\infty}\int_0^{2\pi}{\mathcal E}^{-}(\bs r_c^t, \bs r_c^{1\mathfrak{b}})\sigma_{\ell\ell'}^{1\mathfrak{b}}(k_{\rho})k_{\rho}^{n+n'}e^{\ri (m'-m)\alpha}d\alpha dk_{\rho},\\
T_{nm,n'm'}^{2\mathfrak{b}}=&\frac{(-1)^{m'}D_{nm}^{n'm'}}{8\pi^2}\int_0^{\infty}\int_0^{2\pi}{\mathcal E}^{+}(\bs r_c^t, \bs r_c^{2\mathfrak{b}})\sigma_{\ell\ell'}^{2\mathfrak{b}}(k_{\rho})k_{\rho}^{n+n'}e^{\ri (m'-m)\alpha}d\alpha dk_{\rho},
\end{split}
\end{equation}
where
$$D_{nm}^{n'm'}=c_{n'}^2C_n^mC_{n'}^{m'}.$$
Again, the convergence of the Sommerfeld-type integrals in \eqref{metoleimage2} is ensured by the conditions in \eqref{imagecentercond}.

{\color{black}
The framework of the traditional FMM together with ME \eqref{melayerupgoingimage}, LE \eqref{lelayerimage}, M2L translation \eqref{metoleimage1}-\eqref{metoleimage2} and free space ME and LE center shifting \eqref{metome} and \eqref{letole} constitute the FMM for the computation of reaction components $\Phi_{\ell\ell'}^{\mathfrak{ab}}(\bs r_{\ell i})$, $\mathfrak a, \mathfrak b=1, 2$. In the FMM for each reaction component, a large box is defined to include all equivalent polarization sources associated to the reaction component and corresponding target charges, and an adaptive tree structure will be built by a bisection procedure, see. Fig. \ref{polarizedsource}.  Note that the validity of the ME \eqref{melayerupgoingimage}, LE \eqref{lelayerimage} and M2L translation \eqref{metoleimage1} used in the algorithm imposes restrictions \eqref{imagecentercond} on the centers, accordingly. This can be ensured by setting the largest box for the specific reaction component to be equally divided by the interface between equivalent polarized sources and corresponding targets, see. Fig. \ref{polarizedsource}. Thus, the largest box for the FMM implementation will be different for different reaction components.  With this setting, all source and target boxes of higher than zeroth level in the adaptive tree structure will have centers below or above the interfaces, accordingly. The fast multipole algorithm for the computation of a general reaction component $\Phi_{\ell\ell'}^{\mathfrak{ab}}(\bs r_{\ell i})$ is summarized in Algorithm 1.
Total interactions given by \eqref{totalinteraction} will be obtained by first calculating all components and then summing them up where the algorithm is presented in Algorithm 2.}

\subsection{\color{black}Efficient computation of Sommerfeld-type integrals}
It is clear that the FMM demands efficient computation of the double integrals involved in the MEs, LEs and M2L translations. {\color{black}In this section, we present an accurate and efficient way to compute these double integrals.} Firstly, the double integrals can be simplified by using the following identity
\begin{equation}
J_n(z)=\frac{1}{2\pi \ri^n}\int_0^{2\pi}e^{\ri z\cos\theta+\ri n\theta}d\theta.
\end{equation}
In particular, multipole expansion functions in \eqref{mebasis} can be simplified as
\begin{equation*}
\begin{split}
\widetilde{\mathcal F}_{nm}^{1\mathfrak b}(\bs r, \bs r_c^{1\mathfrak b})=&\frac{(-1)^{n}c_n^2C_n^m\ri^me^{\ri m\phi_s^{1\mathfrak b}}}{4\pi}\int_0^{\infty}J_m(k_{\rho}\rho_s^{1\mathfrak b})e^{-k_{\rho}(z-z_c^{1\mathfrak b})}\sigma_{\ell\ell'}^{1\mathfrak b}(k_{\rho})k_{\rho}^{n}dk_{\rho},\\
\widetilde{\mathcal F}_{nm}^{2\mathfrak b}(\bs r, \bs r_c^{2\mathfrak b})=&\frac{(-1)^mc_n^2C_n^m\ri^me^{\ri m\phi_s^{2\mathfrak b}}}{4\pi}\int_0^{\infty}J_m(k_{\rho}\rho_s^{2\mathfrak b})e^{-k_{\rho}(z_c^{2\mathfrak b}-z)}\sigma_{\ell\ell'}^{2\mathfrak b}(k_{\rho})k_{\rho}^{n} dk_{\rho},
\end{split}
\end{equation*}
and the expression \eqref{lecoeffimage} for LE coefficients can be simplified as
\begin{equation*}
\begin{split}
L_{nm}^{1\mathfrak b}=&\frac{(-1)^mC_n^m\ri^{-m}e^{-\ri m\varphi_t^{1\mathfrak{b}}}}{4\pi }\int_0^{\infty}J_{m}(k_{\rho}\rho_t^{1\mathfrak{b}})e^{-k_{\rho}(z_c^t-z_{1\mathfrak b}')}\sigma_{\ell\ell^{\prime}}^{1\mathfrak b}(k_{\rho})k_{\rho}^{n} dk_{\rho},\\
L_{nm}^{2\mathfrak b}=&\frac{(-1)^{n}C_n^m\ri^{-m}e^{-\ri m\varphi_t^{2\mathfrak{b}}}}{4\pi }\int_0^{\infty}J_{m}(k_{\rho}\rho_t^{2\mathfrak{b}})e^{-k_{\rho}(z_{2\mathfrak b}'-z_c^t)}\sigma_{\ell\ell^{\prime}}^{2\mathfrak b}(k_{\rho})k_{\rho}^{n}dk_{\rho}
\end{split}
\end{equation*}
for $\mathfrak b=1, 2$, where $(\rho_s^{\mathfrak{ab}}, \varphi_s^{\mathfrak{ab}})$ and $(\rho_t^{\mathfrak{ab}}, \varphi_t^{\mathfrak{ab}})$ are polar coordinates of $\bs r-\bs r_c^{\mathfrak{ab}}$ and $\bs r_c^t-\bs r'_{\mathfrak{ab}}$ projected in the $xy$ plane.
\begin{algorithm}\label{algorithm1}
	\caption{FMM for general reaction component $\Phi_{\ell\ell'}^{\mathfrak{ab}}(\bs r_{\ell i}), i=1, 2, \cdots, N_{\ell}$}
	\begin{algorithmic}
		\State Determine equivalent polarized coordinates for all source particles.
		\State {\color{black}Generate an adaptive hierarchical tree structure with polarization sources $\{Q_{\ell'j}, \bs r_{\ell'j}^{\mathfrak{ab}}\}_{j=1}^{N_{\ell'}}$, targets $\{\bs r_{\ell i}\}_{i=1}^{N_{\ell}}$.}
		\State{\bf Upward pass:}
		\For{$l=H \to 0$}
		\For{all boxes $j$ on source tree level $l$ }
		\If{$j$ is a leaf node}
		\State{form the free-space ME using Eq. \eqref{melayerupgoingimage}.}
		\Else
		\State form the free-space ME by merging children's expansions using the free-space center shift translation operator \eqref{metome}.
		\EndIf
		\EndFor
		\EndFor
		\State{\bf Downward pass:}
		\For{$l=1 \to H$}
		\For{all boxes $j$ on target tree level $l$ }
		\State shift the LE of $j$'s parent to $j$ itself using the free-space shifting \eqref{letole}.
		\State collect interaction list contribution using the source box to target box translation operator in Eq. \eqref{metoleimage1} while $T_{nm,n'm'}^{\mathfrak{ab}}$ are computed using \eqref{coefintable} and recurrence formula \eqref{recurrence}.
		\EndFor
		\EndFor
		\State {\bf Evaluate LEs:}
		\For{each leaf node (childless box)}
		\State evaluate the LE at each particle location.
		\EndFor
		\State {\bf Local Direct Interactions:}
		\For{$i=1 \to N$ }
		\State compute Eq. \eqref{reactcompusingpolar} of target particle $i$ in the neighboring boxes {\color{black}using the mixed DE-SE quadrature for $I_{00}^{\mathfrak{ab}}(\rho, z)$.}
		\EndFor
	\end{algorithmic}
\end{algorithm}
\begin{algorithm}\label{algorithm2}
	\caption{3-D FMM for \eqref{totalinteraction}}
	\begin{algorithmic}
		\For{$\ell=0 \to L$}
		\State{use free space FMM to compute $\Phi_{\ell}^{free}(\bs r_{\ell i})$, $i=1, 2, \cdots, N_{\ell}$.}
		\EndFor
		\For{$\ell=0 \to L-1$}
		\For{$\ell'=0 \to L-1$ }
		\State use {\bf Algorithm 1} to compute $\Phi_{\ell\ell'}^{11}(\bs r_{\ell i})$, $i=1, 2, \cdots, N_{\ell}$.
		\EndFor
		\For{$\ell'=1 \to L$ }
		\State use {\bf Algorithm 1} to compute $\Phi_{\ell\ell'}^{12}(\bs r_{\ell i})$, $i=1, 2, \cdots, N_{\ell}$.
		\EndFor
		\EndFor
		\For{$\ell=1 \to L$}
		\For{$\ell'=0 \to L-1$ }
		\State use {\bf Algorithm 1} to compute $\Phi_{\ell\ell'}^{21}(\bs r_{\ell i})$, $i=1, 2, \cdots, N_{\ell}$.
		\EndFor
		\For{$\ell'=1 \to L$ }
		\State use {\bf Algorithm 1} to compute $\Phi_{\ell\ell'}^{22}(\bs r_{\ell i})$, $i=1, 2, \cdots, N_{\ell}$.
		\EndFor
		\EndFor
	\end{algorithmic}
\end{algorithm}
Moreover, the M2L translation \eqref{metoleimage2} can be simplified as
\begin{equation}\label{me2lesimplified}
\begin{split}
T_{nm,n'm'}^{1\mathfrak{b}}=&\frac{(-1)^{n'}\widetilde D_{nm}^{n'm'}(\varphi_{ts}^{1\mathfrak b})}{4\pi}\int_0^{\infty}k_{\rho}^{n+n'} J_{m'-m}(k_{\rho}\rho_{ts}^{1\mathfrak b})e^{-k_{\rho}(z_c^t-z_c^{1\mathfrak b})}\sigma_{\ell\ell'}^{1\mathfrak{b}}(k_{\rho})dk_{\rho},\\
T_{nm,n'm'}^{2\mathfrak{b}}=&\frac{(-1)^{m'}\widetilde D_{nm}^{n'm'}(\varphi_{ts}^{2\mathfrak b})}{4\pi}\int_0^{\infty}k_{\rho}^{n+n'} J_{m'-m}(k_{\rho}\rho_{ts}^{2\mathfrak b})e^{-k_{\rho}(z_c^{2\mathfrak b}-z_c^t)}\sigma_{\ell\ell'}^{2\mathfrak{b}}(k_{\rho})dk_{\rho},
\end{split}
\end{equation}
where $(\rho_{ts}^{\mathfrak{ab}}, {\varphi}_{ts}^{\mathfrak{ab}})$ is the polar coordinates of $\bs r_c^t-\bs r_c^{\mathfrak{ab}}$ projected in the $xy$ plane,
$$\widetilde D_{nm}^{n'm'}(\varphi)=D_{nm}^{n'm'}\ri^{m'-m}e^{\ri(m'-m){\varphi}}.$$
Define integral
\begin{equation}\label{uniformintegral}
\begin{split}
I_{nm}^{\mathfrak{ab}}(\rho, z):=\int_0^{\infty}J_{m}(k_{\rho}\rho)\frac{k_{\rho}^{n}e^{-k_{\rho}z}}{\sqrt{(n+m)!(n-m)!}}\sigma_{\ell\ell'}^{\mathfrak{ab}}(k_{\rho}) dk_{\rho},
\end{split}
\end{equation}
then
\begin{equation}\label{coefintable}
\begin{split}
&\widetilde{\mathcal F}_{nm}^{1\mathfrak b}(\bs r, \bs r_c^{1\mathfrak b})=\frac{c_ne^{\ri m\varphi_s^{1\mathfrak b}}}{4\pi}I_{nm}^{1\mathfrak b}(\rho_s^{1\mathfrak b}, z-z_c^{1\mathfrak b}),\\
& \widetilde{\mathcal F}_{nm}^{2\mathfrak b}(\bs r, \bs r_c^{2\mathfrak b})=\frac{(-1)^{n+m}c_ne^{\ri m\varphi_s^{2\mathfrak b}}}{4\pi}I_{nm}^{2\mathfrak b}(\rho_s^{2\mathfrak b}, z_c^{2\mathfrak b}-z),\\
&L_{nm}^{1\mathfrak b}=\frac{(-1)^nc_n^{-1}e^{-\ri m\varphi_t^{1\mathfrak{b}}}}{4\pi }I_{nm}^{1\mathfrak b}(\rho_t^{1\mathfrak b}, z_c^t-z'_{1\mathfrak b}),\\
& L_{nm}^{2\mathfrak b}=\frac{(-1)^mc_n^{-1}e^{-\ri m\varphi_t^{2\mathfrak{b}}}}{4\pi }I_{nm}^{2\mathfrak b}(\rho_t^{2\mathfrak b}, z'_{2\mathfrak b}-z_c^t),\\
& T_{nm,n'm'}^{1\mathfrak b}=\frac{(-1)^{n+m}Q_{nm}^{n'm'}e^{\ri(m'-m)\varphi_{ts}^{1\mathfrak b}}}{4\pi}I_{n+n',m'-m}^{1\mathfrak b}(\rho_{ts}^{1\mathfrak b}, z_c^t- z_c^{1\mathfrak b}),\\
& T_{nm,n'm'}^{2\mathfrak b}=\frac{(-1)^{n+m+n'+m'}Q_{nm}^{n'm'}e^{\ri(m'-m)\varphi_{ts}^{2\mathfrak b}}}{4\pi}I_{n+n',m'-m}^{2\mathfrak b}(\rho_{ts}^{2\mathfrak b},  z_c^{2\mathfrak b}-z_c^t),
\end{split}
\end{equation}
where
$$Q_{nm}^{n'm'}:=\sqrt{\frac{(2n'+1)(n+n'+m'-m)!(n+n'-m'+m)!}{(2n+1)(n+m)!(n-m)!(n'+m')!(n'-m')!}}.$$
Therefore, we actually need efficient algorithm for the computation of the Sommerfeld-type integrals $I_{nm}^{\mathfrak{ab}}(\rho, z)$ defined in \eqref{uniformintegral}. It is clearly that they have oscillatory integrands. These integrals are convergent when the target and
source particles are not exactly on the interfaces of the layered medium.
High order quadrature rules could be used for direct
numerical computation at runtime. However, this becomes prohibitively expensive due to a large number of integrals needed in the FMM. In fact, $(p+1)(2p+1)$ integrals will be required for each source box to target box translation. Moreover, the involved integrand decays more slowly as $n$ increases.

An important aspect in the implementation of FMM concerns scaling. Since $M_{nm}^{\mathfrak{ab}}\approx(|\bs r-\bs r_c^{\mathfrak{ab}}|)^n$, $L_{nm}^{\mathfrak{ab}}\approx(|\bs r^{\mathfrak{ab}}-\bs r_c^t|)^{-n}$, a naive use of the expansions \eqref{melayerupgoingimage} and \eqref{lelayerimage} in the implementation of FMM is likely to encounter underflow and overflow issues. To avoid this, one must scale expansions, replacing $M_{nm}$ with $M_{nm}^{\mathfrak{ab}}/S^n$ and $L_{nm}^{\mathfrak{ab}}$ with $L_{nm}^{\mathfrak{ab}}\cdot S^n$ where $S$ is the scaling factor. To compensate for this scaling, we replace $\widetilde{\mathcal F}_{nm}^{\mathfrak{ab}}(\bs r, \bs r_c^{\mathfrak{ab}})$ with $\widetilde{\mathcal F}_{nm}^{\mathfrak{ab}}(\bs r, \bs r_c^{\mathfrak{ab}})\cdot S^n$, $T_{nm,n'm'}^{\mathfrak{ab}}$ with $T_{nm,n'm'}^{\mathfrak{ab}}\cdot S^{n+n'}$.  Usually, the scaling factor $S$ is chosen to be the size of the box in which the computation occurs. Therefore, the following scaled Sommerfeld-type integrals
\begin{equation}\label{scaledsommerfeldint}
S^{n}I_{nm}^{\mathfrak{ab}}(\rho, z)=\int_0^{\infty}J_{m}(k_{\rho}\rho)\frac{(k_{\rho}S)^{n}e^{-k_{\rho}z}\sigma_{\ell\ell'}^{\mathfrak{ab}}(k_{\rho})}{\sqrt{(n+m)!(n-m)!}}{\rm d}k_{\rho},\;\;n\geq 0,\;\;m=0, 1, \cdots, n,
\end{equation}
are involved in the implementation of the FMM.

Recall the recurrence formula
$$J_{m+1}(z)=\frac{2m}{z}J_{m}(z)-J_{m-1}(z),$$
and define $a_n=\sqrt{n(n+1)}$. We have
\begin{equation*}
\begin{split}
S^{n}I_{nm+1}^{\mathfrak{ab}}(\rho, z)=&\int_0^{\infty}J_{m+1}(k_{\rho}\rho)\frac{(k_{\rho}S)^{n}e^{-k_{\rho}z}\sigma_{\ell\ell'}^{\mathfrak{ab}}(k_{\rho})}{\sqrt{(n+m+1)!(n-m-1)!}}{\rm d}k_{\rho}\\
=&\frac{2m S}{\rho }\int_0^{\infty}J_{m}(k_{\rho}\rho)\frac{(k_{\rho}S)^{n-1}e^{-k_{\rho}z}\sigma_{\ell\ell'}^{\mathfrak{ab}}(k_{\rho})}{\sqrt{(n+m-1)!(n-m-1)!}}\sqrt{\frac{(n+m-1)!}{(n+m+1)!}}{\rm d}k_{\rho}\\
-&\int_0^{\infty}J_{m-1}(k_{\rho}\rho)\frac{(k_{\rho}S)^{n}e^{-k_{\rho}z}\sigma_{\ell\ell'}^{\mathfrak{ab}}(k_{\rho})}{\sqrt{(n+m-1)!(n-m+1)!}}\sqrt{\frac{(n+m-1)!(n-m+1)!}{(n+m+1)!(n-m-1)!}}{\rm d}k_{\rho},
\end{split}
\end{equation*}
which gives the forward recurrence formula
\begin{equation}\label{recurrence}
S^{n}I_{nm+1}^{\mathfrak{ab}}(\rho, z)=\frac{2m}{a_{n+m}}\frac{S}{\rho}S^{n-1}I_{n-1m}^{\mathfrak{ab}}(\rho, z)-\frac{a_{n-m}}{a_{n+m}}S^{n}I_{nm-1}^{\mathfrak{ab}}(\rho, z),
\end{equation}
for $m\geq 1, n\geq m+1$. This recurrence formula is stable if
\begin{equation}
\frac{2m}{a_{n+m}}<\frac{\rho}{S}.
\end{equation}
%Define $S_{\rho}=\rho/S$, the above inequality can rewritten as
%\begin{equation}
%S_{\rho}^{2}n^{2}+(2m+1)S_{\rho}^{2}n+(S_{\rho}^{2}-4)m^{2}+mS_{\rho}^{2}>0.
%\end{equation}
%Solving the inequality for $n$ with fixed $m$, we obtain
%\begin{equation}
%n>2\sqrt{\frac{m^{2}}{S_{\rho}^{2}}+1}-m-\frac{1}{2}.
%\end{equation}
%Conversely, for
%\begin{equation}
%n\leq2\sqrt{\frac{m^{2}}{S_{\rho}^{2}}+1}-m-\frac{1}{2},
%\end{equation}
%the backward recursion
%\begin{equation}\label{backwardrecurrence1}
%S^{n-1}I_{n-1m}^{\mathfrak{ab}}=\frac{a_{n+m}S_{\rho}%
%}{2m} S^{n} I_{nm+1}^{\mathfrak{ab}}+\frac{a_{n-m}S_{\rho}%
%}{2m}S^{n}I_{nm-1}^{\mathfrak{ab}},
%\end{equation}
%will be adopted.

In the computation of $\widetilde{\mathcal F}_{nm}^{\mathfrak{ab}}(\bs r, \bs r_c^{\mathfrak{ab}})\cdot S^n$ and $L_{nm}^{\mathfrak{ab}}\cdot S^n$,  $\rho^{\mathfrak{ab}}_s$ and $\rho_t^{\mathfrak{ab}}$ could be arbitrary small. Therefore, the forward recurrence formula \eqref{recurrence} may not be able to applied to calculate them. Nevertheless, it is unnecessary to calculate $\widetilde{\mathcal F}_{nm}^{\mathfrak{ab}}(\bs r, \bs r_c^{\mathfrak{ab}})\cdot S^n$ and $L_{nm}^{\mathfrak{ab}}\cdot S^n$ directly in the FMM. The coefficients $L_{nm}^{\mathfrak{ab}}\cdot S^n$ are calculated from ME coefficients via M2L translations and then the potentials are obtained via LEs \eqref{lelayerimage}. Therefore, we only need to consider the computation of the integrals involved in the M2L translation matrices $T_{nm,n'm'}^{\mathfrak{ab}}$.
For any polarization source box in the interaction list of a given target box, one can find that $\rho_{ts}^{\mathfrak{ab}}$ is either $0$ or larger than the box size $S$.  If $\rho_{ts}^{\mathfrak{ab}}=0$, we directly have
\begin{equation}
I_{nm}^{\mathfrak{ab}}(\rho_{ts}^{\mathfrak {ab}},z)=0, \quad \forall m>0,\;\;\forall z>0.
\end{equation}
In all other cases, we have $\rho_{ts}^{\mathfrak{ab}}\geq S$ and the forward recurrence formula \eqref{recurrence} can always be applied as we have
$$\frac{2m}{\sqrt{(n+m+1)(n+m)}}<\frac{1}{\sqrt{3}}<\frac{\rho_{ts}^{\mathfrak{ab}}}{S},\quad n\geq m+1,\quad m\geq 1.$$

{\color{black}
Given a truncation number $p$, we still need to use quadratures to calculate $4p+1$ initial values $ \{I_{n 0}^{\mathfrak{ab}}(\rho, z)\}_{n=0}^{2p}$ and $\{ I_{n 1}^{\mathfrak{ab}}(\rho, z)\}_{n=1}^{2p}$ for each M2L translation. Moreover, integrals $\{I_{00}^{\mathfrak{ab}}(\rho, z)\}_{\mathfrak{a,b}=1}^2$ are also required in the computation of the direct interactions between particles in neighboring boxes. These calculations require an efficient and robust numerical method. Note that $\{I_{00}^{\mathfrak{ab}}(\rho, z)\}_{\mathfrak{a,b}=1}^2$ are exactly the Sommerfeld integrals involved in the calculation of the layered Green's function. A multitude of papers have been published until now, devoted to their efficient calculation (see \cite{michalski2016efficient} and the references there in).

Basically, we will adopt the mixed DE-SE quadrature (cf.  \cite{michalski2016efficient,takahasi1974double}) in this paper for efficient computations of the Sommerfeld-type integrals. Nevertheless, we still need to consider the case of large $n$ which has not been covered in the literature. We have found that the formulation \eqref{scaledsommerfeldint} is not adequate for two reasons: (i) the integrand may decay very slowly if $z$ is small; (ii) the integrand may have increasing oscillating magnitude as $n$ increases if $\rho>z$. As a matter of fact, the asymptotic formula \eqref{densityasymptotic} and
\begin{equation*}
J_m(z)\sim \sqrt{\frac{2}{\pi z}}\cos\Big(z-\frac{m\pi}{2}-\frac{\pi}{4}\Big),\quad z\rightarrow\infty,
\end{equation*}
imply that the integrand in \eqref{scaledsommerfeldint} has an asymptotic form
\begin{equation}
J_{m}(k_{\rho}\rho)\frac{(k_{\rho}S)^{n}e^{-k_{\rho}z}\sigma_{\ell\ell'}^{\mathfrak{ab}}(k_{\rho})}{\sqrt{(n+m)!(n-m)!}}\sim \sqrt{\frac{2}{\pi}} C_{\ell\ell'}^{\mathfrak{ab}}\cos\Big(k_{\rho}\rho-\frac{m\pi}{2}-\frac{\pi}{4}\Big)\frac{(k_{\rho}\rho)^{n-\frac{1}{2}}S^ne^{-k_{\rho}(z+\zeta_{\ell\ell'}^{\mathfrak{ab}})}}{\sqrt{(n+m)!(n-m)!}},
\end{equation}
as $k_{\rho}\rightarrow\infty$. Given $\rho, z> 0$, define
\begin{equation}\label{gnmdefinition}
g_{nm}(k_{\rho}; \rho, z+\zeta_{\ell\ell'}^{\mathfrak{ab}})=\frac{(k_{\rho}\rho)^{n-\frac{1}{2}}S^ne^{-k_{\rho}(z+\zeta_{\ell\ell'}^{\mathfrak{ab}})}}{\sqrt{(n+m)!(n-m)!}},
\end{equation}
which has a maximum value
\begin{equation}
\max\limits_{k_{\rho}\geq 0}g_{nm}(k_{\rho}; \rho, z+\zeta_{\ell\ell'}^{\mathfrak{ab}})=\frac{S^n}{\sqrt{(n+m)!(n-m)!}}\Big(\frac{2n-1}{2}\Big)^{n-\frac{1}{2}}\Big(\frac{\rho}{z+\zeta_{\ell\ell'}^{\mathfrak{ab}}}\Big)^{n-\frac{1}{2}}e^{\frac{1}{2}-n},
\end{equation}
at $k_{\rho}=\frac{n}{z+\zeta_{\ell\ell'}^{\mathfrak{ab}}}-\frac{1}{2(z+\zeta_{\ell\ell'}^{\mathfrak{ab}})}$ for $n\geq 1$. Applying Stirling formula $n!\sim \sqrt{2\pi n}n^n/e^n$ yields
\begin{equation}\label{miximumvalue}
\max\limits_{k_{\rho}\geq 0}g_{nm}(k_{\rho}; \rho, z+\zeta_{\ell\ell'}^{\mathfrak{ab}})\sim \sqrt{\frac{(2n-1)e}{2}}\frac{n!}{\sqrt{(n+m)!(n-m)!}}\Big(\frac{\rho}{z+\zeta_{\ell\ell'}^{\mathfrak{ab}}}\Big)^{n-\frac{1}{2}}S^n.
\end{equation}
Considering the case $m=0$ and setting $S=\sqrt{\rho^2+z^2}$, we have
\begin{equation}
\begin{split}
\max\limits_{k_{\rho}\geq 0}g_{n0}(k_{\rho}, \rho, z+\zeta_{\ell\ell'}^{\mathfrak{ab}})&\sim \sqrt{\frac{(2n-1)(z+\zeta_{\ell\ell'}^{\mathfrak{ab}})e}{2\rho}}\Big(\frac{\rho S}{z+\zeta_{\ell\ell'}^{\mathfrak{ab}}}\Big)^{n}\\
&\geq\sqrt{\frac{(2n-1)(z+\zeta_{\ell\ell'}^{\mathfrak{ab}})e}{2\rho}}\Big(\frac{\rho^2}{z+\zeta_{\ell\ell'}^{\mathfrak{ab}}}\Big)^{n}, \quad {\rm if}\;\;\rho>z+\zeta_{\ell\ell'}^{\mathfrak{ab}}.
\end{split}
\end{equation}
From the above estimate, we can see that the formulation \eqref{scaledsommerfeldint} have very large cancellations in the integrand if $\rho/(z+\zeta_{\ell\ell'}^{\mathfrak{ab}})$ and $n$ are large, see Fig. \ref{integrand} (a) for an example. Therefore, simply applying a quadrature along the real axis will not be efficient.
\begin{figure}[ptbh]
	\center
	\subfigure[along real axis]{\includegraphics[scale=0.37]{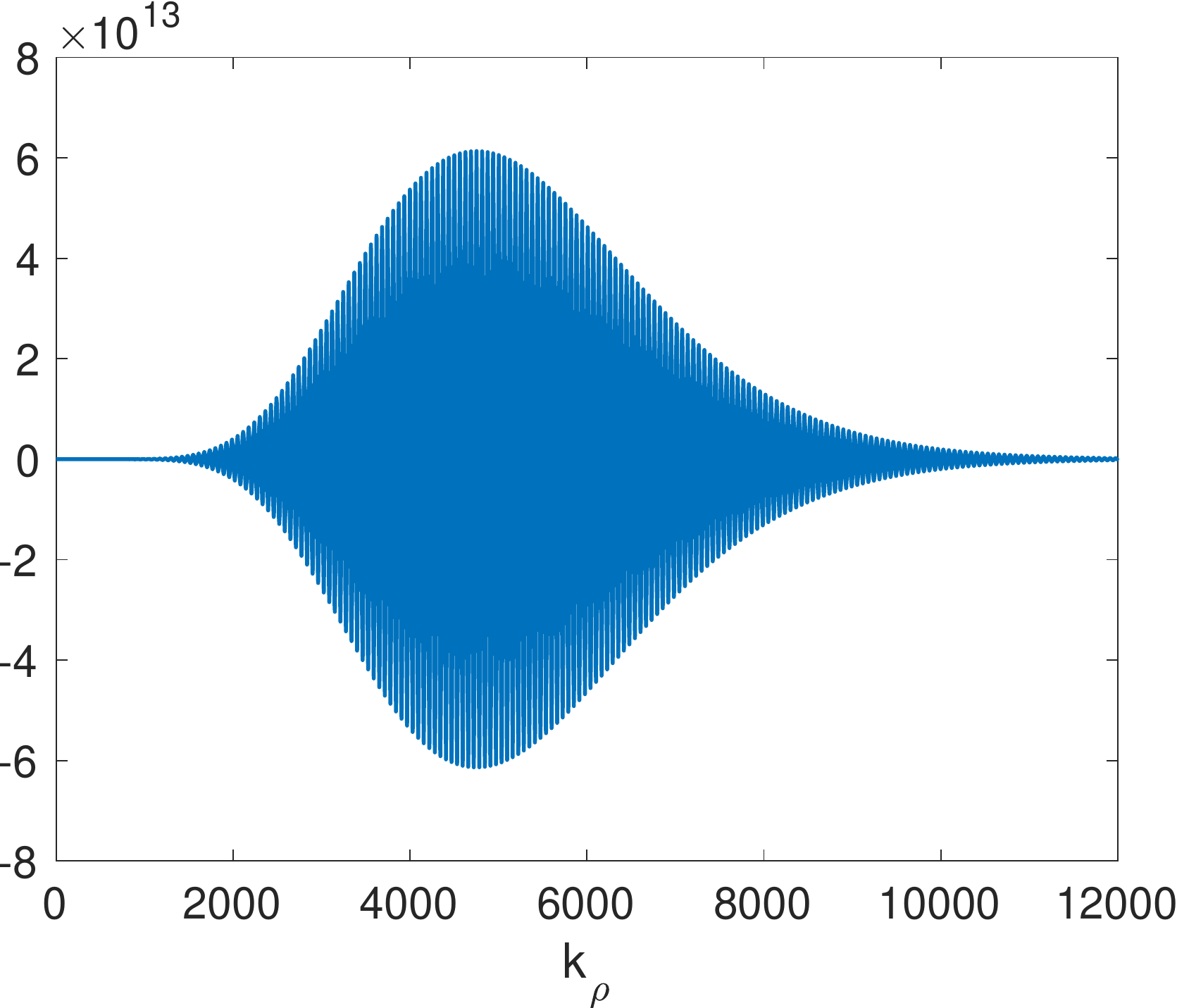}}\quad
	\subfigure[along imaginary axis]{\includegraphics[scale=0.37]{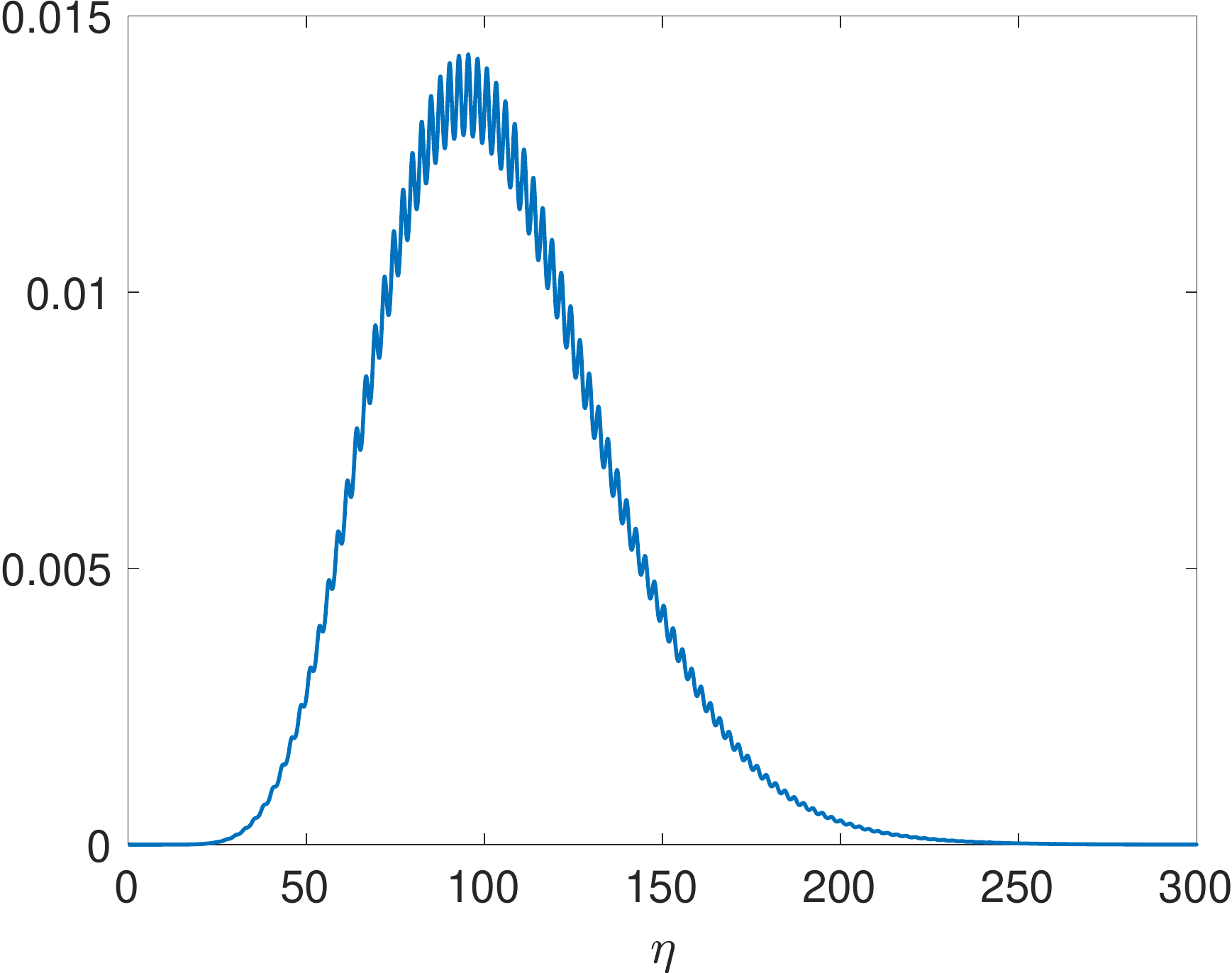}}
	\caption{A comparison of the integrands in \eqref{scaledsommerfeldint} and \eqref{contourchange} with $n=10, m=0$, $\rho=0.1$, $z=0.002$ and $\sigma_{11}^{11}(k_{\rho})$ given in \eqref{densitythreelayer2} ( $d_0=0$, $d_1=-1.2$, $\varepsilon_0=21.2$, $\varepsilon_1=47.5$, $\varepsilon_2=62.8$).}%
	\label{integrand}%
\end{figure}

To handle the case $\rho>(z+\zeta_{\ell\ell'}^{\mathfrak{ab}})$, we change the contour to the imaginary axis as follows. We first reformulate the integral \eqref{scaledsommerfeldint} as
\begin{equation}
\begin{split}
S^nI_{nm}^{\mathfrak{ab}}(\rho, z)=&\frac{1}{2}\int_{0}^{\infty}H_{m}^{(1)}(k_{\rho}\rho)\frac{(Sk_{\rho})^{n}e^{-k_{\rho}z}}{\sqrt{(n+m)!(n-m)!}}\sigma_{\ell\ell'}^{\mathfrak{ab}}(k_{\rho}) dk_{\rho}\\
+&\frac{(-1)^{m+1}}{2}\int_{-\infty}^{0}H_{m}^{(1)}(k_{\rho}\rho)\frac{(-Sk_{\rho})^{n}e^{k_{\rho}z}}{\sqrt{(n+m)!(n-m)!}}\sigma_{\ell\ell'}^{\mathfrak{ab}}(-k_{\rho}) dk_{\rho},
\end{split}
\end{equation}
by using identities
\begin{equation}
J_m(z)=\frac{H_m^{(1)}(z)+H_m^{(2)}(z)}{2},\quad H_m^{(2)}(-x)=(-1)^{m+1}H_m^{(1)}(x).
\end{equation}
As the density function $\sigma_{\ell\ell'}^{\mathfrak{ab}}(k_{\rho})$ is analytic in the right half complex plane, we can change the contour from the real axis to the one which wraps the positive imaginary axis to obtain
\begin{equation}\label{uniformintegralalongimax}
\begin{split}
S^nI_{nm}^{\mathfrak{ab}}(\rho, z)=&\frac{\ri}{2}\int_{0}^{\infty}H_{m}^{(1)}(\ri \eta\rho)\frac{(\ri\eta S)^{n}e^{-\ri\eta z}}{\sqrt{(n+m)!(n-m)!}}\sigma_{\ell\ell'}^{\mathfrak{ab}}(\ri\eta) d\eta\\
-&\frac{\ri}{2}\int_{0}^{\infty}H_{m}^{(1)}(\ri\eta\rho)\frac{(-1)^{m+1}(-\ri\eta S)^{n}e^{\ri\eta z}}{\sqrt{(n+m)!(n-m)!}}\sigma_{\ell\ell'}^{\mathfrak{ab}}(-\ri\eta) d\eta.
\end{split}
\end{equation}
Then, a substitution of the identity (cf. \cite[Eq. (10.27.8)]{Olver2010})
\begin{equation}
H_m^{(1)}(\ri z)=\frac{2\ri^{-m-1}}{\pi}K_m(z),\quad -\pi\leq\arg z\leq\frac{\pi}{2},
\end{equation}
into \eqref{uniformintegralalongimax} gives
\begin{equation}\label{contourchange}
S^nI_{nm}^{\mathfrak{ab}}(\rho, z)=\frac{\ri^{n-m}}{\pi}\int_{0}^{\infty}K_m(\eta\rho)(S\eta)^n\frac{e^{-\ri\eta z}\sigma_{\ell\ell'}^{\mathfrak{ab}}(\ri\eta)+(-1)^{n+m}e^{\ri\eta z}\sigma_{\ell\ell'}^{\mathfrak{ab}}(-\ri\eta)}{\sqrt{(n+m)!(n-m)!}} d\eta.
\end{equation}
According to the expressions given in \eqref{densitythreelayer1}-\eqref{densitythreelayer3}, all decaying terms in $\sigma_{\ell\ell'}^{\mathfrak{ab}}(k_{\rho})$ become bounded oscillating terms in $\sigma_{\ell\ell'}^{\mathfrak{ab}}(\pm \ri\eta)$.  By the asymptotic formulation \cite[Eq. (10.25.3)]{Olver2010}:
\begin{equation*}
K_m(z)\sim \sqrt{\frac{\pi}{2z}}e^{-z},\quad z\rightarrow\infty, \quad |\arg z|<\frac{3\pi}{2},
\end{equation*}
and the definition of $g_{nm}(k_{\rho}; \rho,z)$ in \eqref{gnmdefinition}, the main part of the integrand has an asymptotic expression
\begin{equation}
\frac{K_m(\eta\rho)(S\eta)^n}{\sqrt{(n+m)!(n-m)!}}\sim \sqrt{\frac{\pi}{2\eta\rho}}\frac{(S\eta)^ne^{-\eta\rho}}{\sqrt{(n+m)!(n-m)!}}= \sqrt{\frac{\pi}{2\rho}}g_{nm}(\eta; 1, \rho),\quad\eta\rightarrow\infty.
\end{equation}
Recalling \eqref{miximumvalue} to get
\begin{equation}
\max\limits_{\eta\geq 0}g_{nm}(\eta; 1, \rho)\sim\sqrt{\frac{(2n-1)e}{2}}\frac{n!}{\sqrt{(n+m)!(n-m)!}}\Big(\frac{1}{\rho}\Big)^{n-\frac{1}{2}}S^n,\quad \eta\rightarrow\infty.
\end{equation}
As an example, we consider the case $m=0$ and set $S=\sqrt{\rho^2+z^2}$ again, i.e.,
\begin{equation}
\max\limits_{\eta\geq 0}g_{n0}(\eta; 1, \rho)\sim\sqrt{\frac{(2n-1)e}{2}}\Big(\frac{1}{\rho}\Big)^{n-\frac{1}{2}}S^n= \sqrt{\frac{(2n-1)e}{2}}\Big(1+\frac{z^2}{\rho^2}\Big)^{\frac{n}{2}}\sqrt{\rho}.
\end{equation}
Apparently, the large cancellation in the case $\rho>z+\zeta_{\ell\ell'}^{\mathfrak{ab}}$ can be significantly suppressed by using the formulation \eqref{contourchange}. At the same time, the oscillating term $J_m(k_{\rho}\rho)$ is turned to be exponential decaying function $K_m(\eta\rho)$ and thus produce much fast decay when $\rho/(z+\zeta_{\ell\ell'}^{\mathfrak{ab}})$ is large.  A comparison of the integrands along real and imaginary axises is plotted in Fig. \ref{integrand}.

To end this section, we will give some numerical results to show the accuracy and efficiency of the algorithm using mixed DE-SE quadrature together with formulations \eqref{scaledsommerfeldint} and \eqref{contourchange} for the computation of the Sommerfeld type integrals. We test the integral with densities $\sigma_{\ell\ell'}^{\mathfrak{ab}}(k_{\rho})\equiv 1$ as the asymptotic formula \eqref{densityasymptotic} implies that $\sigma_{\ell\ell'}^{\mathfrak{ab}}(k_{\rho})$ tends to be either the constant $C_{\ell\ell'}^{\mathfrak{ab}}$ or $0$ rapidly as $k_{\rho}\rightarrow\infty$. Letting $S=r:=\sqrt{\rho^2+z^2}$, then the identity \eqref{wavefunspectralform} yields
\begin{equation}
r^nI_{nm}^{\mathfrak{ab}}(\rho, z)=\sqrt{\frac{4\pi}{2n+1}}\frac{1}{r}\widehat{P}_{n}^m\Big(\frac{z}{r}\Big).
\end{equation}
We fix $z=0.001$ and test $\rho=0.0005, 0.01, 0.1$ by using two different quadratures: (i) the composite Gaussian quadrature applied to the integral \eqref{scaledsommerfeldint}; (ii) the mixed DE-SE quadrature applied to \eqref{scaledsommerfeldint} and \eqref{contourchange} for $\rho\leq z$ and $\rho>z$, respectively. For the composite
\begin{table}[ht!]
\color{black}{
	\centering {\small
		\begin{tabular}
			[c]{|c|c|c|c|c|c|c|}\hline
			\multirow{2}{*}{$\rho$} & \multirow{2}{*}{$n$}& \multirow{2}{*}{$m$} & \multicolumn{2}{c|}{Composite Gauss} & \multicolumn{2}{c|}{Mixed DE-SE} \\\cline{4-7}
			& & &  number of points & error & number of points & error \\\hline
			\multirow{4}{*}{0.0005}&\multirow{2}{*}{5} &
			0 & 717523 & -3.307e-12 & 80 & 3.819e-14\\\cline{3-7}
			& &1 & 716016 & 2.576e-11 &
			80 & 5.684e-14\\\cline{2-7}
			& \multirow{2}{*}{10} &0 & 892278 & 6.954e-12 & 72 & -2.842e-14\\\cline{3-7}
			& &1 & 891431 & 1.882e-11 & 72 & -9.059e-14\\\hline
			\multirow{4}{*}{0.01}&\multirow{2}{*}{5} &
			0 & 872989 & -1.427e-10 & 56 & 4.441e-16\\\cline{3-7}
			& &1 & 871511 &-2.716e-11  &
			64 & 3.108e-15\\\cline{2-7}
			& \multirow{2}{*}{10} &0 & 1246898 & 1.147e-5 &
			56 & -1.443e-15\\\cline{3-7}
			& &1 & 1246090 & -6.755e-6 &
			56 & 6.883e-15\\\hline
			\multirow{4}{*}{0.1}&\multirow{2}{*}{5} &
			0 & 1039851 & -8.793e-7 & 48 & -3.078e-12\\\cline{3-7}
			& &1 & 1038393 & -9.250e-7 &
			56 & 4.852e-11\\\cline{2-7}
			& \multirow{2}{*}{10} &0 & 1610764  & -10615.95 &
			48 & 1.943e-16\\\cline{3-7}
			& &1 & 1609974 & 1334.402 & 48 &2.775e-17\\\hline
		\end{tabular}
	}\caption{A comparison of two quadrature rules for the computation of Sommerfeld integrals with $z=0.001$.}%
	\label{Table:gpofapproximationrp}%
}
\end{table}
Gaussian quadrature, the asymptotic formula \eqref{gnmdefinition} is used to determine the truncation points such that the magnitude of the integrand decays to smaller than $1.0e-15$. Then, a uniform mesh with mesh size equal to $2$ and $30$ Gauss points in each interval is used to achieve machine accuracy in regular case. Due to the small value of $z$, a very large truncation is needed if the formulation \eqref{scaledsommerfeldint} is used. The results are compared in Table. \ref{Table:gpofapproximationrp}. We can see that the truncation is larger than 47834 in the case $\rho=0.0005$, $n=5$ and $m=0,1$. The truncation in all other tested cases is even larger. Thus, a large number of quadrature points have been used to achieve good accuracy if the composite Gauss quadrature is applied to \eqref{scaledsommerfeldint}. In contrast, the mixed DE-SE quadrature can obtain machine accuracy using no more than 100 points. Moreover, as the ratio $\rho/z$ increases, applying composite Gauss quadrature to \eqref{scaledsommerfeldint} can not give correct values due to the large cancellation in \eqref{scaledsommerfeldint}. Instead, the mixed DE-SE quadrature applied to \eqref{contourchange} can provide results with machine accuracy using a few quadrature points.}

\begin{rem}
	Apparently, the technique of using pre-computed tables together with polynomial interpolation can still be applied for efficient computation of the initial values $ \{I_{n 0}^{\mathfrak{ab}}(\rho, z)\}_{n=0}^{2p}$ and $\{ I_{n 1}^{\mathfrak{ab}}(\rho, z)\}_{n=1}^{2p}$ at run time. Then, $4p+1$ tables need to be pre-computed on the 2-D grid in a domain of interest. Efficient improvement by using pre-computed tables is validated by some numerical tests in next section.
\end{rem}

%In order to efficiently compute the direct interactions between particles in neighboring boxes, we use pre-computed tables for $I_{00}^{\mathfrak{ab}}(\rho, z)$. Since only interactions between charges in neighboring leaf boxes require direct computation, we only need pre-computed tables for $I_{00}^{\mathfrak{ab}}(\rho, z)$ in a small 2-D domain $[0, 3S_{\rm min}]\times[2d_{\rm min}, 3S_{\rm min}]$ where $S_{\rm min}$ is the size of the leaf boxes in the tree structure, $d_{\rm min}$ is the minimum distance between charges and corresponding interface. Note that $I_{00}^{\mathfrak{ab}}(\rho, z)$ is a smooth function in the domain of interest.

%As an example, we plot the density $\sigma_{11}^{\mathfrak{ab}}(k_{\rho})$ along the positive real line (see, Fig. \ref{deformedcontour} (c)). The three layers case with $k_0=21.2$, $k_1=47.5$, $k_2=62.8$,
%$d_0=0, d_1=-1.2$ and densities given
%in \eqref{densitythreelayer2} are used. We can see that the density $\{\sigma_{11}^{\mathfrak{ab}}(k_{\rho})\}_{\mathfrak{a,b}=1,2}$ are bounded and smooth along the real line.

\section{Numerical results}

In this section, we present numerical results to demonstrate the performance
of the proposed FMM. The algorithm is implemented based on an open-source adaptive FMM package DASHMM (cf. \cite{debuhr2016dashmm}) on a workstation with two Xeon E5-2699 v4 2.2 GHz processors (each has 22 cores) and 500GB RAM using the gcc compiler version 6.3.

We test problems in a three layers medium with interfaces placed at $z_{0}=0$, $z_{1}=-1.2$. Charges
are set to be uniformly distributed in irregular domains which are obtained by shifting the domain determined by $r=0.5-a+\frac{a}{8}(35\cos^{4}\theta-30\cos^{2}\theta+3)$
%\begin{equation}
%r=0.5-a+\frac{a}{8}(35\cos^{4}\theta-30\cos^{2}\theta+3),
%\end{equation}
with $a=0.1,0.15,0.05$ to new centers $(0,0,0.6)$, $(0,0,-0.6)$
and $(0,0,-1.8)$, respectively (see
Fig. \ref{fmmconvergence} (a) for the cross section of the domains). All particles are generated by keeping the
uniform distributed charges in a larger cube within corresponding irregular
domains. In the layered medium, the material parameters are set to be $\varepsilon_0=21.2$, $\varepsilon_1=47.5$, $\varepsilon_2=62.8$. Let $\widetilde{\Phi}_{\ell}%
(\boldsymbol{r}_{\ell i})$ be the approximated values of $\Phi_{\ell
}(\boldsymbol{r}_{\ell i})$ calculated by FMM. Define $\ell^{2}$ and maximum errors as
\begin{equation}
Err_{2}^{\ell}:=\sqrt{\frac{\sum\limits_{i=1}^{N_{\ell}}|\Phi_{\ell
		}(\boldsymbol{r}_{\ell i})-\widetilde{\Phi}_{\ell}(\boldsymbol{r}_{\ell
			i})|^{2}}{\sum\limits_{i=1}^{N_{\ell}}|\Phi_{\ell}(\boldsymbol{r}_{\ell
			i})|^{2}}},\qquad Err_{max}^{\ell}:=\max\limits_{1\leq i\leq{N_{\ell}}}%
\frac{|\Phi_{\ell}(\boldsymbol{r}_{\ell i})-\widetilde{\Phi}_{\ell
	}(\boldsymbol{r}_{\ell i})|}{|\Phi_{\ell}(\boldsymbol{r}_{\ell i})|}.
\end{equation}
\begin{figure}[ptbh]
	\center
	\subfigure[distribution of charges]{\includegraphics[scale=0.335]{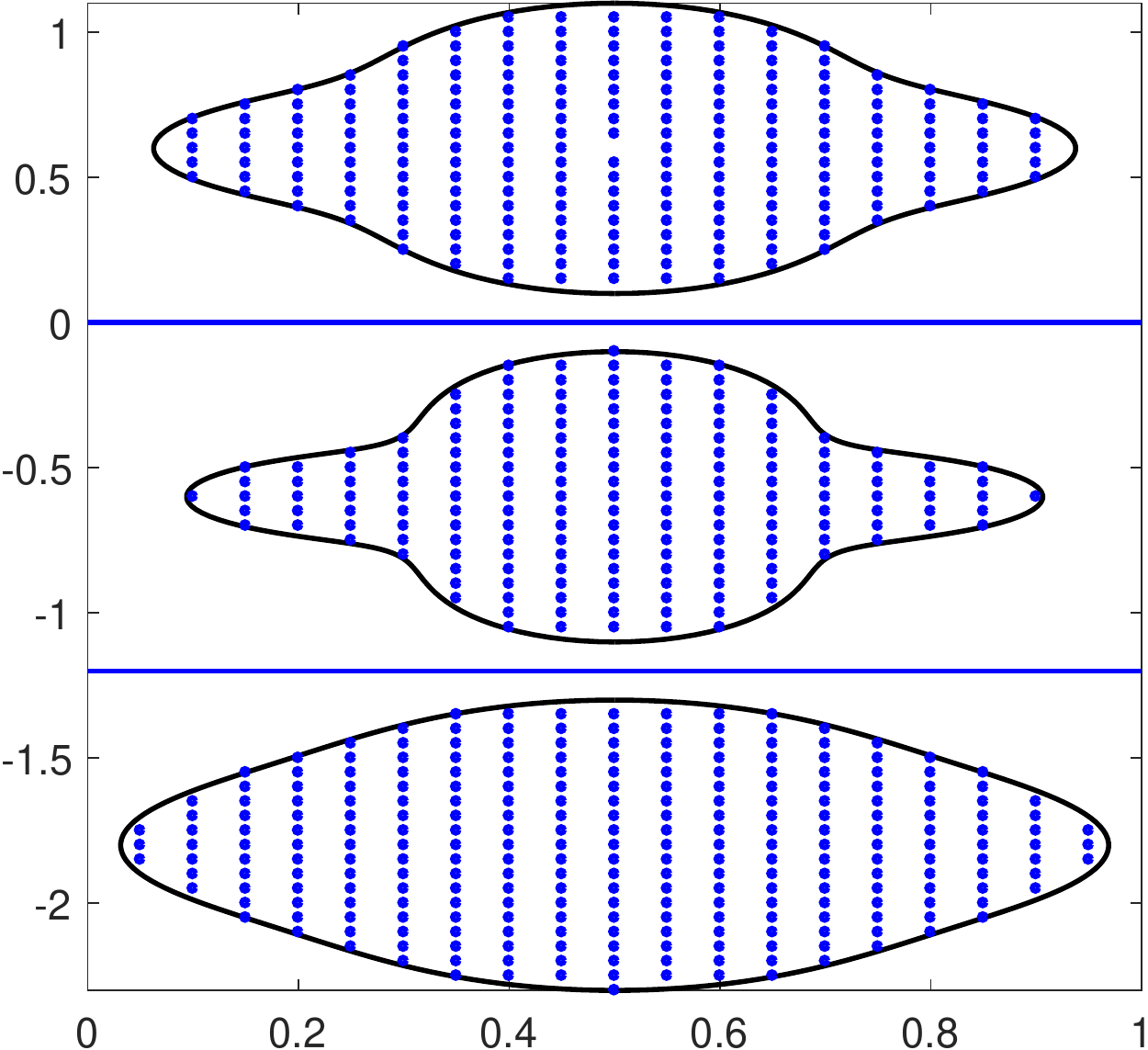}}
	\subfigure[convergence rates vs. $p$]{\includegraphics[scale=0.29]{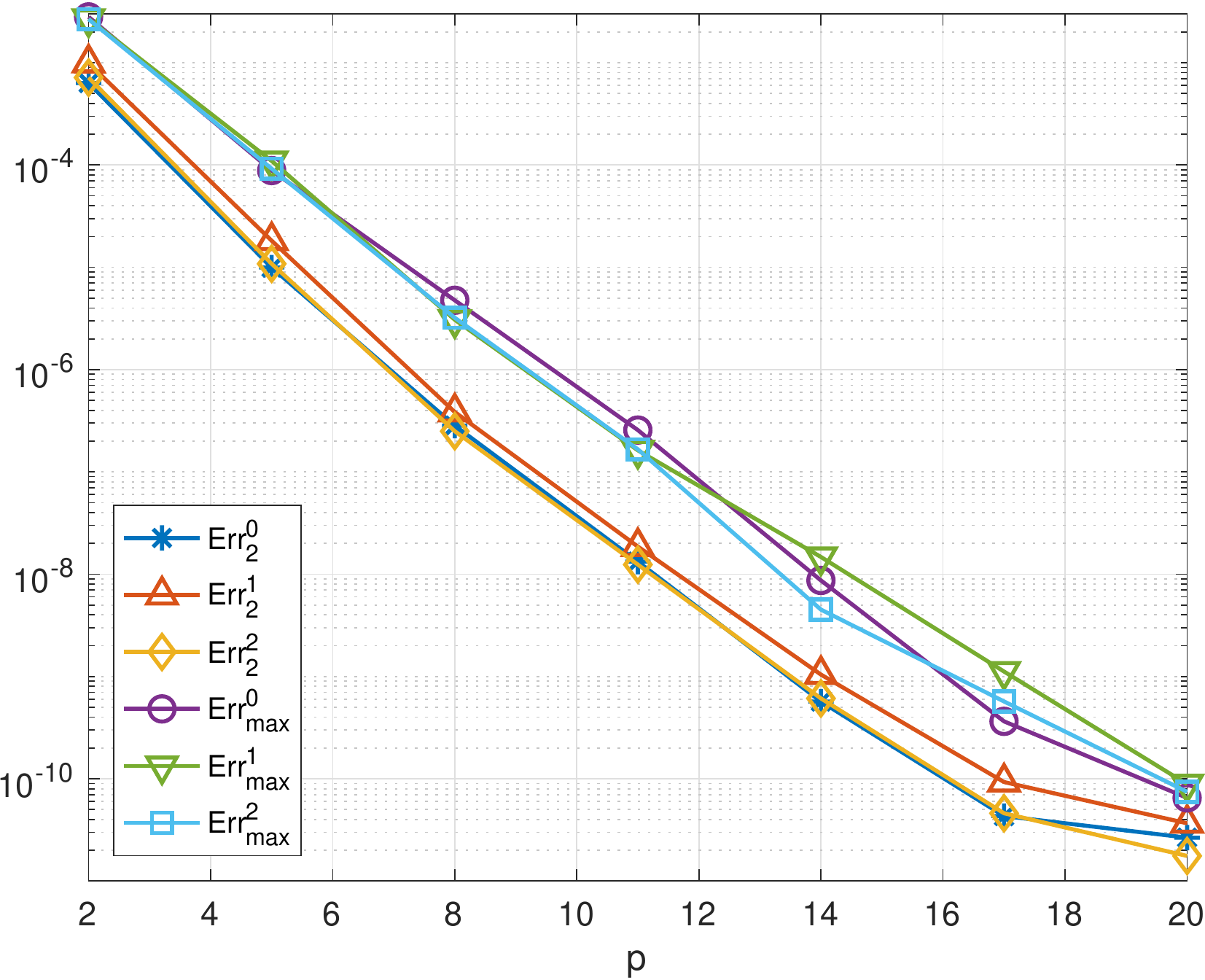}}
	\subfigure[ CPU time (sec) vs. $N$]{\includegraphics[scale=0.25]{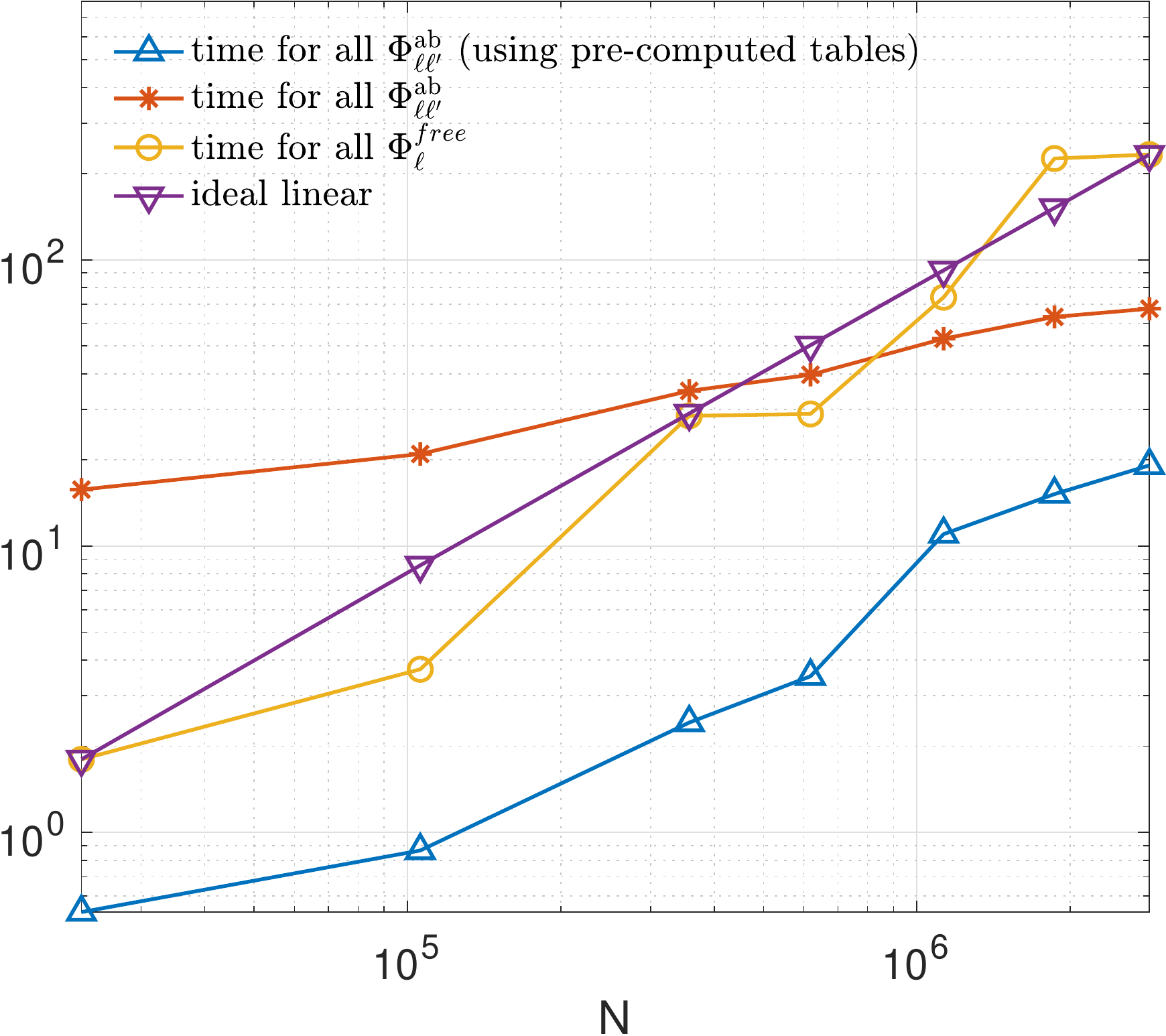}}
	\caption{Performance of FMM for problem in a three layers medium.}%
	\label{fmmconvergence}%
\end{figure}
For accuracy test, we put
$N=912+640+1296$ charges in the irregular domains in three layers, see Fig. \ref{fmmconvergence} (a). Convergence rates against $p$ are depicted in Fig. \ref{fmmconvergence} (b). Clearly, the proposed FMM has spectral convergence with respect to truncation number $p$.
The CPU time for the computation of all three free space components $\{\Phi^{free}_{\ell}(\boldsymbol{r}_{\ell i})\}_{\ell=0}^2$ and sixteen reaction components $\Phi^{\mathfrak{ab}}_{\ell\ell'}(\boldsymbol{r}_{\ell i})$ with fixed truncation number $p=5$ are compared in Fig. \ref{fmmconvergence} (c) for up to 3 millions charges. It shows that all of them have an $O(N)$ complexity while the CPU time for the computation of reaction components has a much smaller linear scaling constant due to the fact that most of the equivalent polarization sources are well-separated with the targets.
%Moreover, the CPU time complexities with respect to truncation number $p$ (fixed number of charges $N=34178+23851+47919$) are compared in Fig. \ref{fmmperformance1} (b) . We can see that the FMMs for free space and reaction components have $O(p^2)$ complexity.
CPU time with multiple cores is given in Table \ref{Table:exthree} and they show that the speedup of the parallel computing for reaction components is little bit lower than that for the free space components.  Here, we only use parallel implementation within the computation of each component. Note that the computation of each component is independent of all other components. Therefore, it is straightforward to implement a version of the code which computes all components in parallel.

%\begin{figure}[ptbh]
%	\center
%	\subfigure[]{\includegraphics[scale=0.35]{fmmperformance1}}\quad
%	\subfigure[]{\includegraphics[scale=0.35]{fmmperformance2}}
%	\caption{(a) CPU time (sec) vs. number of particles $N$ with $p=5$ fixed; (b) CPU time (sec) vs. truncation number $p$ with $N=34178+23851+47919$ fixed.}%
%	\label{fmmperformance1}%
%\end{figure}
\begin{table}[ptbhptbhptbhptbh]
	\centering {\small
		\begin{tabular}
			[c]{|c|c|c|c|c|}\hline
			\multirow{2}{*}{cores} & \multirow{2}{*}{$N$} & \multirow{2}{*}{time for all $\{\Phi_{\ell}^{free}\}_{\ell=0}^2$} & \multicolumn{2}{c|}{time for all $\{\Phi^{\mathfrak{ab}}_{\ell\ell'}\}$ }\\\cline{4-5}
			& & & 	not use pre-computed tables	& use pre-computed tables	\\\hline
			\multirow{4}{*}{1} &618256 & 28.89 & 39.61  & 3.51 \\\cline{2-5}
			& 1128556 & 73.16 & 54.86 & 11.01  \\\cline{2-5}
			& 1862568 & 223.15 & 63.15 & 15.19 \\\cline{2-5}
			& 2861288 & 237.45 & 69.70 & 19.14 \\\hline
			\multirow{4}{*}{6} & 618256 & 5.57 & 8.13 & 1.22\\\cline{2-5}
			& 1128556 & 13.92 &11.31  &3.53 \\\cline{2-5}
			& 1862568 & 42.07 & 13.81 & 5.18 \\\cline{2-5}
			& 2861288 & 45.06 & 15.42 & 6.33 \\\hline
			\multirow{4}{*}{36} & 618256 & 1.52 & 3.67  & 1.21\\\cline{2-5}
			& 1128556 & 3.52 & 5.56  & 2.60\\\cline{2-5}
			& 1862568 & 10.59 & 7.86  & 3.57\\\cline{2-5}
			& 2861288 & 11.22 & 9.63  & 4.85 \\\hline
		\end{tabular}
	}
	\caption{Comparison of CPU time (sec) using multiple cores ($p=5$).}%
	\label{Table:exthree}%
\end{table}
%
%\begin{table}[ptbhptbhptbhptbh]
%	\centering {\small
%		\begin{tabular}
%			[c]{|c|c|c|c|}\hline
%			cores & $N$ & time for all $\{\Phi_{\ell}^{free}\}_{\ell=0}^2$ & time for all $\{\Phi^{\mathfrak{ab}}_{\ell\ell'}\}$ 			\\\hline
%			\multirow{4}{*}{1} &618256 & 28.90 & 3.51  \\\cline{2-4}
%			& 1128556 & 73.20 & 11.01  \\\cline{2-4}
%			& 1862568 & 224.6 & 15.19 \\\cline{2-4}
%			& 2861288 & 233.8 & 19.14 \\\hline
%			\multirow{4}{*}{6} & 618256 & 5.564 & 1.22 \\\cline{2-4}
%			& 1128556 & 13.87 &3.53  \\\cline{2-4}
%			& 1862568 & 41.98 & 5.18  \\\cline{2-4}
%			& 2861288 & 43.95 & 6.33  \\\hline
%			\multirow{4}{*}{36} & 618256 & 1.513 & 1.21  \\\cline{2-4}
%			& 1128556 & 3.442 & 2.60  \\\cline{2-4}
%			& 1862568 & 10.75 & 3.57  \\\cline{2-4}
%			& 2861288 & 11.07 & 4.85  \\\hline
%		\end{tabular}
%	}
%	\vspace{5pt}
%	\caption{Comparison of CPU time (sec) with multiple cores ($p=5$)  where pre-computed tables are used for efficient computation of the initial values of the recursion \eqref{recurrence}.}%
%	\label{Table:ex2three}%
%\end{table}

\section{Conclusion}

In this paper, we have presented a fast multipole
method for the efficient calculation of the interactions between charged particles embedded in 3-D layered media. The layered media Green's function of the Laplace equation is decomposed into a free space and four types of reaction components.
%Based on the fact that the reaction components rely on the interface dependent local coordinates instead of the physical ones,
The associated equivalent polarization sources are introduced to re-express the reaction components. New MEs and LEs of $O(p^{2})$ terms for the far field of the reaction components  and M2L translation operators are derived, accordingly. As a result, the traditional FMM framework can be applied to both the free space and reaction components once the polarization sources are used together with the original sources.
%Due to the separation of the polarization sources and the corresponding targets by a material interface,
The computational cost from the reaction component is only a fraction of that of the FMM for the free space component if a sufficient large number of charges are presented in the problem.  Therefore, computing the interactions of many sources in layered media basically costs the same as that for the interactions in the free space.

For the future work, we will carry out error estimate of the FMM for the Laplace equation in 3-D
layered media, which requires an error analysis for the new MEs and M2L operators for the reaction components. The application of the FMM in capacitance extraction of interconnects in VLSI will also be considered in a future work.

{\color{black}
\begin{appendix}
\section{A stable recursive algorithm for computing reaction densities} \label{spherehar}
\renewcommand{\theequation}{A.\arabic{equation}}
\setcounter{equation}{0}
Denote the solution of the problem \eqref{Laplaceeqlayer}-\eqref{transmissioncond} in the $\ell$-th layer by $u_{\ell\ell'}(\bs r, \bs r')$ and its partial Fourier transform along $x-$ and $y-$directions by
\[
\widehat{u}_{\ell\ell^{\prime}}(k_{x},k_{y}%
,z)=\mathscr{F}[u_{\ell\ell'}(\bs r, \bs r')](k_{x},k_{y},z):=\int_{-\infty
}^{\infty}\int_{-\infty}^{\infty}u_{\ell\ell^{\prime}}%
(\boldsymbol{r},\boldsymbol{r}^{\prime})e^{-\ri(k_{x}x+k_{y}y)}dxdy.
\]
Then, $\widehat{u}_{\ell\ell^{\prime}}(k_{x},k_{y},z)$ satisfies second order
ordinary differential equation
\begin{equation}
\frac{d^2\widehat u_{\ell\ell'}(k_x,k_y, z)}{dz^2}-k_{\rho}^2\widehat u_{\ell\ell'}(k_x,k_y, z)=-e^{-\ri(k_xx'+k_yy')}\delta(z, z'), \quad z\neq d_{\ell}.\label{freqdomhelm}
\end{equation}
Here, we consider the following general interface conditions
\begin{equation}
\label{generalinterfacecond}
a_{\ell-1}\widehat{u}_{\ell-1,\ell^{\prime}}(k_{x},k_{y},z)=a_{\ell}\widehat{u}_{\ell
	\ell^{\prime}}(k_{x},k_{y},z),\quad b_{\ell-1}\frac{d\widehat{u}_{\ell
		-1,\ell^{\prime}}(k_{x},k_{y},z)}{dz}=b_{\ell}\frac{d\widehat{u}_{\ell
		\ell^{\prime}}(k_{x},k_{y},z)}{dz},
\end{equation}
in the frequency domain for $\ell=1, 2, \cdots, L$, where $\{a_{\ell}, b_{\ell}\}$ are given constants. Apparently, the classic transmission condition \eqref{transmissioncond} will lead to a special case of \eqref{generalinterfacecond} with $a_{\ell}=1$, $b_{\ell}=\varepsilon_{\ell}$. In the top and bottom-most layers, we also have decaying condition
\begin{equation}
\widehat u_{0\ell'}(k_x,k_y, z)\rightarrow 0,\quad \widehat u_{L\ell'}(k_x,k_y, z)\rightarrow 0, \quad{\rm as}\;\; z\rightarrow\pm\infty.
\end{equation}
This interface problem has a general solution
\begin{equation}%
\begin{cases}
\displaystyle\widehat{u}_{0\ell^{\prime}}(k_{x},k_{y},z)=\sigma^{1}_{0
	\ell^{\prime}}e^{-k_{\rho}(z-d_0)},\\[10pt]%
\displaystyle\widehat{u}_{\ell^{\prime}\ell^{\prime}}(k_{x},k_{y}%
,z)=\sigma^{1}_{\ell^{\prime}\ell^{\prime}}e^{-k_{\rho}(z-d_{\ell'})}+\sigma^{2}_{\ell^{\prime}\ell^{\prime}}e^{-k_{\rho}(d_{\ell'-1}-z)}+\delta_{\ell\ell'}\widehat{G}(k_x, k_y, z, z^{\prime}),\\[10pt]%
\displaystyle\widehat{u}_{L\ell^{\prime}}(k_{x},k_{y},z)=\sigma^{2}_{L\ell^{\prime}}e^{-k_{\rho}(d_{L-1}-z)},
\end{cases}
\label{gensolutionformulafreq}%
\end{equation}
where $\delta_{\ell\ell'}$ is the kronecker symbol, and
\begin{equation}
\widehat{G}(k_x, k_y, z, z^{\prime})=\vartheta e^{-k_{\rho}|z-z^{\prime}|},\quad \vartheta=\frac{ e^{-\ri(k_{x}x^{\prime}+k_{y}y^{\prime})}}{2k_{\rho}},
\end{equation}
is the Fourier transform of the free space Green's function. We will use the decomposition
\begin{equation}
\widehat{G}(k_x, k_y, z, z^{\prime})=\widehat{G}^{1}(k_x, k_y, z, z^{\prime})+\widehat{G}^{2}(k_x, k_y, z, z^{\prime}),
\end{equation}
where the two components are defined as
\begin{equation}\label{freegreenfreq}
\widehat{G}^{1}(k_x, k_y, z, z^{\prime}):=H(z'-z)\vartheta e^{-k_{\rho}(z^{\prime}-z)},\quad\widehat{G}^{2}(k_x, k_y, z, z^{\prime}):=H(z-z')\vartheta e^{-k_{\rho}(z-z^{\prime})},
\end{equation}
and $H(x)$ is the Heaviside function.

We first consider the $\ell$-th layer without source ($\ell\neq\ell'$), where the right hand side of \eqref{freqdomhelm} becomes zero, the solution is given by
\begin{equation}
\begin{split}
\widehat u_{\ell\ell'}&(k_x,k_y, z)=\sigma^{1}_{\ell\ell^{\prime}}(k_x, k_y)e^{-k_{\rho}(z-d_{\ell})}+\sigma^{2}_{\ell\ell^{\prime}}(k_x, k_y)e^{-k_{\rho}(d_{\ell-1}-z)}.
\end{split}
\end{equation}
Applying the interface condition \eqref{generalinterfacecond} at $z=d_{\ell-1}$ gives
\begin{equation}\label{layertransmission}
\begin{split}
a_{\ell-1}\sigma^{1}_{\ell-1,\ell^{\prime}}+a_{\ell-1}e^{-k_{\rho}D_{\ell-1}}\sigma^{2}_{\ell-1,\ell^{\prime}}=a_{\ell}e^{-k_{\rho}D_{\ell}}\sigma^{1}_{\ell\ell^{\prime}}+a_{\ell}\sigma^{2}_{\ell\ell^{\prime}},\\
b_{\ell-1}\sigma^{1}_{\ell-1,\ell^{\prime}}-b_{\ell-1}e^{-k_{\rho}D_{\ell-1}}\sigma^{2}_{\ell-1,\ell^{\prime}}=b_{\ell}e^{-k_{\rho}D_{\ell}}\sigma^{1}_{\ell\ell^{\prime}}-b_{\ell}\sigma^{2}_{\ell\ell^{\prime}},
\end{split}
\end{equation}
or in matrix form
\begin{equation}
\widehat{\mathbb S}^{(\ell-1)}\begin{pmatrix}
\sigma^{1}_{\ell-1,\ell^{\prime}}\\
\sigma^{2}_{\ell-1,\ell^{\prime}}
\end{pmatrix}=\widetilde{\mathbb S}^{(\ell)}\begin{pmatrix}
\sigma^{1}_{\ell\ell^{\prime}}\\
\sigma^{2}_{\ell\ell^{\prime}}
\end{pmatrix},
\end{equation}
where
\begin{equation}
\widehat{\mathbb S}^{(\ell)}:=\begin{pmatrix}
a_{\ell} & a_{\ell}e_{\ell}\\
b_{\ell} & -b_{\ell}e_{\ell}
\end{pmatrix}, \quad \widetilde{\mathbb S}^{(\ell)}:=\begin{pmatrix}
a_{\ell}e_{\ell} & a_{\ell}\\
b_{\ell}e_{\ell} & -b_{\ell}
\end{pmatrix},\quad \ell=2,3,\cdots, L-1,
\end{equation}
and
\begin{equation}\label{expdef}
e_{\ell}:=e^{- k_{\rho}D_{\ell}},\quad d_{-1}:=d_0,\quad d_{L+1}:=d_{L},\quad D_{\ell}=d_{\ell-1}-d_{\ell},\quad \ell=0, 1, \cdots, L.
\end{equation}
Solving the above equations for $\{\sigma^{1}_{\ell-1,\ell'}, \sigma^{2}_{\ell-1,\ell'}\}$, we obtain
\begin{equation}\label{transmissionnosource}
\begin{pmatrix}
\sigma^{1}_{\ell-1,\ell^{\prime}}\\
\sigma^{2}_{\ell-1,\ell^{\prime}}
\end{pmatrix}=\mathbb T^{\ell-1,\ell}\begin{pmatrix}
\sigma^{1}_{\ell\ell^{\prime}}\\
\sigma^{2}_{\ell\ell^{\prime}}
\end{pmatrix}
\end{equation}
for $\ell= 2, 3,\cdots, L-1$, where
\begin{equation}\label{transmissionmat}
\begin{split}
\mathbb T^{\ell-1,\ell}=&\begin{pmatrix}
a_{\ell-1} & a_{\ell-1}e_{\ell-1}\\
b_{\ell-1} & -b_{\ell-1}e_{\ell-1}
\end{pmatrix}^{-1}\begin{pmatrix}
a_{\ell}e_{\ell} & a_{\ell}\\
b_{\ell}e_{\ell} & -b_{\ell}
\end{pmatrix}
=\frac{1}{2e_{\ell-1}}\begin{pmatrix}
e_{\ell-1} & 0\\
0 & 1
\end{pmatrix}
\widehat{\mathbb T}^{\ell-1,\ell}\begin{pmatrix}
e_{\ell} & 0\\
0 & 1
\end{pmatrix},
\end{split}
\end{equation}
and
\begin{equation}
\widehat{\mathbb T}^{\ell-1,\ell}:=\begin{pmatrix}
\displaystyle\frac{a_{\ell}}{a_{\ell-1}}+\frac{b_{\ell}}{b_{\ell-1}} & \displaystyle\frac{a_{\ell}}{a_{\ell-1}}-\frac{b_{\ell}}{b_{\ell-1}}\\
\displaystyle\frac{a_{\ell}}{a_{\ell-1}}-\frac{b_{\ell}}{b_{\ell-1}} & \displaystyle\frac{a_{\ell}}{a_{\ell-1}}+\frac{b_{\ell}}{b_{\ell-1}}
\end{pmatrix}.
\end{equation}
For the top and bottom most layers, we have $\sigma^{\downarrow}_{0\ell'}=0$ and $\sigma^{\uparrow}_{L\ell'}=0$, we can also verify that
\begin{equation}
\begin{pmatrix}
\sigma^{1}_{0\ell'}\\
0
\end{pmatrix}=\mathbb T^{01}\begin{pmatrix}
\sigma^{1}_{1\ell'}\\
\sigma^{2}_{1\ell'}
\end{pmatrix},\quad \begin{pmatrix}
\sigma^{1}_{L-1,\ell'}\\
\sigma^{2}_{L-1,\ell'}
\end{pmatrix}=\mathbb T^{L-1,L}
\begin{pmatrix}
0\\
\sigma^{2}_{L\ell'}
\end{pmatrix}.
\end{equation}

Next, we consider the solution in the layer with source $\bs r'$ inside, i.e., the solution in the $\ell'$-th layer. The general solution is given by
\begin{equation}
\widehat u_{\ell'\ell'}(k_x, k_y, z)=\sigma^{1}_{\ell^{\prime}\ell^{\prime}}e^{\ri k_{\ell' z}(z-d_{\ell'})}+\sigma^{2}_{\ell^{\prime}\ell^{\prime}}e^{\ri k_{\ell' z}(d_{\ell'-1}-z)}+\widehat{G}(k_x, k_y, z, z^{\prime}).
\end{equation}
At the interfaces $z=d_{\ell'-1}$ and $z=d_{\ell'}$, the interface conditions \eqref{generalinterfacecond} lead to equations
\begin{equation}\label{sourcelayerupper}
\begin{split}
&a_{\ell'-1}\big(\sigma^{1}_{\ell'-1,\ell^{\prime}}+e_{\ell'-1}\sigma^{2}_{\ell'-1,\ell^{\prime}}\big)=a_{\ell'}\big(e_{\ell'}\sigma^{1}_{\ell'\ell^{\prime}}+\sigma^{2}_{\ell'\ell^{\prime}}+\widehat{G}^{2}(k_x, k_y, d_{\ell'-1}, z')\big),\\
&b_{\ell'-1}\big(\sigma^{1}_{\ell'-1,\ell^{\prime}}-e_{\ell'-1}\sigma^{2}_{\ell'-1,\ell^{\prime}}\big)=b_{\ell'}\big(e_{\ell'}\sigma^{1}_{\ell'\ell^{\prime}}-\sigma^{2}_{\ell'\ell^{\prime}}\big)-\frac{ b_{\ell'}}{k_{\rho}}\partial_z\widehat{G}^{2}(k_x, k_y, d_{\ell'-1}, z'),
\end{split}
\end{equation}
and
\begin{equation}\label{sourcelayerbottom}
\begin{split}
&a_{\ell'}\big(\sigma^{1}_{\ell'\ell^{\prime}}+e_{\ell'}\sigma^{2}_{\ell'\ell^{\prime}}\big)=a_{\ell'+1}\big(e_{\ell'+1}\sigma^{1}_{\ell'+1\ell^{\prime}}+\sigma^{2}_{\ell'+1,\ell^{\prime}}\big)-a_{\ell'}\widehat{G}^{1}(k_x, k_y, d_{\ell'}, z'),\\
&b_{\ell'}\big(\sigma^{1}_{\ell'\ell^{\prime}}-e_{\ell'}\sigma^{2}_{\ell'\ell^{\prime}}\big)=b_{\ell'+1}\big(e_{\ell'+1}\sigma^{1}_{\ell'+1\ell^{\prime}}-\sigma^{2}_{\ell'+1,\ell^{\prime}}\big)+\frac{ b_{\ell'}}{k_{\rho}} \partial_z\widehat{G}^{1}(k_x, k_y, d_{\ell'}, z').
\end{split}
\end{equation}
Note that
\begin{equation*}
\partial_z\widehat{G}^{2}(k_x, k_y, d_{\ell'-1}, z')=- k_{\rho}\widehat{G}^{2}(k_x, k_y, d_{\ell'-1}, z'),\quad \partial_z\widehat{G}^{1}(k_x, k_y, d_{\ell'}, z')= k_{\rho}\widehat{G}^{1}(k_x, k_y, d_{\ell'}, z').
\end{equation*}
Then, equations \eqref{sourcelayerupper}-\eqref{sourcelayerbottom} can be reformulated as
\begin{equation}\label{downarrowsource}
\begin{pmatrix}
\sigma^{1}_{\ell'-1,\ell^{\prime}}\\
\sigma^{2}_{\ell'-1,\ell^{\prime}}
\end{pmatrix}=\mathbb T^{\ell'-1,\ell'}\begin{pmatrix}
\sigma^{1}_{\ell'\ell^{\prime}}\\
\sigma^{2}_{\ell'\ell^{\prime}}
\end{pmatrix}+\breve{\mathbb S}^{(\ell'-1)}\begin{pmatrix}
a_{\ell'}\\
b_{\ell'}
\end{pmatrix}\widehat{G}^{2}(k_x, k_y,d_{\ell'-1}, z')
\end{equation}
and
\begin{equation}\label{uparrowsource}
\begin{pmatrix}
\sigma^{1}_{\ell'\ell^{\prime}}\\
\sigma^{2}_{\ell'\ell^{\prime}}
\end{pmatrix}=\mathbb T^{\ell'\ell'+1}\begin{pmatrix}
\sigma^{1}_{\ell'+1,\ell^{\prime}}\\
\sigma^{2}_{\ell'+1,\ell^{\prime}}
\end{pmatrix}+\breve{\mathbb S}^{(\ell')}\begin{pmatrix}
-a_{\ell'}\\
b_{\ell'}
\end{pmatrix}\widehat{G}^{1}(k_x, k_y,d_{\ell'}, z'),
\end{equation}
where
\begin{equation}
\breve{\mathbb S}^{(\ell)}=\big(\widehat{\mathbb S}^{(\ell)}\big)^{-1}=\frac{1}{2}\begin{pmatrix}
1 & 0\\
0 & e_{\ell}^{-1}
\end{pmatrix}\begin{pmatrix}
\displaystyle\frac{1}{a_{\ell}} & \displaystyle\frac{1}{b_{\ell}}\\
\displaystyle\frac{1}{a_{\ell}} & \displaystyle-\frac{1}{b_{\ell}}
\end{pmatrix}:=\begin{pmatrix}
\breve S_{11}^{(\ell)} & \breve S_{12}^{(\ell)}\\
\breve S_{21}^{(\ell)} & \breve S_{22}^{(\ell)}
\end{pmatrix}.
\end{equation}
Define
\begin{equation}\label{regularmat}
\widetilde{\mathbb T}^{\ell-1,\ell}=2e_{\ell-1}\mathbb T^{\ell-1,\ell},\quad C^{(\ell)}=\prod\limits_{j=0}^{\ell-1}\frac{1}{2e_j},\quad \mathbb A^{(\ell)}=\widetilde{\mathbb T}^{01}\widetilde{\mathbb T}^{12}\cdots \widetilde{\mathbb T}^{\ell-1,\ell}:=\begin{pmatrix}
\alpha_{11}^{(\ell)} & \alpha_{12}^{(\ell)}\\
\alpha_{21}^{(\ell)} & \alpha_{22}^{(\ell)}
\end{pmatrix},
\end{equation}
for $\ell=1, 2,\cdots, L$. Then, recursions in \eqref{transmissionnosource}, \eqref{downarrowsource} and \eqref{uparrowsource} result in the system
\begin{equation}\label{Lsystem}
\begin{split}
\begin{pmatrix}
\sigma^{1}_{0\ell'}\\
0
\end{pmatrix}=&C^{(L)}\mathbb A^{(L)}\begin{pmatrix}
0\\
\sigma^{2}_{L\ell'}
\end{pmatrix}+C^{(\ell'-1)}\mathbb A^{(\ell'-1)}\breve{\mathbb S}^{(\ell'-1)}\begin{pmatrix}
a_{\ell'}\\
b_{\ell'}
\end{pmatrix}\widehat{G}^{2}(k_x, k_y, d_{\ell'-1}, z')\\
&+C^{(\ell')}\mathbb A^{(\ell')}\breve{\mathbb S}^{(\ell')}\begin{pmatrix}
-a_{\ell'}\\
b_{\ell'}
\end{pmatrix}\widehat{G}^{1}(k_x, k_y, d_{\ell'}, z').
\end{split}
\end{equation}

It is not numerically stable to directly solve \eqref{Lsystem} for $\sigma_{0\ell'}^{1}$ and $\sigma_{L\ell'}^{2}$ then apply recursions \eqref{transmissionnosource}, \eqref{downarrowsource} and \eqref{uparrowsource} to obtain all other reaction densities due to the exponential functions involved in the formulations. According to the expression \eqref{transmissionmat}, the recursions \eqref{transmissionnosource}, \eqref{downarrowsource} and \eqref{uparrowsource} are stable for the computation of the components $\sigma_{\ell\ell'}^{1}(k_{\rho})$. As for the computation of the components $\sigma_{\ell\ell'}^{2}(k_{\rho})$, we need to form linear systems similar as \eqref{Lsystem} using recursions \eqref{transmissionnosource}, \eqref{downarrowsource} and \eqref{uparrowsource} and then solve it.

We first solve the second equation in \eqref{Lsystem} to get
$$\sigma^{2}_{L\ell'}=\sigma_{L\ell'}^{21}\widehat{G}^{1}(k_x, k_y,  d_{\ell'}, z')+\sigma_{L\ell'}^{22}\widehat{G}^{2}(k_{\ell'z}, d_{\ell'-1}, z'),$$
where
\begin{equation}\label{bottommostdensity}
\begin{split}
&\sigma_{L\ell'}^{21}=-\frac{C^{(\ell'+1)}}{C^{(L)}\alpha_{22}^{(L)}}\begin{pmatrix}
\alpha^{(\ell')}_{21} & \alpha^{(\ell')}_{22}
\end{pmatrix}2e_{\ell'}\breve{\mathbb S}^{(\ell')}\begin{pmatrix}
-a_{\ell'}\\
b_{\ell'}
\end{pmatrix},\quad 0\leq\ell'<L,\\
&\sigma_{L\ell'}^{22}=-\frac{C^{(\ell')}}{C^{(L)}\alpha_{22}^{(L)}}\begin{pmatrix}
\alpha^{(\ell'-1)}_{21} & \alpha^{(\ell'-1)}_{22}
\end{pmatrix}2e_{\ell'-1}\breve{\mathbb S}^{(\ell'-1)}\begin{pmatrix}
a_{\ell'}\\
b_{\ell'}
\end{pmatrix},\quad 0<\ell'\leq L.
\end{split}
\end{equation}
According to the recursion \eqref{transmissionnosource},\eqref{downarrowsource} and \eqref{uparrowsource}, all other reaction densities also have decompositions
\begin{equation}\label{desitydecomposition}
\begin{split}
\sigma_{\ell\ell'}^{1}=\sigma_{\ell\ell'}^{11}\widehat{G}^{1}(k_x, k_y, d_{\ell'}, z')+\sigma_{\ell\ell'}^{12}\widehat{G}^{2}(k_x, k_y, d_{\ell'-1}, z'),\\ \sigma_{\ell\ell'}^{2}=\sigma_{\ell\ell'}^{21}\widehat{G}^{1}(k_x, k_y, d_{\ell'}, z')+\sigma_{\ell\ell'}^{22}\widehat{G}^{2}(k_x, k_y, d_{\ell'-1}, z').
\end{split}
\end{equation}
For each $0\leq\ell<L$, we first calculate $\{\sigma_{\ell\ell'}^{11}, \sigma_{\ell\ell'}^{12}\}$ by using one of the recursions \eqref{transmissionnosource}, \eqref{downarrowsource} and \eqref{uparrowsource}, then formulate a linear system for $\{\sigma_{0\ell'}^{1},\sigma_{\ell\ell'}^{2}\}$ as the linear system \eqref{Lsystem}. Next, we solve the second equation in the linear system to obtain reaction densities $\{\sigma_{\ell\ell'}^{21},\sigma_{\ell\ell'}^{22}\}$. In summary, the formulations are given as follows:
\begin{align}
&\sigma_{\ell\ell'}^{11}=\begin{cases}
\displaystyle T^{\ell'\ell'+1}_{11}\sigma_{\ell'+1,\ell'}^{11}+T^{\ell'\ell'+1}_{12}\sigma_{\ell'+1,\ell'}^{21}-\breve S_{11}^{(\ell')}a_{\ell'}+\breve S_{12}^{(\ell')}b_{\ell'},& \ell=\ell',\\
\displaystyle T^{\ell\ell+1}_{11}\sigma_{\ell+1,\ell'}^{11}+T^{\ell\ell+1}_{12}\sigma_{\ell+1,\ell'}^{21}, & {\rm else},
\end{cases}\\
&\sigma_{\ell\ell'}^{12}=\begin{cases}
\displaystyle T^{\ell'-1,\ell'}_{11}\sigma_{\ell'\ell'}^{12}+T^{\ell'-1,\ell'}_{12}\sigma_{\ell'\ell'}^{22}+\breve S_{11}^{(\ell'-1)}a_{\ell'}+\breve S_{12}^{(\ell'-1)}b_{\ell'},& \ell=\ell'-1,\\
\displaystyle T^{\ell\ell+1}_{11}\sigma_{\ell+1,\ell'}^{12}+T^{\ell\ell+1}_{12}\sigma_{\ell+1,\ell'}^{22}, & {\rm else},
\end{cases}\\
&\sigma_{\ell\ell'}^{21}=\begin{cases}
\displaystyle-\frac{1}{\alpha_{22}^{(\ell)}}\begin{pmatrix}
0 & 1
\end{pmatrix}\left[\frac{C^{(\ell'+1)}}{C^{(\ell)}}\mathbb A^{(\ell')}2e_{\ell'}\breve{\mathbb S}^{(\ell')}\begin{pmatrix}
-a_{\ell'}\\
b_{\ell'}
\end{pmatrix}+\mathbb A^{(\ell)}\begin{pmatrix}
\sigma_{\ell\ell'}^{11}\\
0
\end{pmatrix}\right],& \ell>\ell',\\[10pt]
\displaystyle-\frac{\alpha_{21}^{(\ell)}}{\alpha_{22}^{(\ell)}}\sigma_{\ell\ell'}^{11}, & {\rm else},
\end{cases}\\
&\sigma_{\ell\ell'}^{22}=\begin{cases}
\displaystyle-\frac{1}{\alpha_{22}^{(\ell)}}\begin{pmatrix}
0 & 1
\end{pmatrix}\left[\frac{C^{(\ell')}}{C^{(\ell)}}\mathbb A^{(\ell'-1)}2e_{\ell'-1}\breve{\mathbb S}^{(\ell'-1)}\begin{pmatrix}
a_{\ell'}\\
b_{\ell'}
\end{pmatrix}+\mathbb A^{(\ell)}\begin{pmatrix}
\sigma_{\ell\ell'}^{12}\\
0
\end{pmatrix}\right],& \ell\geq \ell',\\[10pt]
\displaystyle-\frac{\alpha_{21}^{(\ell)}}{\alpha_{22}^{(\ell)}}\sigma_{\ell\ell'}^{12}, & {\rm else}.
\end{cases}\label{downdowndensity}
\end{align}
Substituting \eqref{desitydecomposition} and \eqref{freegreenfreq} into \eqref{gensolutionformulafreq} and taking inverse Fourier transform, we obtain expressions \eqref{layeredGreensfun}-\eqref{zexponential}.

From the definition \eqref{transmissionmat} and \eqref{regularmat}, we have
\begin{equation*}
\begin{split}
&T^{\ell\ell+1}_{11}=\frac{a_{\ell+1}b_{\ell}+a_{\ell}b_{\ell+1}}{2a_{\ell}b_{\ell}}e_{\ell+1},\quad  T^{\ell\ell+1}_{12}= \frac{a_{\ell+1}b_{\ell}-a_{\ell}b_{\ell+1}}{2a_{\ell}b_{\ell}},\\
&2e_{\ell}\breve{\mathbb S}^{(\ell)}=\begin{pmatrix}
\displaystyle a_{\ell}^{-1}e_{\ell} & \displaystyle b_{\ell}^{-1}e_{\ell}\\
\displaystyle a_{\ell}^{-1} & \displaystyle-b_{\ell}^{-1}
\end{pmatrix},\quad
\frac{C^{(\ell_1)}}{C^{(\ell_2)}}=\begin{cases}
1 & \ell_1=\ell_2,\\
2^{\ell_2-\ell_1}e^{-k_{\rho}(d_{\ell_1-1}-d_{\ell_2-1})} & 0\leq\ell_1<\ell_2,
\end{cases}
\end{split}
\end{equation*}
and an asymptotic behavior
\begin{equation}
\mathbb A^{(\ell)}\sim\begin{pmatrix}
\tilde\alpha_{11}^{(\ell)}e_0e_1\cdots e_{\ell} & \tilde\alpha_{12}^{(\ell)}e_0\\
\tilde\alpha_{21}^{(\ell)}e_{\ell} & \tilde\alpha_{22}^{(\ell)}
\end{pmatrix}, \quad k_{\rho}\rightarrow\infty,
\end{equation}
where $\{\tilde\alpha_{11}^{(\ell)}, \tilde\alpha_{12}^{(\ell)}, \tilde\alpha_{21}^{(\ell)}, \tilde\alpha_{22}^{(\ell)}\}$ are constants independent of $k_{\rho}$. By using these formulations in \eqref{bottommostdensity}-\eqref{downdowndensity}, we can show that all reaction densities $\{\sigma_{\ell\ell'}^{\mathfrak{ab}}(k_{\rho})\}_{\mathfrak{a,b}=1}^2$ have an asymptotic behavior
\begin{equation}\label{densityasymptotic}
\sigma_{\ell\ell'}^{\mathfrak{ab}}(k_{\rho})\sim C_{\ell\ell'}^{\mathfrak{ab}}e^{-k_{\rho}\zeta_{\ell\ell'}^{\mathfrak{ab}}},\quad k_{\rho}\rightarrow\infty,
\end{equation}
where $C_{\ell\ell'}^{\mathfrak{ab}}$ and $\zeta_{\ell\ell'}^{\mathfrak{ab}}$ are constants independent of $k_{\rho}$. For example, we have
\begin{equation}
\begin{split}
&\sigma_{L\ell'}^{21}(k_{\rho})\sim 2^{L-\ell'-1}\frac{\tilde\alpha_{22}^{(\ell')}}{\alpha_{22}^{(L)}}e^{-k_{\rho}(d_{\ell'}-d_{L-1})},\quad k_{\rho}\rightarrow\infty,\\
&\sigma_{L\ell'}^{22}(k_{\rho})\sim 2^{L-\ell'}\frac{\alpha_{22}^{(\ell')}}{\alpha_{22}^{(L)}}\Big(\frac{a_{\ell'}}{a_{\ell'-1}}+\frac{b_{\ell'}}{b_{\ell'-1}}\Big)e^{-k_{\rho}(d_{\ell'-1}-d_{L-1})},\quad k_{\rho}\rightarrow\infty.
\end{split}
\end{equation}

If the number of layers is not large, we are able to write down explicit expressions of the reaction densities. Here, we give expressions for the case of a three layers media with $a_{\ell}=1$, $b_{\ell}=\varepsilon_{\ell}$ as an example.
\begin{itemize}
	\item Source in the top layer:
	\begin{equation}\label{densitythreelayer1}
	\begin{split}
	\sigma_{00}^{11}(k_{\rho})=&\frac{(\varepsilon_0-\varepsilon_1)(\varepsilon_1+\varepsilon_2)+(\varepsilon_0+\varepsilon_1)(\varepsilon_1-\varepsilon_2)e^{2d_1k_{\rho}}}{2\kappa(k_{\rho})},\\
	\sigma_{10}^{21}(k_{\rho})=&\frac{\varepsilon_0(\varepsilon_1+\varepsilon_2)}{\kappa(k_{\rho})},\quad
	\sigma_{10}^{11}(k_{\rho})=\frac{\varepsilon_0(\varepsilon_1-\varepsilon_2)e^{d_1k_{\rho}}}{\kappa(k_{\rho})},\quad
	\sigma_{20}^{21}(k_{\rho})=\frac{2 \varepsilon_0\varepsilon_1 e^{d_1k_{\rho}}}{\kappa(k_{\rho})}.
	\end{split}
	\end{equation}
	\item Source in the middle layer:
	\begin{equation}\label{densitythreelayer2}
	\begin{split}
	\sigma_{01}^{12}(k_{\rho})=&\frac{\varepsilon_1(\varepsilon_1+\varepsilon_2) }{\kappa(k_{\rho})},\quad
	\sigma_{01}^{11}(k_{\rho})=\frac{\varepsilon_1(\varepsilon_1-\varepsilon_2)e^{d_1k_{\rho}} }{\kappa(k_{\rho})},\\
	\sigma_{11}^{11}(k_{\rho})=&\frac{(\varepsilon_1-\varepsilon_2)(\varepsilon_1+\varepsilon_0)}{2\kappa(k_{\rho})},\quad
	\sigma_{11}^{21}(k_{\rho})=\frac{(\varepsilon_1-\varepsilon_2)(\varepsilon_1-\varepsilon_0)e^{d_1k_{\rho}}}{2\kappa(k_{\rho})},\\
	\sigma_{11}^{12}(k_{\rho})=&\frac{(\varepsilon_1-\varepsilon_2)(\varepsilon_1-\varepsilon_0)e^{d_1k_{\rho}}}{2\kappa(k_{\rho})},\quad\sigma_{11}^{22}(k_{\rho})=\frac{(\varepsilon_1+\varepsilon_2)(\varepsilon_1-\varepsilon_0)}{2\kappa(k_{\rho})},\\
	\sigma_{21}^{22}(k_{\rho})=&\frac{\varepsilon_1(\varepsilon_1-\varepsilon_0) e^{d_1k_{\rho}}}{\kappa(k_{\rho})},\quad
	\sigma_{21}^{21}(k_{\rho})=\frac{\varepsilon_1(\varepsilon_0+\varepsilon_1) }{\kappa(k_{\rho})}.
	\end{split}
	\end{equation}
	\item Source in the bottom layer:
	\begin{equation}\label{densitythreelayer3}
	\begin{split}
	\sigma_{02}^{12}(k_{\rho})=&\frac{2 \varepsilon_1\varepsilon_2 e^{d_1k_{\rho}}}{\kappa(k_{\rho})},\quad
	\sigma_{12}^{22}(k_{\rho})=\frac{\varepsilon_2(\varepsilon_1-\varepsilon_0)e^{d_1k_{\rho}}}{\kappa(k_{\rho})},\quad
	\sigma_{12}^{12}(k_{\rho})=\frac{\varepsilon_2(\varepsilon_0+\varepsilon_1)}{\kappa(k_{\rho})},\\
	\sigma_{22}^{22}(k_{\rho})=&\frac{(\varepsilon_1-\varepsilon_0)(\varepsilon_1+\varepsilon_2)+(\varepsilon_0+\varepsilon_1)(\varepsilon_2-\varepsilon_1)e^{2d_1k_{\rho}} }{2\kappa(k_{\rho})},
	\end{split}
	\end{equation}
\end{itemize}
where
$$\kappa(k_{\rho})=\frac{1}{2}\big[(\varepsilon_0+\varepsilon_1)(\varepsilon_1+\varepsilon_2)+(\varepsilon_0-\varepsilon_1)(\varepsilon_2-\varepsilon_1)e^{2d_1k_{\rho}}\big].$$
Apparently, these expressions also verify our conclusion \eqref{densityasymptotic} on the asymptotic behavior of the reaction densities.
\end{appendix}

}

\section*{Acknowledgement}

This work was supported by US Army Research Office (Grant No.
W911NF-17-1-0368) and US National Science Foundation (Grant No. DMS-1764187). The research of the first author is partially supported by NSFC (grant 11771137), the Construct Program of the Key Discipline in Hunan Province and a Scientific Research Fund of Hunan Provincial Education Department (No. 16B154).

\section*{References}
%\bibliography{../../../../../refsphere}

\end{document}